\documentclass[12pt]{amsart}
\usepackage{enumerate,amssymb,amsfonts,latexsym,psfrag}
\usepackage{amscd,amsmath}
\usepackage[dvips]{graphicx,epsfig}
\usepackage{graphics}
\usepackage[all]{xy}
\usepackage{fullpage}
\usepackage{subfig}
\usepackage{multirow,array}
\usepackage{color}
\usepackage{amsthm}
\usepackage{mathtools}
\usepackage{url}

\usepackage{titlesec}

\titleformat{\section}
  {\Large\bfseries\center}{\thesection.}{0.5em}{}
  \titlespacing*{\section} {0pt}{1em}{2.3ex plus .2ex}
\titlelabel{\thetitle.\quad}
\titleformat{\subsection}
  [runin]{\normalfont\bfseries}{\thesubsection.}{0.5em}{}[.]
\titlespacing*{\subsection}{0pt}{1em}{2.3ex plus .2ex}

\usepackage{titletoc}
\titlecontents{section}[1.5em]
  {\normalfont\bfseries}
  {\contentslabel{1.5em}}
  {\hspace*{-2.3em}}
 {\titlerule*[1pc]{}\contentspage}

\newtheorem{thm}{Theorem}[section]
\newtheorem{cor}[thm]{Corollary}
\newtheorem{lem}[thm]{Lemma}
\newtheorem{prop}[thm]{Proposition}
\theoremstyle{remark}

\theoremstyle{definition}
\newtheorem{defn}{Definition}[section]
\newtheorem{rem}{Remark}[section]

\numberwithin{equation}{section}
\numberwithin{figure}{section}

\font\nt=cmr7

\def\note#1
{\marginpar
{\nt $\leftarrow$
\par
\hfuzz=20pt \hbadness=9000 \hyphenpenalty=-100 \exhyphenpenalty=-100
\pretolerance=-1 \tolerance=9999 \doublehyphendemerits=-100000
\finalhyphendemerits=-100000 \baselineskip=6pt
#1}\hfuzz=1pt}

\renewcommand{\Im}{\operatorname{Im}}

\newcommand{\di}{\partial}

\newcommand{\ra}{\rightarrow}

\def\lra{\longrightarrow}

\def\ssk{\smallskip}
\def\msk{\medskip}
\def\bsk{\bigskip}

\def\nin{\noindent}

\newcommand{\diam}{\operatorname{diam}}
\newcommand{\dist}{\operatorname{dist}}

\renewcommand{\mod}{\operatorname{mod}}

\newcommand{\id}{\operatorname{id}}

\newcommand{\Per}{\operatorname{Per}}

\newcommand{\Id}{\operatorname{Id}}

\newcommand{\Orb}{\operatorname{Orb}}

\newcommand{\tip}{\operatorname{tip}}

\newcommand{\const}{\mathrm{const}}
\def\loc{{\mathrm{loc}}}

\newcommand{\eps}{{\varepsilon}}

\newcommand{\de}{{\delta}}
\newcommand{\la}{{\lambda}}
\newcommand{\La}{{\Lambda}}
\newcommand{\si}{{\sigma}}

\newcommand{\bde}{{\boldsymbol{\delta}}}
\newcommand{\bxi}{{\boldsymbol{\xi}}}

\newcommand{\boldeta}{{\boldsymbol{\eta}}}

\newcommand{\AAA}{{\mathcal A}}

\newcommand{\CC}{{\mathcal C}}

\newcommand{\II}{{\mathcal I}}

\newcommand{\OO}{{\mathcal O}}

\newcommand{\TT}{{\mathcal T}}

\newcommand{\UU}{{\mathcal U}}

\newcommand{\N}{{\mathbb N}}

\newcommand{\R}{{\mathbb R}}

\newcommand{\Z}{{\mathbb Z}}

\def\Bv{{\mathbf{v}}}
\def\Bz{{\mathbf{z}}}

\def\Bd{{\mathbf{d}}}
\def\Bu{{\mathbf{u}}}

\def\Br{{\mathbf{r}}}

\def\BPhi{{\boldsymbol{\BPhi}}}
\def\B0{{\mathbf{0}}}

\newcommand{\Jac}{\operatorname{Jac}}

\newcommand{\Dom}{\operatorname{Dom}}

\catcode`\@=12

\def\Empty{}
\newcommand\oplabel[1]{
  \def\OpArg{#1} \ifx \OpArg\Empty {} \else
  	\label{#1}
  \fi}
		
\newcommand{\comm}[1]{}
\newcommand{\comment}[1]{}

\allowdisplaybreaks[4]

\begin{document}

\bigskip\bigskip

\title[H\'enon renormalization]{Renormalization of H\'enon map in arbitrary dimension I $ \colon $ Universality and reduction of ambient space }
\author{Young Woo Nam}

\address {Young Woo Nam \\ \quad e-mail : namyoungwoo\,@hongik.ac.kr, ellipse7\,@daum.net}

\date{February 4, 2015}

\begin{abstract}
Period doubling H\'enon renormalization of strongly dissipative maps is generalized in arbitrary finite dimension. In particular, a small perturbation of toy model maps with dominated splitting has invariant $ C^r $ surfaces embedded in higher dimension and the Cantor attractor has unbounded geometry with respect to full Lebesgue measure on the parameter space. It is an extension of dynamical properties of three dimensional infinitely renormalizable H\'enon-like map in arbitrary finite dimension. 
\end{abstract}

\maketitle


\thispagestyle{empty}


\setcounter{tocdepth}{1}

\tableofcontents


\renewcommand{\labelenumi} {\rm {(}\arabic{enumi}{)}}

\section{Introduction}
Universality of one dimensional dynamical system was discovered by Feigenbaum and independently by Coullet and Tresser in the mid 1970's. Moreover, the universality of the higher dimensional maps is conjectured by Coullet and Tresser in \cite{CT}. 
The similar universality properties are expected in higher dimensional maps which are strongly dissipative and close to the one dimensional maps. In particular, renormalizable maps with {\em period doubling type} are interesting in higher dimension. Universality of two dimensional strongly dissipative infinitely renormalizable H\'enon-like maps was introduced in \cite{CLM}. The Cantor attractor of H\'enon-like map is the counterpart of that of one dimensional maps but it has typically unbounded geometry. The same geometric properties are common in certain classes of the sectionally dissipative three dimensional H\'enon-like family in \cite{Nam1,Nam2,Nam3}. There exists the universal expression of Jacobian determinant of infinitely renormalizable three dimensional H\'enon-like maps but it does not imply the universal expression. However, {\em unbounded geometry} of Cantor attractor were generalized in the special invariant subspace of infinitely renormalizable maps in \cite{Nam1,Nam3}. 
\ssk \\
This paper is about the generalization of H\'enon renormalization in arbitrary finite dimension. In the viewpoint of a perturbation of one dimensional map in higher dimension, Higher dimensional renormalizable H\'enon-like map is a counterpart of the perturbation of Misiurewicz maps in higher dimension which appear in \cite{WY}.\footnote{The claim for counterpart might require renormalizable higher dimensional H\'enon-like map has {\em rank one} attractor. Renormalizable two dimensonal H\'enon-like maps have one dimensional global attracting set in \cite{LM} and by slight modifying proof can show that it is true for higher dimensional H\'enon-like maps. But we would not deal with this fact in this paper. } 
A certain invariant class of infinitely renormalizable three dimensional maps is generalized and every results in \cite{Nam1,Nam3} are extended in arbitrary dimension. For instance, the following results are extended in arbitrary dimension. \ssk
\begin{itemize}
\item[---] Universality of Jacobian determinant of renormalized maps. \ssk
\item[---] Existence of single invariant surfaces under certain conditions. \ssk
\item[---] Existence of $ C^r $ renormalizable two dimensonal H\'enon-like map with invariant $ C^r $ surface. \ssk
\item[---] Unbounded geometry of Cantor attractor. \ssk
\end{itemize}
\nin H\'enon-like map $F$ from the hypercube $ B $ to $ \R^{m+2} $ is defined as follows
$$ F \colon (x,y,\Bz) \mapsto (f(x) - \eps(x,y,\Bz),\ x,\ \bde(x,y,\Bz)) $$
where $ f(x) $ is a unimodal map, $ \Bz = (z_1,z_2,\ldots ,z_m) $ and $ \bde = (\de^1,\de^2,\ldots,\de^m) $ is a map from $ B \Subset \R^{m+2} $ to $ \R^m $. Let us assume that $ F $ has two hyperbolic fixed points, $ \beta_0 $ which has positive eigenvalues and $ \beta_1 $ which has both positive and negative eigenvalues. In this paper, we assume that both $ \| \eps \| $ and $ \| \bde \| $ are bounded above by $ O(\bar \eps) $ where $ \bar \eps $ is a small enough positive number. Since $ \| \:\! \bde \| $ is sufficiently small, each fixed point has only one expanding direction. We assume that the product of any two different eigenvalues is strictly less than one at fixed points of $ F $, namely, {\em sectionally dissipative} at fixed points. H\'enon-like map is called {\em renormalizable} if $ W^u(\beta_0) $ intersects $ W^s(\beta_1) $ at the orbit of a single point. Thus $ F $ has one dimensional unstable manifold and codimension one stable manifold at fixed points. 

\begin{thm}[Universality of $ \Jac R^nF $]
Let $ F $ be the $ m+2 $ dimensional infinitely renormalizable H\'enon-like map. Then 
$$ \Jac R^nF = b^{2^n}a(x)(1+ O(\rho^n))$$
where $ b = b_F $ is the average Jacobian of $ F $, $ a(x) $ is the universal function for $ \rho \in (0,1) $.
\end{thm}
\ssk 

\nin Let the H\'enon-like map with the condition $ \di_{z_j} \eps \equiv 0 $ for all $ 1 \leq j \leq m $ be the {\em toy model map}\,, say $ F_{\mod} $ as follows
$$ F_{\mod}(x,y,\Bz) = (f(x) - \eps(x,y),\ x,\ \bde(x,y,\Bz)).  $$
If $ F_{\mod} $ is infinitely renormalizable, then $ n^{th} $ renormalization of $ F_{\mod} $ contains two dimensional renormalized map with universality
$$ \pi_{xy} \circ R^nF_{\mod}(x,y,\Bz) = R^nF_{2d}(x,y) = (f_n(x) + b_1^{2^n}a(x)\,y\, (1+O(\rho^n)),\ x) $$ 
where $ b_1 $ is the average Jacobian of the two dimensional H\'enon-like map for some $ 0 < \rho <1 $. If $ \| \bde \| \ll b_1 $, then there exists the continuous invariant plane field over the critical Cantor set under the $ DF_{\mod} $.  Moreover, $ F $ is a {\em small perturbation of the model map}, that is, $ \eps(x,y,\Bz) = \eps_{2d}(x,y) + \widetilde \eps(x,y,\Bz) $ and $ \max \{\,\| \:\!\di_{z_j} \eps \| \;|\; 1\leq j \leq m \,\} $, then it also has continuous invariant plane field. Furthermore, there exists a single surface $ Q $ invariant under $ F $. Additionally if $ F $ is infinitely renormalizable, then there exists an invariant surface $ Q_n $ under $ R^nF $ for each $ n \in \N $ as the graph of $ C^r $ map $ \bxi = (\xi^1, \xi^2, \ldots, \xi^m) $ from $ xy- $plane to $ z_j- $axis for all $ 1 \leq j \leq m $ for $ 2 \leq r < \infty $(Lemma \ref{invariant surfaces on each deep level}). Then two dimensional $ C^r $ H\'enon-like map is defined as follows
\begin{equation} \label{eq-cr Henon map with inv surface}
F_{2d,\:\bxi} (x,y) = (f(x) - \eps(x,y, \bxi),\ x )
\end{equation}
where $ Q \equiv \textrm{graph} (\bxi) $ is a $ C^r $ invariant surface under $ F $. Then 
universality theorem of infinitely renormalizable $ C^r $ H\'enon-like maps are obtained. 
\begin{thm}[Universality of\; $ C^r $ H\'enon-like maps with invariant single surfaces for $ 2 \leq r < \infty $] 
Let H\'enon-like map $ F_{2d,\, \bxi} $ be the $ C^r $ map for some $ 2 \leq r < \infty $ which is defined in \eqref{eq-cr Henon map with inv surface}. Suppose that $ F_{2d,\, \bxi} $ is infinitely renormalizable. Then
$$ R^nF_{2d,\, \bxi} = (f_n(x) - b_{1,\, 2d}^{2^n}\, a(x)\, y (1+ O(\rho^n)),\ x) $$
where $ b_{1,\,2d} $ is the average Jacobian of $ F_{2d,\, \bxi} $ and $ a(x) $ is the universal function for some $ 0 < \rho < 1 $.
\end{thm}
\nin 
Cantor attractor of $ C^r $ H\'enon-like map also has the geometric properties which are the same as that of Cantor attractor for analytic maps. In particular, the critical Cantor set has unbounded geometry (Theorem \ref{Unbounded geometry for model maps}). Moreover, all of these dynamical properties of Cantor attractor are generalized in higher dimensional H\'enon-like map $ F $ through its invariant surfaces.
\msk

\subsection{Notations}
For the given map $ F $, we denote  the set $ A $ is related to $F$ to be $ A(F) $ or $ A_F $. $ F $ can be omitted if there is no confusion without $ F $. The domain of the function $ F $ is denoted to $ \Dom(F) $ and the image of the set $ B $ under a function $ F $ is denoted to be $ \Im(F) $. $ F|_A $ is called the restriction $ F $ on $ A $ where $ A \subset \Dom(F) $. If $ F(B) \subset B $, then we call $ B $ is an (forward) invariant set under $ F $. 
\ssk \\
Let the projection from $ \R^{m+2} $ to its $ x- $axis, $ y- $axis and $ z- $axis be $ \pi_x $, $ \pi_y $ and $ \pi_z $ respectively. Moreover, let the projection from $ \R^3 $ to $ xy- $plane be $ \pi_{xy} $ and so on. 
The derivative of the map $ f $ is expressed as $ Df $. 
The chain rule implies that $ D(f\circ g)(w) = Df \circ g(w) \cdot Dg(w) $. In this paper, the boldfaced letter means the condensed expression with $ m $ coordinates. For example, 
\begin{align*}
\Bz = (z_1,z_2, \ldots,z_m),\ \ \bde = (\de^1,\de^2,\ldots,\de^m) \ \text{and} \ \bxi =(\xi^1,\xi^2,\ldots,\xi^m) .
\end{align*}
\nin The dot product of two objects presented boldfaced letters, say $ {\bf A} $ and $ {\bf B} $ means the inner product of them. Denote it by $ {\bf A} \cdot {\bf B} $. Let the set distance $ \dist_{\min}(R,S) $ be the minimal distance between two sets, $ R $ and $ S $ 
$$ \dist_{\min} (R, S) = \inf \left\{ \ \dist(r,s) \ \text{for all} \ r \in R \ \text{and} \ s \in S \ \right\} . $$

\nin Denote the set of periodic points of $F$ to be $ \Per_F $. The orbit of the point $ w $ under the map $ f $ is denoted to be $ \Orb(w, f) $. 
Denote the (complete) orbit of $ w $ to be $ \Orb(w) $ unless the map is emphasized or is ambiguous on the context in the related description. 
$ A = O(B) $ means that there exists a positive number $ C $ such that $ A \leq CB $. Moreover, $ A \asymp B $ means that there exists a positive number $ C $ which satisfies $ \dfrac{1}{C} B \leq A \leq CB$.

\bsk

\section{H\'enon renormalization in higher dimension} \label{maps in general dimension}

\subsection{H\'enon-like maps in higher dimension}
Let $ B_{2d} $ be the domain of two dimensional H\'enon-like map and it is the square region with the center origin. Let $B$ be the box domain which is a thickened domain of two dimensional H\'enon-like map, that is, $B = B_{2d} \times [-c,c]^m$ for some $c>0$ and a fixed positive number $ m $. 
Let us define the {\em $ m+2 $ dimensional H\'enon-like map} on the hypercube $B$ 
as follows 
\begin{equation}  \label{higher diml Henon map}
\begin{aligned}
F(x,y,\Bz) = (f(x) - \eps(x,y,\Bz), \ x, \ \bde(x,y,\Bz))
\end{aligned}
\end{equation}
where $f : I^x \rightarrow I^x $ is a unimodal map, $ \Bz $ is $ (z_1,z_2,\ldots ,z_m) $ and $ \bde = (\de^1,\,\de^2,\ldots, \de^m) $ is the map from $ B $ to $ \R^m $ $  $. For simplicity, let us assume that the length of each side of $B$ is same. Denote the domain, $B =I^x \times  {\bf I}^v $ where $I^x$ is the line parallel to $x$-axis and ${\bf I}^v = I^y \times I^{\Bz}$ where $I^y$ is the line parallel to $y$-axis and $I^{\Bz}$ is the hypercube $ [-c, c]^m $. 
\ssk 
\begin{rem}
On the following sections, some objects defined on the two dimensional space has the subscript $2d$. For example, $B_{2d}$ is the square domain of the two dimensional H\'enon-like map and $F_{2d}$ is the two dimensional H\'enon-like map defined on $B_{2d}$. However, same notation without any index indicates the $ m+2 $ dimensional object. For instance, $F$ and $B$ are the higher dimensional H\'enon-like map and its box domain respectively.
\end{rem}
\nin Observe that the image of the codimension one hyperplane, $ \lbrace x = C \rbrace $ under $F$ is contained in the codimension one hyperplane, $ \lbrace y = C \rbrace  $. 

\begin{figure}[htbp]
\begin{center}
\psfrag{M1}[c][c][0.7][0]{\Large $F $}

\includegraphics[scale=1.0]{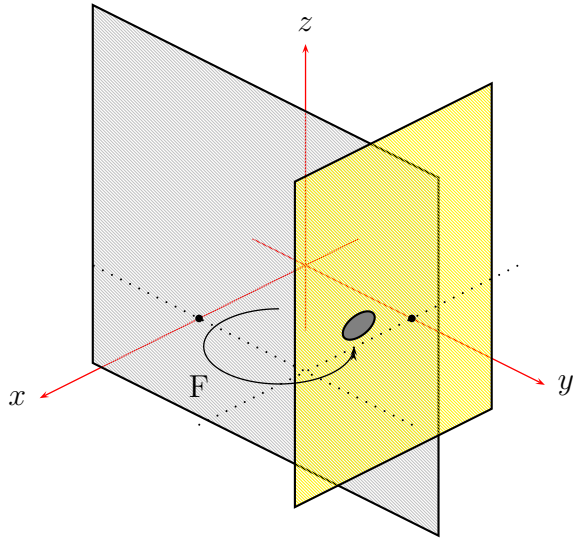}
\end{center}
\caption{Image of $ \lbrace x=const. \rbrace$ under higher dimensional H\'enon-like map}
\label{fig:imageofplane}
\end{figure}

\nin In this paper, higher dimensional analytic H\'enon-like maps in \eqref{higher diml Henon map} have the following properties.
\begin{itemize}
\item[---] $ F $ is orientation preserving map. \msk
\item[---] $ f $ has non flat unique critical point in $ I^x $. \msk
\item[---] Both $ \| \eps \| \leq \bar \eps $ and $\| \bde \| \leq \bar \de$ with sufficiently small positive numbers $ \bar \eps $ and $ \bar \de $. \msk
\item[---] $ F $ has only two fixed points one of which, say $ \beta_0 $, has only positive eigenvalues. \msk
\item[---] $ F $ is {\em sectionally dissipative} at fixed points, that is, any product of two eigenvalues of each fixed point has the absolute value strictly less than one. \msk
\item[---] The fixed points, $\beta_0$ and $\beta_1$ have codimension one stable manifold and one dimensional unstable manifold. \msk
\end{itemize}
\nin The orientation preserving higher dimensional H\'enon-like map is called \textit{renormalizable} if $ W^u(\beta_0) $ and $W^s(\beta_1) $ intersects in the single orbit of a point. Observe that {\em sectional dissipativeness} imply that each fixed point has condimension one stable manifold and one dimensional unstable manifold.

\nin Let $p_0 \in \Orb_{\Z}(w)$ be the point  which is farthest point from $\beta_1$ along the component of $ W^s_{\loc}(\beta_1)$ which contains $ \beta_1 $. Let $p_k = F^k(p_0)$ for each $k \in \Z$. Then $ p_0 $ and the forward orbit of $ p_0 $ are on $  W^s_{\loc}(\beta_1) $. The local stable manifold of $p_{-n}$, $ W^s_{\loc}(p_{-n}) $ where $ n \leq 0 $ is pairwise disjoint component of  $ W^s(\beta_1)$ and $ W^s_{\loc}(p_{-n}) $ converges to $ W^s(\beta_1)$ because $ p_{-n} $ converges to $\beta_0 $ as $ n \rightarrow +\infty $. 
Let $M_{-n}$ be the connected component of $ W^s(\beta_1) $ which contains $ p_{-n} $, say $ W^s_{\loc}(p_{-n}) $, to be  for every $n \geq 0$. For instance, $M_0$ denotes $ W^s_{\loc}(\beta_1)$. Moreover, we can define $M_1$ as the component of $ W^s(\beta_1)$ whose image under $F$ is contained in $M_{-1}$ such that it does not have any point of $\Orb_{\Z}(w)$. $ M_{1} $ is on the opposite side of $M_{-1}$ from $M_0$. We may assume that $M_1$ is a curve connecting the up and down sides of the square domain $B$ inside. Then we can easily check the curves $ [p_0, p_1]^u_{\beta_0} $ and  $ [p_1, p_2]^u_{\beta_0} $ which are parts of $ W^u(\beta_0) $ does not intersect $M_1$ and $M_{-1}$ respectively when $F$ is renormalizable.

\begin{figure}[htbp]
\begin{center}
%
\psfrag{M1}[c][c][0.7][0]{\normalsize $M_1$}
\psfrag{M-1}[c][c][0.7][0]{\normalsize $M_{-1}$}
\psfrag{M-2}[c][c][0.7][0]{\normalsize $M_{-2}$}
\psfrag{M-3}[c][c][0.7][0]{\normalsize $M_{-3}$}
\psfrag{Wb0}[c][c][0.7][0]{\normalsize $W^s_{\loc}(\beta_1)$}
\psfrag{M0}[c][c][0.7][0]{\normalsize $W^s_{\loc}(\beta_1)$}
\psfrag{b0}[c][c][0.7][0]{\large $\beta_0$}
\psfrag{p0}[c][c][0.7][0]{\large $p_0$}
\psfrag{p1}[c][c][0.7][0]{\large $p_1$}
\psfrag{p2}[c][c][0.7][0]{\large $p_2$}
\psfrag{p-1}[c][c][0.7][0]{\large $p_{-1}$}
\psfrag{p-2}[c][c][0.7][0]{\large $p_{-2}$}
\psfrag{p-3}[c][c][0.7][0]{\large $p_{-3}$}
\psfrag{b1}[c][c][0.7][0]{\large $\beta_1$}
\includegraphics[scale =0.8]{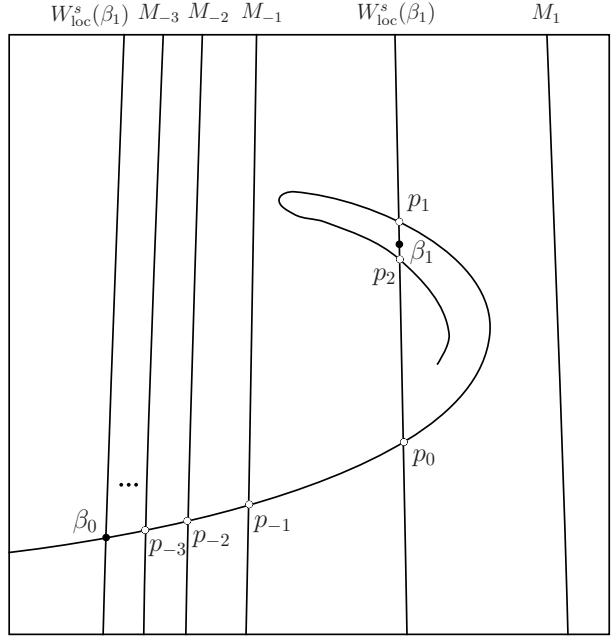}
\end{center}
\caption{Stable and unstable manifolds at fixed points}
\label{fig:invariant domain}

\end{figure}

\msk

\subsection{Renormalization of $ m+2 $ dimensional H\'enon-like maps} \label{def of renormal}
The analytic definition of period doubling renormalization of H\'enon-like map requires conjugation which is not just dilation because $F^2$ is not H\'enon-like map, that is, the image of the hyperplane, $ \lbrace x=C \rbrace $ in $B$ under $F^2$ is not the part of the hyperplane, $ \lbrace y=C \rbrace $. Thus we need non-linear coordinate change map for renormalization. Define the \textit{horizontal-like diffeomorphism} as follows
\begin{align} \label{horizontal-like diffeo}
  H(x,y,z) = (f(x) - \eps (x,y,\Bz), \ y, \ \Bz - \bde (y, f^{-1}(y), \B0)) .
\end{align}
The inverse of $ H $, $ H^{-1} $ is as follows
\begin{align*}
H^{-1}(x,y,\Bz) & \equiv  (\phi^{-1}(w),\ y,\ \Bz + \bde(y,\,f^{-1}(y),\,\B0))  
\end{align*}
where $ \phi^{-1} $ is the straightening map satisfying $ \phi^{-1} \circ H(w) = x $ for $ w = (x, y, \Bz) $. 
\ssk \\
Let us define $ \Dom(H) $ as the region enclosed by hypersurfaces, $ \{f(x) - \eps(x,y,\Bz) = \const .\} $, $ \{y = \const . \} $ and $ \{z_j = \const . \} $ for $ 1 \leq j \leq m $ such that the image of this region under $ H $ is $ V \times {\bf I}^v $. Let $ \UU_{\,U} $ be the space of unimodal maps defined on the set $U$.

\ssk
\begin{prop} \label{preconv}
Let $F(w) = (f(x) - \eps(w), \  x, \ \bde(w)) $ be a higher dimensional H\'enon-like map and  $H$ be the horizontal-like diffeomorphism defined in \eqref{horizontal-like diffeo} where $\| \:\! \eps \| \leq \bar\eps$ and $\| \:\! \bde\| \leq \bar\de$ with small enough positive numbers $ \bar \eps $ and $ \bar \de $. Suppose that the unimodal map $ f $ is renormalizable. Then there exists a unimodal map $f_1 \in \UU_{\,V}$ such that $\| f_1 -f^2 \|_V < C\bar\eps$ and the map $H \circ F^2 \circ H^{-1}$ is the H\'enon-like map $(x,y,z) \mapsto (f_1(x) - \eps_1(x,y,\Bz), \ x, \ \bde_1(x,y,\Bz))$ with the norm, $\| \:\! \eps_1 \| = O(\bar\eps^2 + \bar\eps \bar\de) $ and $\| \:\! \bde_1 \| = O(\bar\eps \bar\de + \bar\de^2)$.
\end{prop}

\begin{proof}
By the straightforward calculation, $H \circ F^2 \circ H^{-1} $ is 
\comm{********************
\begin{equation*}
\begin{aligned}
(x,\ y,\ \Bz) \ 
\xmapsto{\hspace*{3pt}H^{-1}} & \ (\phi^{-1}(w),\ y,\ z + \bde(y, f^{-1}(y),\B0))  \\
\xmapsto{\hspace*{7pt} F \hspace*{7pt}} & \ (x,\ \phi^{-1}(w),\ \bde \circ H^{-1}(w)) \\
\xmapsto{\hspace*{7pt} F \hspace*{7pt}} & \ (f(x) - \eps \circ F \circ H^{-1}(w),\ x,\ \bde \circ F \circ H^{-1}(w)) \\
\xmapsto{\hspace*{6pt} H \hspace*{7pt}} & \ (f(f(x) - \eps \circ F \circ H^{-1}(w)) - \eps \circ F^2 \circ H^{-1}(w),\\
& \hspace{2.5in} x,\ \bde \circ F \circ H^{-1}(w) - \bde(x, f^{-1}(x),\B0))
\end{aligned} \msk
\end{equation*} 
**********************************}
\begin{equation*}
(f(f(x) - \eps \circ F \circ H^{-1}(w)) - \eps \circ F^2 \circ H^{-1}(w),\ \ x,\ \ \bde \circ F \circ H^{-1}(w) - \bde(x, f^{-1}(x),\B0))
\end{equation*}
Thus the first coordinate function of $H \circ F^2 \circ H^{-1}$ is 
$$ f(f(x) - \eps \circ F \circ H^{-1}(w)) - \eps \circ F^2 \circ H^{-1}(w) . $$
By the linear appoximation, we obtain the estimation of first coordinate function. \ssk
\begin{align*}
& f(f(x)) - \eps \circ F \circ H^{-1}(w)) - \eps \circ F^2 \circ H^{-1}(w) \\[0.4em]
= \ & f^2(x) - f'(f(x)) \cdot \eps \circ F \circ H^{-1}(w) - \big[\,\eps(f(x),\,x,\, \B0) + \di_x \eps \circ (f(x),\,x,\, \B0) \cdot \eps \circ F \circ H^{-1}(w) \\
\quad & + \sum_{j=1}^m \di_{z_j} \eps \circ (f(x),\,x,\,\B0) \cdot \de^j \circ F \circ H^{-1}(w)\,\big] + h.\,o.\,t. \\[0.4em]
= \ & f^2(x) - v \circ f(x) - [\,f'(f(x)) - \di_x \eps \circ (f(x),\,x,\, \B0)\, ] \cdot v(x) \\[0.4em]
\quad & - \big[\,f'(f(x)) - \di_x \eps \circ (f(x),\,x,\, \B0)\,\big] \cdot \Big[\,\di_y \eps \circ (x, f^{-1}(x), \B0) \cdot (f^{-1})'(x) \cdot \eps \circ H^{-1}(w) \\
 & \qquad + \sum_{j=1}^m \di_{z_j} \eps \circ (x, f^{-1}(x), \B0) \cdot \de^j \circ H^{-1}(w) \,\Big]
 \\[-0.4em]  &
  - \sum_{j=1}^m \di_{z_j} \eps \circ (f(x),\,x,\, \B0) \cdot \de^j \circ F \circ H^{-1}(w) + h.\,o.\,t. 
\end{align*} 
Let $ v(x) = \eps(x,\, f^{-1}(x),\,\B0) $. 
Then the unimodal map, $ f_1(x)$ of the first component of $H \circ F^2 \circ H^{-1}$ is the following
$$ f^2(x) - v \circ f(x) - \big[\,f'(f(x)) - \di_x \eps \circ (f(x),\,x,\, \B0)\,\big] \cdot v(x) . $$ 
Thus $ \| f_1(x) - f^2(x) \| = O(\| \:\!\eps \|)$. Moreover, the norm of $ \eps_1(w) $ is $ O \big(\| \:\!\eps \|^2 + \| \:\! \eps \| \, \| \:\! \bde \| \,\big) $. 
Let us estimate the the from third to $ (m+2)^{th} $ coordinate maps of $H \circ F^2 \circ H^{-1}$. Recall $ \bde_1(w) = \bde \circ F \circ H^{-1}(w) - \bde(x,\,f^{-1}(x),\,\B0) $. The estimation of each coordinate map of $ \bde_1 $ is
\msk
\begin{equation*}
\begin{aligned}
& \de^i \circ F \circ H^{-1}(w) - \de^i(x,\,f^{-1}(x),\,\B0) \\[0.6em]
= \ & \de^i \big(x,\ \phi^{-1}(w),\ \bde \circ H^{-1}(w) \big) - \de^i(x,\,f^{-1}(x),\,\B0) \\[0.3em]
= \ & \di_y \de^i \circ (x,\,f^{-1}(x),\,\B0) \cdot (\phi^{-1}(w) - f^{-1}(x)) + \sum_{j=1}^m \di_{z_j} \de^i \circ (x,\,f^{-1}(x),\,\B0) \cdot \de^j \circ H^{-1}(w) +  h.\,o.\,t. \\[0.3em]
= \ & \di_y \de^i \circ (x,\,f^{-1}(x),\,\B0) \cdot (f^{-1})'(x) \cdot \eps \circ H^{-1}(w) \\
& \qquad + \sum_{j=1}^m \di_{z_j} \de^i \circ (x,\,f^{-1}(x),\,\B0) \cdot \de^j \circ H^{-1}(w) +  h.\,o.\,t. 
\end{aligned} 
\end{equation*}
for $ 1 \leq i \leq m $. \ssk Then $ \| \:\! \de_1^i \| $ is $ O(\| \:\!  \eps \| \, \| \:\!  \bde \| + \| \:\!  \bde \|^2) $ for all $ 1 \leq i \leq m $. Hence, $ \| \:\! \bde_1 \| $ is also $ O(\| \:\!  \eps \| \, \| \:\!  \bde \| + \| \:\!  \bde \|^2) $.

\end{proof}

\nin Define {\em pre-renormalization} of $F$ as $ PRF \equiv H \circ F^2 \circ H^{-1}$ on $ H(B^1_v) $. Define  $ H(B^1_v) $ as $ \Dom(PRF) $ unless any other statements are specified. Since $ H(B^1_v) $ is the hypercube, domain $ B $ is recovered as the image of $ H(B^1_v) $ under the appropriate linear expanding map $\La(x,y,\Bz) = (sx, sy, s\Bz)$ for some $s<-1$. Thus we see that $ \Dom(PRF) $ is $\La^{-1}(B)$.
\comm{*******************
Notice that the map $ H^{-1} $ from $\La^{-1}(B)$ to $B^1_v$ preserves each hyperplane, $ \{\,y =C \} $ and the map $ F \circ  H^{-1} $ from $\La^{-1}(B)$ to $B^1_c$ preserves each hyperplane, $ \{\,x =C \} $ for every constant $ C $. Let $B^1_c$ be 
$ F(B^1_v) $. 
\begin{equation*}
 \begin{aligned}
 H^{-1} : & \ \La^{-1}(B) \lra B^1_v ,  & (x,y,\Bz) & \mapsto (\phi^{-1}(x,y,\Bz) ,\ y,\ \Bz + \bde(y, f^{-1}(y),\B0)) & \\
 F \circ  H^{-1} : & \ \La^{-1}(B) \lra B^1_c,  &  (x,y,\Bz) &\mapsto  (x,\ \phi^{-1}(x,y,\Bz),\ \bde \circ H^{-1}(w)) &
 \end{aligned} \ssk
 \end{equation*}
 ****************************}
\begin{defn}[Renormalization]
Let  $V$ be the (minimal) closed subinterval of $I^x$ such that  $V \times {\bf I}^v$ is invariant under $ H \circ F^2 \circ H^{-1}$ and let $s \colon V \ra I $ be the orientation reversing affine rescaling. With the rescaling map $\La(x,y,\Bz) = (sx, sy, s\Bz)$, The {\it renormalization} of the $ m+2 $ dimensional H\'enon-like map is defined as $\La \circ H \circ F^2 \circ H^{-1} \circ \La^{-1}$ on the domain $ B \equiv I^x \times \bf{I}^v $
\begin{align*}
RF = \La \circ H \circ F^2 \circ H^{-1} \circ \La ^{-1} .
\end{align*}
If $RF$ is also renormalizable, we can define the second renormalization of $F$ as the renormalization of $RF$. Then if $F$ is $n$ times renormalizable, then the $n^{th}$ renormalization is defined successively
\begin{align*}
 R^{n}F =  \La_{n-1} \circ H_{n-1} \circ (R^{n-1}F)^2 \circ H^{-1}_{n-1} \circ \La_{n-1} ^{-1} .
\end{align*}
where $ R^{n-1}F $ is the $(n-1)^{th}$th renormalization of $F$ for $n \geq 1 $.
\end{defn}
\nin Let the set of infinitely renormalizable higher dimensional H\'enon-like map be $ \II(\bar \eps) $ where $ \max \{ \,\| \eps \|,\,\| \bde \| \ \} \leq \bar \eps $ for small enough positive number $ \bar \eps $. We call the higher dimensional H\'enon-like map just H\'enon-like map unless it makes any confusion in the following sections.

\bsk
\section{Critical Cantor set} \label{critical cantor set}
The minimal attracting set for two dimensional infinitely renormalizable H\'enon-like map $ F $ is the Cantor set on which $ F $ acts as the dyadic adding machine. The topological construction of the invariant Cantor set of higher dimensional H\'enon-like map is exactly same as that of two-dimensional H\'enon-like map (Corollary \ref{dyadic adding machine} below). Thus we use the same definitions and notions of the two dimensional case in this section. 
\ssk \\
Denote $\Psi^1_v $ by $ \psi^1_v \equiv H^{-1} \circ \La^{-1} $. Thus it is the non-linear scaling map which conjugates $F^2$ to $RF$ on $\Psi^1_v (B)$. Denote $\Psi^1_c $ by $ \psi^1_c \equiv F \circ \psi_v$. The subscript $v$ and $c$ are associated to the maps with the {\it critical value} and the {\it critical point} respectively.
Similarly, let $\psi^2_v$ and $\psi^2_c$ be the non linear scaling maps conjugating $RF$ to $R^2F$. Let
$$
 \Psi^2_{vv} = \psi^1_v \circ \psi^2_v, \ \  \Psi^2_{cv} = \psi^1_c \circ \psi^2_v, \ \  \Psi^2_{vc} = \psi^1_v \circ \psi^2_c,  \ldots
$$
and so on. Successively we can define the non-linear scaling map of the $n^{th}$ level for any $n \in \N$ as follows 
$$
 \Psi^n_{\bf w} = \psi^1_{w_1} \circ \cdots \circ \psi^n_{w_n},  \ \ {\bf w}=(w_1, \ldots , w_n) \in \lbrace v, c \rbrace^n 
$$
where ${\bf w}=(w_1, \ldots , w_n)$ is the word of length $n$ and $W^n = \lbrace v, c \rbrace^n$ is the $n$-fold Cartesian product of $\lbrace v, c \rbrace$. 

\begin{lem} \label{diameter}
Let the H\'enon-like map, $F $ be in $ \II(\bar{\eps})$. Then the derivative of the map $\Psi^n_{\bf w}$ is exponentially shrinking for $ n \in \N $ with $ \si $, that is, $\| D\Psi^n_{\bf w} \| \leq C \si^n$ for every words ${\bf w} \in W^n$ where $C>0$ depends only on $B$ and $\bar{\eps}$.
\end{lem}

\begin{proof}
The identical equation $ H \circ H^{-1} = \id $ implies that
\footnote{ The first coordinate map of $ H^{-1}(w) $, $ \phi^{-1}(x,y,z) $ is {\em not} the {\em inverse function} of the some function $ \phi(w) $. However, $ \phi^{-1}(w) $ is a $ \eps - $perturbation of $ f^{-1}(x) $ as follows 
$$ f \circ \phi^{-1}(w) - \eps \circ H^{-1}(w) = x . $$  }
$$ f \circ \phi^{-1}(w) - \eps \circ H^{-1}(w) = x . $$
Thus we have $ \phi^{-1}(w) = f^{-1}(x + \eps \circ H(w)) $. Recall that 
$$ \eps \circ H^{-1}(w) = \eps ( \phi^{-1}(w),\ y,\ z + \de(y, f^{-1}(y), \B0)) .$$ 
Then by the chain rule, each partial derivatives of $ \phi^{-1} $ is as follows \msk
\begin{equation*}
\begin{aligned}
\di_x \phi^{-1}(w) &= \ (f^{-1})'(x + \eps \circ H^{-1}(w)) \cdot \big[\, 1 + \di_x \eps \circ H^{-1}(w) \cdot \di_x \phi^{-1}(w) \,\big] \\[0.4em]
\di_y \phi^{-1}(w) &= \ (f^{-1})'(x + \eps \circ H^{-1}(w)) 
\cdot \Big[\, \di_x \eps \circ H^{-1}(w) \cdot \di_y \phi^{-1}(w) + \di_y \eps \circ H^{-1}(w) 
\\ 
&\qquad + \sum_{j=1}^m \di_{z_j} \eps \circ H^{-1}(w) \cdot \frac{d}{dy}\, \de(y,f^{-1}(y),\B0) \,\Big] \\
\di_{z_i} \phi^{-1}(w) &= \ (f^{-1})'(x + \eps \circ H^{-1}(w)) \cdot \big[\, \di_x \eps \circ H^{-1}(w) \cdot \di_{z_i} \phi^{-1}(w) + \di_{z_i} \eps \circ H^{-1}(w) \,\big] .
\end{aligned} \msk
\end{equation*}
for $ 1 \leq i \leq m $. Then
\begin{equation} \label{partial derivatives of phi-1}
\begin{aligned}
\di_x \phi^{-1}(w) &= \ \frac{(f^{-1})'(x + \eps \circ H^{-1}(w))}{1 - (f^{-1})'(x + \eps \circ H^{-1}(w)) \cdot \di_x \eps \circ H^{-1}(w)} \\ 
\di_y \phi^{-1}(w) 
 &= \ \di_x \phi^{-1}(w) \cdot  \Big[\,\di_y \eps \circ H^{-1}(w) + \sum_{j=1}^m \di_{z_j} \eps \circ H^{-1}(w) \cdot \frac{d}{dy}\, \de^j(y,f^{-1}(y),\B0)\,\Big] \\[0.3em]
\di_{z_i} \phi^{-1}(w) 
&= \ \di_x \phi^{-1}(w) \cdot \di_{z_i} \eps \circ H^{-1}(w) .
\end{aligned} \msk
\end{equation}
for $ 1 \leq i \leq m $. Let us estimate $ \| \:\! D \phi^{-1} \| $. The above equation \eqref{partial derivatives of phi-1} implies that $ \| \:\! \di_x \phi^{-1} \| \asymp \| (f^{-1})' \| $ \ and furthermore, $  \phi^{-1}(x,y_0,\Bz_0) \asymp f^{-1}(x) $ for every $ (y_0,\Bz_0) \in {\bf I}^v $. The fact that $ \phi^{-1} \colon \La^{-1}(B) \ra \pi_x(B^1_v) $ implies that the domain of $ f^{-1} $ is $ \pi_x(\La^{-1}(B)) $. \ssk Then $ \| (f^{-1})' \| $ is away from one and $ \| \:\! \di_y \phi^{-1} \| $ and $ \| \:\! \di_{z_i} \phi^{-1} \| $ are $ O(\bar \eps) $ for $ i=1,2,\ldots,m $ by the equation \eqref{partial derivatives of phi-1}. 
\ssk \\
The norm of derivatives of $ \phi^{-1}_n(w) $ for each $ n $ has the same upper bounds because $ f_n \ra f_* $ exponentially fast and $ \| \:\! \di_y \phi^{-1}_n \| $ and $ \| \:\! \di_{z_i} \phi^{-1}_n \| $ for $ i=1,2,\ldots,m $ are bounded by $ O(\bar \eps^{2^n}) $. The dilation, $ \La^{-1} $ contracts by the factor $ \si(1 + O(\rho^n)) $ where $ \rho = \dist(F, F_*) $. The above estimations imply that $ \| DH^{-1}_n \| $ and $ \| D(F\circ H^{-1}_n) \| $ are uniformly bounded and the upper bounds are independent of $ n $. Thus $ \| \psi^n_w \| \leq C\si $ for $ w =v, c $. Hence, the composition of $ \psi^k_w $ for $ k = 1,2,\ldots, n $ contracts by the factor $ C\si^n $, that is, $ \| D\Psi_{\bf w} \| \leq C\si^n $ for some $ C>0 $.

\end{proof}
\nin Let $B_v^1 \equiv B_v^1(F)$ and $B_c^1 \equiv B_c^1(F)$. Thus by the definitions on Section \ref{def of renormal}, $ B_v^1(F) = \psi^1_v(B) $ and $ B_c^1(F) = F \circ \psi^1_v(B) $. Moreover, $F(B_c^1) \subset B_v^1 $. 
If the H\'enon-like map $F$ is $ n $ times renormalizable, we can define $ B_v^1(R^nF) $ and $ B_c^1(R^nF) $ as $ \psi^{n+1}_v(B) $ and $F_n \circ \psi^{n+1}_v(B) $ respectively for each $n \geq 1$. We call the regions $B^n_{\bf w} \equiv B^n_{\bf w}(F) = \Psi^n_{\bf w}(B)$ {\em pieces} of the $n^{th}$ level or $n^{th}$ {\em generation} where ${\bf w} \in W^n$. 
Moreover, $W^n$ can be a additive group under the following correspondence from $W^n$ to the numbers with base 2 of mod $2^n$
$$ 
{\bf w} \mapsto \sum^{n-1}_{k=0} w_{k+1} 2^k \qquad  (\mod 2^n)
$$
where the symbols $v$ and $c$ are corresponding to $0$ and $1$ respectively. 
\begin{cor} \label{diameter 2}
The diameter of each piece shrinks exponentially fast for each $ n \geq 1 $, that is, $\diam (B_{\bf w}^n) \leq C\si^n$ for all ${\bf w} \in W^n$ where the constant $C > 0$ depends only on $B$ and $\bar \eps$.
\end{cor}
The construction of the critical Cantor set for higher dimensional H\'enon-like map in $ \II(\bar \eps) $ is the same as that of two dimensional H\'enon-like map. See Section 5 in \cite{CLM}. Define the critical Cantor set of the infinitely renormalizable H\'enon-like map $F$ as follows
$$
\OO \equiv \OO_F = \bigcap_{n=1}^{\infty} \bigcup_{{\bf w} \in W^n} B^n_{\bf w}.
$$
Infinitely renormalizable H\'enon-like map, $ F $ acts as the dyadic adding machine on the above invariant Cantor set. For detailed construction of the dyadic group as Cantor set, see ~\cite{BB}.

\msk

\section{Average Jacobian}
\nin Let us consider the average Jacobian of  the infinitely renormalizable map $F$. The definition and properties of average Jacobian of higher dimensional H\'enon-like maps is the same as those of two dimensional ones. For the sake of completeness, we describe Lemma and Theorem in this section below. Let the Jacobian determinant of $ F $ at $w$ be $ \Jac F(w) $. Thus
$$
\log \left| \frac{\Jac F(y)}{\Jac F(z)} \right| \leq C \quad \text{for any} \ \ y,z \in B
$$
by some constant $C$ which is not depending on $y$ or $z$. The diameter of the domain $B^n_{\bf w} $ converges to zero exponentially fast by Lemma \ref{diameter}. It implies the following lemma.

\begin{lem}[Distortion Lemma]  \label{distortion}
There exist a constant $C$ and the positive number $\rho <1$ satisfying the following estimate
$$
\log \left| \frac{\Jac F^k(y)}{\Jac F^k(z)} \right| \leq C \rho^n \quad \textrm{for any} \ \ y,z \in B^n_{\bf w}
$$
where $k= 1, 2, 2^2, \ldots , 2^n$.
\end{lem}

\nin Existence of the unique invariant probability measure, say $\mu$, on $\OO_F$ enable us to define the average Jacobian.
$$
b_F \equiv b = \exp \int_{\OO_F} \log \Jac F \; d\mu \, .
$$
On each level $ n $, the measure $\mu$ on $\OO_F$ satisfies that $\mu (B^n_{{\bf w}_n} \cap \OO_F) = 1/ 2^n$ for every word ${\bf w}_n$ of length $ n $. 
\msk
\begin{cor} \label{average}
For any piece of $B^n_{{\bf w}}$ on the level $ n $ and any point $ w \in B^n_{{\bf w}} $,
$$
\Jac F^{2^n} \!(w) = b^{2^n} (1 + O(\rho ^n))
$$
where $ b $ is the average Jacobian of $ F $ for some positive $ \rho < 1 $.
\end{cor}

\begin{proof}
Since
$$
\int _{B^n_{\bf w}} \log \Jac F^{2^n} d\mu = \int _{\OO} \log \Jac F \; d\mu = \log b,
$$
there exists a point $\eta \in B^n_{\bf w}$ such that  $ \log \Jac F^{2^n} \!(\eta) = \dfrac{\log b}{\mu(B^n_{\bf w})}  = 2^n \log b$ \\
For any $ w \in B^n_{\bf w}$, $\log \, \Jac F^{2^n} \! (z) \leq C \rho^n + \log \, \Jac F^{2^n} \!(\eta)$, and $ O(\rho^n) = \log (1 + O(\rho^n))$ for a fixed constant $\rho$. Then
\begin{align*}
\log \Jac F^{2^n} (w) &= \log (1 + O(\rho^n)) + \log \Jac F^{2^n} (\eta)\\
                            &= \log \big[(1 + O(\rho^n)) \cdot  b^{2^n} \big]
\end{align*}
\msk
Hence, \  $ \Jac F^{2^n} (w) = b^{2^n} (1 + O(\rho^n)) $. 
\end{proof}
\nin Since $ F $ is the $ m+2 $ dimensional map, it has Lyapunov exponents $\chi_0, \chi_1, \ldots , \chi_{m+2}$. Let $\chi_0$ be the maximal one. 
Since $F$ is ergodic with respect to the invariant finite measure $\mu$ on the critical Cantor set, we get the following inequality for any Lyapunov exponent $ \chi $
$$
| \: \! \mu |\, \chi(x) \leq \int_{\OO_F} \log \| DF(x) \| \; d\mu(x)
$$
where $| \:\!\mu| $ is the total mass of $ \mu $ on $ \OO_F $.
\msk
\begin{thm} \label{max exponent is 1}
The maximal Lyapunov exponent of $F$ on $\OO_F$ is 0.
\end{thm}
\begin{proof}
Let $\mu_n$ be \,$2^n \mu | _{B^n_w}$, an invariant measure under $F^{2^n}$ and let $\nu_n$ be the (unique) invariant measure on $R^nF |_{\;\!\OO_{R^nF}}$. Then
$$
2^n \chi_0(F, \mu) = \chi_0(F^{2^n} |_{B^n_{v^n}}, \mu_n) = \chi_0(R^nF, \nu_n) \leq \int_{B^n_w} \log \| D(R^nF) \|\, d\nu_n \leq C 
$$
 for every $ n \in \N$, where $C$ is a constant independent of $n$. The last inequality comes from the uniformly bounded $C^1$ norm of derivative of $R^nF$. Then the maximal Lyapunov exponent $\chi_0 \leq 0$. If $\chi_0 <0$, then the support of $\mu$ contains some periodic cycles by Pesin's theory. But $\OO_F$ does not contain any periodic cycle because $ F $ acts on $ \OO_F $ as a dyadic adding machine. Hence, $\chi_0 =0$. 
\end{proof}
\bsk

\section{Universal expression of Jacobian determinant}  \label{universality of Jacobian}

\nin The universality of average Jacobian is involved with the asymptotic behavior of the non linear scaling map $\Psi^n_{v^n}$ between the renormalized map $F_n \equiv R^nF$ and $F^{2^n}$ for each $n \in \N$. $ \Psi^n_{v^n}$ conjugate $F^{2^n}$ to $F_n$. Thus using chain rule and Corollary \ref{average}, $ \Jac F_n $ is the product of the average Jacobian of $F^{2^n}$ and the ratio of the $ \Jac \Psi^n_{v^n} $ at $ w $ and $ F_n(w) $ as follows
\begin{equation} \label{chain rule}
\begin{aligned}
\Jac F_n(w) &= \Jac F^{2^n}(\Psi^n_{v^n}(w)) \frac{\Jac \Psi^n_{v^n}(w)}{\Jac \Psi^n_{v^n}(F_n(w))}  \\
                  &= b^{2^n} \frac{\Jac \Psi^n_{v^n}(w)}{\Jac \Psi^n_{v^n}(F_n(w))} (1+ O(\rho^n)).
\end{aligned} \msk
\end{equation}
Denote the word, $ v^n $ by $ \Bv $. 
\nin Then the universality of Jacobian of $ \Psi^n_{v^n}$ implies the universality of $ \Jac F_n $ in Theorem \ref{Universality of the Jacobian} below. The asymptotic of non-linear part of $ \Psi^n_{\Bv}$ is essential to the universal expression of $ \Jac \Psi^n_{\Bv} $.

\msk 
\subsection{Asymptotic of $\Psi^n_k$ for fixed $ k^{th} $ level } \label{asymptotic coordinate}
For every infinitely renormalizable H\'enon-like map $F$, we have a well defined {\em tip}
\begin{align}
\{\tau \} \equiv \{ \tau_F \} = \bigcap_{n \ge 0} B^n_{v^n}
\end{align}
where the pieces $B^n_{v^n}$ are defined as $ \Psi^n_{\Bv}(B(R^nF)) $. The tip of the renormalized map, $ R^kF $ is denoted by $\tau_k =\tau (R^k F)$ for each $ k \in \N $. Since every $ B^n_{\Bv}(F) $ contains $ \tau_F $, let us condense the notation $ \Psi^n_{\Bv} $ into $ \Psi^n_{\tip} $. Moreover, in order to simplify the notation and calculations, we would let the tip move to the origin as a fixed point of each $\Psi^1_v (R^kF)$ for every $ k \in \N $ by conjugation of the appropriate translations. Let us define $ \Psi_k^{k+1} $ in this section.\footnote{If we need to distinguish the scaling maps, $ \Psi^n_k $ around tip from its composition with translations, then we use the notation, $ \Psi^n_{k,\, \tip} $ }
\ssk
\begin{align} \label{origin as tip}
    \Psi_k(w) \equiv \Psi_k^{k+1}(w) = \Psi^{k+1}_v (w + \tau_{k+1}) - \tau_k 
\end{align}

\nin Let the derivative of the map defined $ \Psi_k $ on \eqref{origin as tip} at the origin be $D_k \equiv D_k^{k+1}$.
\begin{equation*}
\begin{aligned}   
  \quad D^{k+1}_k \equiv D_k = D\Psi^{k+1}_k(\B0) &= D(\Psi^1_v(R^k F))(\tau_{k+1} ) & \\
       &= D(T_k \circ \Psi^1_v (R^kF) \circ T^{-1}_{k+1})(\B0) &
\end{aligned} 
\end{equation*}       
where \ $T_j : w \mapsto  w-\tau_j$  for $ j= k,\, k+1$.
Then we can decompose $ D_k $ into the matrix of which diagonal entries are 1s and the diagonal matrix as follows
\ssk
\begin{equation} \label{decomposition}
\begin{aligned}   
 \left(    \begin{array}{ccccc}
      1 &  t_k   & u_k^1 & \cdots & u_k^m    \\[0.2em]
        &  1     &       &        &          \\[0.2em]
        &  d_k^1 & 1     &        &          \\
        & \vdots &       & \ddots &          \\
        &  d_k^m &       &        & 1
    \end{array}      \right)
 \left(    \begin{array}{ccccc}
\alpha_k &       &       &        &           \\
         & \si_k &       &        &           \\
         &       & \si_k &        &           \\
         &       &       & \ddots &           \\
         &       &       &        &  \si_k 
    \end{array}       \right) 
 =
 \left(    \begin{array} {cccccc}
\alpha_k & \si_k t_k   & \si_k u_k^1 & \cdots & \si_k u_k^m  &   \\[0.2em]
         & \si_k       &             &        &              &   \\[0.2em] \cline{3-5}
         & \multicolumn{1}{c|}{\si_k d_k^1} & \multicolumn{3}{c}{\multirow{3}{*}{ $ \si_k \cdot \Id_{m \times m}  $ }} & \multicolumn{1}{|c}{}\\
         & \multicolumn{1}{c|}{\vdots}      &  &  &  & \multicolumn{1}{|c}{} \\ 
         & \multicolumn{1}{c|}{\si_k d_k^m} &  &  &  & \multicolumn{1}{|c}{} \\       \cline{3-5}
    \end{array}       \right)       
\end{aligned}  \msk
\end{equation}
where $ \Id_{m \times m} $ is the $ m \times m $ identity matrix. 
We condense the expression of $ D_k $ using the boldfaced letters. Let $ m $ dimensional vector $ (\,d_k^1 \ d_k^2 \ldots d_k^m\;)^{Tr} $ be $ \Bd_k $ and $ (\,u_k^1 \ u_k^2 \ldots u_k^m\,) $ be $ \Bu_k $ where $ Tr $ is the transpose of the matrix. Recall that $\si_k = - \si \left(1+O(\rho^k) \right)$.
Decompose $\Psi^{k+1}_k$ into the linear and non-linear parts. Then
\begin{equation}
\begin{aligned}
D_k = 
\begin{pmatrix}
1 & t_k   & \Bu_k \\
  & 1     &       \\
  & \Bd_k & 1
\end{pmatrix}
\begin{pmatrix}
\alpha_k &       &     \\
         & \si_k &     \\
         &       & \si_k \cdot \Id_{m\times m}
\end{pmatrix} =
\begin{pmatrix}
\alpha_k & \si_k \;\! t_k & \si_k \;\! \Bu_k     \\
         & \si_k          &      \\
         & \si_k \;\! \Bd_k    & \si_k
\end{pmatrix}
\end{aligned}
\end{equation}
\begin{align} \label{Psi on k level}
\Psi^{k+1}_k \equiv \Psi_k(w) = D_k \circ (\id+ {\bf s}_k)(w)
\end{align}
where $w=(x, y, \Bz)$ \ and   ${\bf s}_k (w)= (s_k(w),\; 0,\; r_k^1(y),\;r_k^2(y),\;\ldots, r_k^m(y)) = O(|w|^{\:\!2})$\ near the origin. Denote $ ( r_k^1(y),\;r_k^2(y),\;\ldots, r_k^m(y)) $ by $ {\bf r}_k $. 
Comparing the derivative of $ H^{-1}_k \circ \La^{-1}_k $ at the tip and $D_k$ and Corollary \ref{diameter 2}, we obtain the following estimations
\begin{equation} \label{bounds at 0}
\begin{aligned}
t_k &= \di_y \phi_k^{-1} (\tau_{k+1}) =  \di_x \phi_k^{-1}(\tau_{k+1}) \cdot \Big[\, \di_y \eps_k (\tau_k) + \sum_{j=1}^m \di_{z_j} \eps_k(\tau_{k}) \cdot d_k^j \,\Big] \\
u_k^i &= \di_{z_i} \phi_k^{-1} (\tau_{k+1}) = \di_x \phi_k^{-1}(\tau_{k+1}) \cdot \di_{z_i} \eps_k(\tau_k) \\[0.2em]
 \textrm{and}\quad d_k^{\,i} &= \frac{d}{dy}\,\de_k^i \left(\pi_y(\tau_{k+1}),\, f^{-1}_k(\pi_y(\tau_{k+1})),\,\B0 \right) 
\end{aligned} \ssk
\end{equation}
where $ \phi^{-1}_k(w) = \pi_x \circ H^{-1}_k(w) $ for $ 1 \leq i \leq m $. Thus the norm of each element of $ D_k $, $ |\:\!t_k| $, $ \| \:\! \Bu_k \| $ and $ \| \:\! \Bd_k \| $ is bounded by $ O(\bar \eps^{2^k}) $. Since $ \| \:\! \di_x\phi^{-1}_k \| $ at the tip exponentially converges to $\si$ as $k \ra \infty$, $\alpha_k = \si^2 \left(1+O(\rho^k) \right)$ for some $\rho \in (0,1)$.   

\msk 
\begin{lem} \label{bounds on domain}
Let $s_k$ be the function defined on \eqref{Psi on k level}. For each $k \in \N$, \msk
\begin{enumerate}
\item
$ \ | \, \partial_xs_k| \; = O(1), \qquad \ \ | \, \partial_y s_k|  \ = O(\bar \eps^{2^k}), \qquad | \, \partial_{z_i} s_k| \ = O(\bar \eps^{2^k})  $ \ssk
  \item 
$ \ | \, \partial^2_{xx} s_k| = O(1), \qquad  \ \ \! | \, \partial^2_{xy} s_k| = O(\bar \eps^{2^k}), \qquad | \, \partial^2_{yy} s_k| = O(\bar \eps^{2^k}) $ \ssk
 \item
$ \; | \, \partial^2_{yz_i} s_k| = O(\bar \eps^{2^k}), \quad \ \ | \, \partial^2_{z_ix}s_k| = O(\bar \eps^{2^k}), \quad \ \ | \, \partial^2_{z_iz_j} s_k| = O(\bar \eps^{2^k}) $ \ssk
 \item
$  \| \, \Br_k(y) \| = O(\bar \eps^{2^k}), \quad \ \; \| \, \Br'_k(y) \| = O(\bar \eps^{2^k}), \quad \ \ \| \, \Br_k''(y) \| = O(\bar \eps ^{2^k}) $ \msk
\end{enumerate}
for $ 1 \leq i,j \leq m $.
\end{lem}

\begin{proof}
The map $\Psi_k $ has the two expressions, $D_k \circ (\id+ {\bf s}_k)(w)$ and $T_k \circ H^{-1}_k \circ \La_k \circ T_{k+1}^{-1}$, that is,
\begin{align*} 
\Psi_k &= D_k \circ (\id+ {\bf s}_k)(w) \\
&=T_k \circ H^{-1}_k \circ \La_k^{-1} \circ T_{k+1}^{-1}(w) = H^{-1}_k \circ \La_k^{-1} (w+ \tau_{k+1}) - \tau_k
\end{align*}
Recall that \msk
\begin{equation*}
\begin{aligned}
& \quad \ D_k \circ (\id+ {\bf s}_k)(w) \\[0.3em]
&=
 \left(    \begin{array} {cccccc}
\alpha_k & \si_k t_k   & \si_k u_k^1 & \cdots & \si_k u_k^m  &   \\[0.2em]
         & \si_k       &             &        &              &   \\[0.2em] \cline{3-5}
         & \multicolumn{1}{c|}{\si_k d_k^1} & \multicolumn{3}{c}{\multirow{3}{*}{ $ \si_k \cdot \Id_{m \times m}  $ }} & \multicolumn{1}{|c}{}\\
         & \multicolumn{1}{c|}{\vdots}      &  &  &  & \multicolumn{1}{|c}{} \\ 
         & \multicolumn{1}{c|}{\si_k d_k^m} &  &  &  & \multicolumn{1}{|c}{} \\       \cline{3-5}
    \end{array}       \right) 
\begin{pmatrix}
x + s(w) \\
y \\[0.2em]
z_1 + r_k^1(y) \\
\vdots \\
z_m + r_k^m(y)
\end{pmatrix} = 
\begin{pmatrix}
\alpha_k & \si_k \;\! t_k & \si_k \;\! \Bu_k     \\
         & \si_k          &      \\
         & \si_k \;\! \Bd_k    & \si_k
\end{pmatrix}
\begin{pmatrix}
x + s(w) \\
y \\
\Bz + \Br_k(y) 
\end{pmatrix}
\end{aligned} \bsk
\end{equation*}
In order to obtain the asymptotic of the non-linear part of $\Psi_k $, we need to compare the first and $ z_i $ coordinates for each $ 1 \leq i \leq m $ of above two expressions of $\Psi_k $. Let $ \tau_k = (\tau_k^x, \tau_k^y, \tau_k^{\,z_1}, \tau_k^{z_2} , \ldots , \tau_k^{z_m}) \equiv (\tau_k^x, \tau_k^y, \tau_k^{\Bz}) $ for each $ k \geq 1 $. Firstly, let us compare the $ z_i $ coordinates of two expression of $\Psi_k $. 
\begin{align*}
\si_k (d_k^{\,i}\;\! y + z_i+ r_k^i(y)) 
&= \pi_{z_i} \big(H^{-1}_k \circ \La_k^{-1} (w+ \tau_{k+1}) - \tau_k \big) \\
&= \si_k(z_i + \tau_{k+1}^{z_i}) + \de_k^{\,i} \big(\, \si_k (y+\tau_{k+1}^y), \ f^{-1}_k (\si_k (y+\tau_{k+1}^y)), \ \B0 \big) - \tau_k^{z_i} 
\end{align*}
Thus we have the following equation
\[ \si_k  r_k^i(y) = - \si_k d_k^{\,i} \:\! y + \de_k^{\,i} \big( \si_k (y+\tau_{k+1}^y), \; f^{-1}_k (\si_k (y+\tau_{k+1}^y)), \; \B0 \big) +\si_k \tau_{k+1}^{z_i} - \tau_k^{z_i} .
\]
Then $ | \:\! r_k^i(y)| \leq C \big( \,| \:\! d_k^{\,i} y | + \| \:\! \de_k^i \|_{\,C^0} \big)$ for some $C>0$ and for all $ 1 \leq i \leq m $. The domain is bounded and $ \| \:\! \bde_k \| $ is $ O(\bar \eps ^{2^k}) $. Then we obtain $ \| \;\! \Br_k(y) \| = O(\bar \eps ^{2^k}) $. Moreover, 
\[ (r_k^i)'(y) = -d_k^{\,i} + \frac{d}{dy}\,\de_k^{\,i} \big( \si_k (y+\tau_{k+1}^y), \; f^{-1}_k (\si_k (y+\tau_{k+1}^y), \; \B0 \big)
\]
Thus $ | \;\! (r_k^{i})'(y)| $ is controlled by $ \| \bde \|_{C^1} $ for all $ 1 \leq i \leq m $. Similarly, the second derivative $ | \;\! (r_k^i)''(y)| $ is also controlled by $ \| \bde \|_{C^2} $ for all $ 1 \leq i \leq m $. Then $ \| \:\! \Br_k'(y) \| = O(\bar \eps ^{2^k}) $ and $ \| \;\! \Br_k''(y)\| = O(\bar \eps ^{2^k}) $.
\ssk \\
Secondly, compare first coordinates using \eqref{decomposition} and \eqref{Psi on k level}. Thus
\begin{align}  \label{comparison}
\alpha_k x + \alpha_k \cdot s_k(w) + \si_k t_k \:\! y + \si_k \Bu_k \cdot ( \Bz + \Br_k(y)) = \phi_k^{-1}(\si_k w + \si_k \tau_{k+1}) - \pi_x(\tau_k).
\end{align}
\ssk
It implies the following equations
\begin{equation} \label{bounds of sk}
\begin{aligned}
\alpha_k \cdot \di_x s_k &= \si_k \cdot \di_x \phi^{-1}_k - \alpha_k  \\
\alpha_k \cdot \di_y s_k &= \si_k \cdot \di_y \phi^{-1}_k - \si_k t_k -\si_k \cdot \Bu_k \cdot \Br'_k(y)  \\
\alpha_k \cdot \di_{z_i} s_k &= \si_k \cdot \di_{z_i} \phi^{-1}_k - \si_k u_k^i
\end{aligned} \msk
\end{equation}
for $ 1 \leq i \leq m $. 
Then by the equation \eqref{partial derivatives of phi-1} \ssk , $ \| \;\!\di_x \phi^{-1}_k \| = O(1)$, $ \| \;\! \di_y \phi^{-1}_k \| = O \big(\bar \eps ^{2^k} \big) $ and $ \| \;\! \di_{z_i} \phi^{-1}_k \| = O\big(\bar \eps ^{2^k}\big)$ for all $ 1 \leq i \leq m $. By the equation \eqref{bounds at 0}, $ | \;\!t_k | $ and $ \| \:\! \Bu_k \| $ is $ O\big(\bar \eps ^{2^k}\big) $. Hence, $ \| \:\! \partial_xs_k \| \, = O(1) $,  $  \| \:\! \partial_y s_k \| \, = O \big(\bar \eps^{2^k} \big)$ and $ \| \:\! \partial_{z_i} s_k \| \; = O \big(\bar \eps^{2^k} \big) $ for all $ 1 \leq i \leq m $. By the above equation \eqref{bounds of sk}, each second partial derivatives of $s_k$ are comparable with the second partial derivatives of $\phi^{-1}_k$ over the same variables because $  \| \;\! \Br_k''(y) \| = O \big(\bar \eps ^{2^k} \big) $. 
\ssk \\
Let us estimate some second partial derivatives of $ \phi^{-1}_k $. Recall that 
\begin{equation*}
\begin{aligned}
\phi^{-1}_k(w) &= \ f^{-1}_k(x + \eps_k \circ H^{-1}_k(w)) \\[0.2em]
\eps_k \circ H^{-1}_k(w) &= \ \eps_k(\phi^{-1}_k(w),\,y,\,\Bz + \bde_k (y,f^{-1}(y),\B0)) .
\end{aligned}
\end{equation*}
Thus
\begin{equation*}
\begin{aligned}
\di_x \phi^{-1}_k(w) &= \ (f^{-1}_k)'(x + \eps_k \circ H^{-1}_k(w)) \cdot \big[\, 1 + \di_x (\eps_k \circ H^{-1}_k(w))\,\big] \\[0.4em]
\di_x( \eps_k \circ H^{-1}_k(w)) &= \ \di_x \eps_k \circ H^{-1}_k(w) \cdot \di_x \phi^{-1}_k(w) \\[0.4em]
\di_{xx}( \eps_k \circ H^{-1}_k(w)) &= \ \di_x( \eps_k \circ H^{-1}_k(w)) \cdot \di_x \phi^{-1}_k(w) + \di_x \eps_k \circ H^{-1}_k(w) \cdot \di_{xx} \phi^{-1}_k(w) \\
&= \ \di_x \eps_k \circ H^{-1}_k(w) \cdot \big[\,\di_x \phi^{-1}_k(w)\,\big]^2 + \di_x \eps_k \circ H^{-1}_k(w) \cdot \di_{xx} \phi^{-1}_k(w) .
\end{aligned} \msk
\end{equation*}
\nin Moreover, $ \| \:\! \eps_k \|_{C^2} $ and $ \| \:\! \bde_k \|_{C^2} $ bounds the norm of every second derivatives of $ \| \:\! \phi^{-1}_k \| $ except $ \| \:\! \di_{xx} \phi^{-1}_k \| $. Let us estimate $ \di_{xx} \phi^{-1}_k(w) $
\msk 
\begin{equation*}
\begin{aligned}
\di_{xx} \phi^{-1}_k(w) &= \ (f^{-1}_k)''(x + \eps_k \circ H^{-1}_k(w)) \cdot \big[\, 1 + \di_x (\eps_k \circ H^{-1}_k(w))\,\big] \\
& \quad \ \ + (f^{-1}_k)'(x + \eps_k \circ H^{-1}_k(w)) \cdot \di_{xx} (\eps_k \circ H^{-1}_k(w)) .
\end{aligned} \ssk
\end{equation*}
\nin Recall that $ \| \:\! \eps_k \|_{C^2} $ and $ \| \:\! \bde_k \|_{C^2} $ are $ O \big(\bar \eps^{2^k} \big) $. Since both $ \| (f^{-1})' \| $ and $ \| (f^{-1})'' \| $ are $ O(1) $, so is $ \| \:\! \di_{xx}\phi^{-1}_k \| $. Any other second derivative of $ \| \:\! \phi^{-1}_k \| $ is bounded by $ O\big( \bar \eps^{2^k} \big) $. For example, the following expression of $ \di_{yx}\phi^{-1}_k $ \msk
\begin{equation*}
\begin{aligned}
&\quad \ \di_{yx}\phi^{-1}_k(w) \\
&= \di_{xx}\phi^{-1}_k(w) \cdot \left[\, \di_y \eps_k \circ H^{-1}_k(w) + \sum_{j=1}^m \di_{z_j} \eps_k \circ H^{-1}_k(w) \cdot \frac{d}{dy}\,\de_k^j(y,f^{-1}_k(y),\B0)\,\right] \\
& \quad \ \ + \di_x \phi^{-1}_k(w) \cdot \left[\, \di_x (\di_y \eps_k \circ H^{-1}_k(w) ) + \di_x \left(\sum_{j=1}^m \di_{z_j} \eps_k \circ H^{-1}_k(w) \right) \cdot \frac{d}{dy}\,\de_k(y,f^{-1}_k(y),\B0)\,\right]
\end{aligned}
\end{equation*}
implies that $ \| \;\! \di_{yx} \phi^{-1}_k \| $ is bounded by $ O\big( \bar \eps^{2^k} \big) $. The norm estimation of other second partial derivatives of $ \phi^{-1}_k $ is left to the reader.
\end{proof}
\msk

\subsection{Estimation of non linear part of $S^n_k$}
We consider the behavior of the 
non linear scaling map from $k^{th}$ level to $n^{th}$ level. 
Let
$$
\Psi^n_k = \Psi_k \circ \cdots \circ \Psi_{n-1}, \quad B^n_k = \Im \Psi^n_k
$$
By Lemma \ref{diameter}, 
$$ \diam(B^n_k) = O(\si^{n-k}) \qquad \textrm{ for} \quad k<n
$$
Then combining Lemma \ref{diameter} and Lemma \ref{bounds on domain}, we have the following corollary.

\begin{cor}  \label {bounds at k level}
 For all points $w = (x,y,\Bz) \in B^n_{k}$ and where $k<n$, we have
\begin{align*}
& |\; \! \di_x s_k(w) | = O(\si^{n-k}) \qquad |\; \! \di_y s_k(w) | = O  \big(\bar \eps^{2^k} \si^{n-k} \big)
 \qquad |\; \! \di_z s_k(w) | = O \big( \bar \eps^{2^k} \si^{n-k} \big) \\
 & \|\; \! \Br'_k(y) \|=O \big(\bar \eps^{2^k} \si^{n-k} \big) \qquad \quad \|\; \! \Br''_k(y)\|=O \big(\bar \eps^{2^k} \si^{n-k} \big)
\end{align*}
\end{cor}
\begin{proof}
By definition, $ s_k(w) $ is quadratic and higher order terms at the tip, $ \tau_k $. Similarly, $ \Br'_k(y) $ only contains quadratic and higher order terms at the tip. Then use Taylor's expansion and upper bounds of $ \diam(B^n_k) $ is $ O(\si^{n-k}) $.
\end{proof}
\msk
\nin Since the origin is the fixed point of each $ \Psi_j $ and $ D_j $ is $ \Psi_j(0) $ for every $ k \leq j \leq n $, we can let the derivative of $\Psi^n_k$ at the origin be the composition of consecutive $ D_i $s for $ k \leq i \leq n-1 $.
\begin{equation*}
 D^n_k = D_k \circ D_{k+1} \circ \cdots \circ D_{n-1} 
\end{equation*}
We can decompose $D^n_k$ into two matrices, the matrix whose diagonal entries are ones and the diagonal matrix by reshuffling. 
\begin{rem}
The notations $t_{n+1,\,n}, \Bu_{n+1,\,n}$ and $ \Bd_{n+1,\,n} $ are simplified as $t_n, \Bu_n$ and $ \Bd_n$ like the notations used in \eqref{decomposition}. Moreover, $ \alpha_{n+1,\,n}, \si_{n+1,\,n}$ are abbreviated as $\alpha_n, \si_n$ respectively. Thus  $\alpha_n = \si^2(1+ O(\rho^n)),\ \si_n = -\si (1+ O(\rho^n))$. Using the similar abbreviation, $D_n$ denote $D^{n+1}_n$ and $s_n$ is the $s^{n+1}_n$.
\end{rem}
\begin{lem} \label{decomposition of derivative}
The derivative of $ \Psi^n_k $ at the origin, $ D^n_k $ is decomposed into the dilation and non dilation parts as follows
\begin{align*}
D_k^n=
\left ( \begin{array} {ccc}
1 & t_{n,\,k} & \Bu_{n,\, k} \\[0.2em]
   & 1          & \\
   & \Bd_{n,\, k}& 1 
\end{array} \right )
\left ( \begin{array} {ccc}
\alpha_{n,\, k} &                &  \\
                & \si_{n,\,k}    &  \\
                &                & \si_{n,\,k} \cdot \Id_{m \times m}
\end{array} \right ) .
\end{align*}
\msk \\
Moreover, $\alpha_{n,\,k} = (\si^2)^{n-k}(1+O(\rho^k))$ and $\si_{n,\,k}= (-\si)^{n-k}(1+O(\rho^k)) $ for some $\rho \in (0,1)$. Each $ t_{n,\,k} $, $ \Bu_{n,\,k} $ and $ \Bd_{n,\,k} $ are comparable with the $ t_{k+1,\,k} $, $ \Bu_{k+1,\,k} $ and $ \Bd_{k+1,\,k} $ respectively and converges to the numbers $ t_{*,\,k} $, $ \Bu_{*,\,k} $ and $ \Bd_{*,\,k} $ super exponentially fast as $ n \ra \infty $.
\end{lem}
\begin{proof}
Using the definition of each derivatives of $ \Psi_j $ on the equation \eqref{decomposition} at the fixed point zero, we obtain 
\begin{align*}
D^n_{k} = \prod^{n-1}_{j=k} D_j = \prod^{n-1}_{j=k}
\begin{pmatrix}
\alpha_j & \,\si_j \,t_j  & \,\si_j\,\Bu_j \\
              & \si_j &  \\
              & \,\si_j \,\Bd_j &\si_j \cdot \Id_{m \times m}
\end{pmatrix} .
\end{align*}
By the straightforward calculation, we have following matrix,
\begin{align}
D^n_k = 
\begin{pmatrix}
\ \ \displaystyle\prod^{n-1}_{j=k} \alpha_j & T_{n,\,k} & \quad {\bf U}_{n,\,k} \phantom{**} \\
 & \displaystyle\prod^{n-1}_{j=k} \si_j & \\[2em]
 & \displaystyle\prod^{n-1}_{j=k} \si_j \sum^{n-1}_{j=k} \Bd_j &\quad \displaystyle\prod^{n-1}_{j=k} \si_j  \cdot \Id_{m \times m} \phantom{*}
\end{pmatrix}
\end{align}
where 
\begin{align*}
{\bf U}_{n,\,k} \ &= \ \si_k \,\si_{k+1} \,\si_{k+2} \cdots \si_{n-2} \,\si_{n-1}\, \Bu_k \\[0.2em]
& \qquad + {\color{blue} \alpha_k}\, \si_{k+1}\, \si_{k+2} \cdots \si_{n-2} \,\si_{n-1} \,\Bu_{k+1} \\[0.2em]
& \qquad + {\color{blue} \alpha_k \,\alpha_{k+1}}\, \si_{k+2} \cdots \si_{n-2} \,\si_{n-1}\, \Bu_{k+2} \\
& \hspace{1in} \vdots \\
& \qquad + {\color{blue} \alpha_k \,\alpha_{k+1}\, \alpha_{k+2}} \cdots {\color{blue} \alpha_{n-2}}\, \si_{n-1}\, \Bu_{n-1} \\[0.8em]
T_{n,\,k} \ &= \ \si_k \,\si_{k+1} \,\si_{k+2} \cdots \si_{n-3}\, \si_{n-2} \,\si_{n-1}\, \big[\ \Bu_k \cdot (\Bd_{k+1} + \Bd_{k+2} + \Bd_{k+3} + \cdots + \Bd_{n-1} ) + t_k \big] \\[0.2em]
& \quad \ +\ {\color{blue} \alpha_k} \, \si_{k+1}\, \si_{k+2} \cdots \si_{n-3}\, \si_{n-2} \,\si_{n-1} \,\big[\, \Bu_{k+1} \cdot( \qquad \quad \Bd_{k+2} + \Bd_{k+3} + \cdots + \Bd_{n-1} ) + t_{k+1} \,\big] \\[0.2em]
& \quad \ +\ {\color{blue} \alpha_k \,\alpha_{k+1}} \, \si_{k+2} \cdots \si_{n-3}\, \si_{n-2} \,\si_{n-1}\,\big[\, \Bu_{k+2} \cdot( \qquad \qquad \qquad \; \Bd_{k+3} + \cdots + \Bd_{n-1} ) + t_{k+2}\, \big]\\
& \hspace{2in} \vdots \\
& \quad \ +\ {\color{blue} \alpha_k \,\alpha_{k+1}\, \alpha_{k+2}} \cdots {\color{blue} \alpha_{n-3}} \,\si_{n-2}\, \si_{n-1}\,\big[\, \Bu_{n-2} \cdot \Bd_{n-1} + t_{n-2} \,\big] \\[0.2em]
& \quad \ +\ {\color{blue} \alpha_k \,\alpha_{k+1}\, \alpha_{k+2}} \cdots {\color{blue} \alpha_{n-3} \,\alpha_{n-2}}\, \si_{n-1}\cdot t_{n-1}.
\end{align*}
Moreover,
\begin{equation} \label{scaling from kth to nth level}
\begin{aligned}
\si_{n,\,k} &= \prod^{n-1}_{j=k} \si_j = \prod^{n-1}_{j=k} (-\si) (1+ O(\rho^j)) = (-\si)^{n-k}  (1+ O(\rho^k))& \\
\alpha_{n,\,k} &= \prod^{n-1}_{j=k} \alpha_j = \prod^{n-1}_{j=k} \si^2 (1+ O(\rho^j)) = \si^{2(n-k)}  (1+ O(\rho^k))& .
\end{aligned}
\end{equation}
By the definition of $ \Bd_{n,\,k} $ and \eqref{scaling from kth to nth level}, each components of the diffeomorphic part and the scaling part are separated
\begin{equation} \label{sum of d, u and t respectively}
\begin{aligned}
\Bd_{n ,\: k} =& \sum^{n-1}_{j=k} \Bd_j \\
\Bu_{n,\: k} =&  \sum_{j=k}^{n-2} (-\si)^{j-k} \Bu_j \, (1+ O(\rho^k)) \\
t_{n ,\: k} =&  \sum_{j=k}^{n-1} (-\si)^{j-k} \left[\, \Bu_j \cdot \sum_{i=j}^{n-2} \Bd_{i+1} + t_j + t_{n-1} \,\right] (1+ O(\rho^k)) .
\end{aligned}
\end{equation}
Since $ \| \:\!\Bd_j \| = O\big(\bar \eps^{2^j} \big) $, $ \| \:\! \Bu_j \| = O\big(\bar \eps^{2^j} \big) $ and $ |\:\! t_j | = O\big(\bar \eps^{2^j} \big) $ for each $ j \in \N $, each terms of the \\
series 
in \eqref{sum of d, u and t respectively} shrink super exponentially fast. Then the sum $ \Bd_{n,\,k} $, $ \Bu_{n,\,k} $ and $ t_{n,\,k} $ are comparable with the first terms of each series. Moreover, $ \Bd_{n,\,k} $, $ \Bu_{n,\,k} $ and $ t_{n,\,k} $ converge to $ \Bd_{*,\,k} $, $ \Bu_{*,\,k} $ and $ t_{*,\,k} $ as $ n \ra \infty $ super exponentially fast respectively.
\end{proof}
\msk
\nin After reshuffling of $\Psi^n_k$, we can factor out $ D^n_k $ from the map $\Psi^n_k$. Then we have
\begin{align} \label{coor change from k to n}
 \Psi_k^n=D_k^n \circ (\id + {\bf S}^n_k)
\end{align}
where  ${\bf S} ^n_k=(S_k^n(w),\ 0,\ {\bf R}_{n,\,k}(y)) =O(| \: \! w|^{\,2})$ near the origin. \ssk When we calculate directly the composition, $H^{-1}_k \circ \La^{-1}_k \circ \cdots \circ H^{-1}_{n-1} \circ \La^{-1}_{n-1}$. Observe that the map 
$$ {\bf R}_{n,\,k} = ( R^1_{n,\,k},\,  R^2_{n,\,k},\,\ldots,  R^m_{n,\,k}) $$ 
depends only on $ y $.

\msk
\begin{prop} \label{bounds of R}
The third coordinate of\, ${\bf S} ^n_k  $, $ {\bf R}_{n,\,k}(y) $ has the following norm estimations.
\begin{align*}
\|{\bf R}_{n,\,k}\| = O \big(\bar \eps^{2^k} \big), \quad
\| \: \!({\bf R}_{n,\,k})'\| = O\big(\bar \eps^{2^k} \si^{n-k} \big) \quad \text{and} \quad  \ 
\| \: \!({\bf R}_{n,\,k})''\|= O\big(\bar \eps^{2^k} \si^{2(n-k)} \big) 
\end{align*}
for all $ k<n $.
\end{prop}

\begin{proof}
The proof comes from the recursive formula between each partial derivatives of $S^n_k$ and $S^n_{k+1}$. So before proving this lemma we need some intermediate calculations. 
For a point $w=(x,y,\Bz) \in B $, let 
$$
 w_{k+1}^n =
\left ( \begin{array} {c}
x_{k+1}^n\\ [0.3em]
y_{k+1}^n\\ [0.3em]
\Bz_{k+1}^n
\end{array} \right )
= \Psi_{k+1}^n(w) \in B^n_{k+1} 
$$
By the equation \eqref{coor change from k to n}, we have
\msk
\begin{equation*}
\begin{aligned}
\left ( \begin{array} {c}
x_{k+1}^n\\ [0.3em]
y_{k+1}^n\\ [0.3em]
\Bz_{k+1}^n
\end{array} \right )=
\left ( \begin{array} {c l l}
\alpha_{n,\,k+1}     & \si_{n,\,k+1}  \cdot t_{n,\,k+1}  & \si_{n,\,k+1} \cdot \Bu_{n,\,k+1} \\ [0.3em]
                     & \si_{n,\,k+1}              &                 \\ [0.3em]
                     & \si_{n,\,k+1} \cdot \Bd_{n,\,k+1}  &  \si_{n,\,k+1} \cdot \Id_{m \times m}
\end{array} \right )
\left ( \begin{array} {c}
x+ S _{k+1}^n(w) \\ [0.3em]
y \\ [0.3em]
\Bz+{\bf R}_{n,\,k+1}(y)
\end{array} \right )  .            
\end{aligned} \bsk
\end{equation*}
Then each coordinate of $w^n_{k+1}$ is
\msk
\begin{equation}  \label{the image of the Psi from nth level to k+1th level}
\begin{aligned} 
x_{k+1}^n =& \ \alpha_{n,\,k+1}(x+S_{k+1}^n(w))+\si_{n,\,k+1}\, t_{n,\, k+1} \cdot y + \si_{n,\,k+1} \Bu_{n,\,k+1} \cdot ( \Bz+ {\bf R}_{n,\,k+1}(y)) \\[0.2em]
y_{k+1}^n =& \ \si_{n,\,k+1} \cdot y \\[0.2em]
\Bz_{k+1}^n =& \ \si_{n,\,k+1}\, \Bd_{n,\,k+1} \cdot y+ \si_{n,\,k+1}\,(\Bz+ {\bf R}_{n,\,k+1}(y)) . 
\end{aligned} \msk
\end{equation}
For any fixed $k<n$, the recursive formula of $ \Psi^n_k$ is
\msk
\begin{equation}      \label{recursive form}
\begin{aligned} 
D_k^n \circ (\id + {\bf S}_k^n) &= \Psi _k^n = \Psi _k \circ \Psi_{k+1}^n = D_k \circ (\id + {\bf s}_k) \circ \Psi _{k+1}^n  \\[0.2em]
 &= D_k^n \circ (\id + {\bf S}_{k+1}^n) + D_k \circ {\bf s}_k \circ \Psi_{k+1}^n  \\[0.2em]
\text {Thus} \qquad  \Psi_k^n(w) &= D_k^n \circ (\id + {\bf S}_{k+1}^n)(w) + D_k \circ {\bf s}_k (w_{k+1}^n)
\end{aligned}
\end{equation}
and note that
\begin{equation*}
 D_k \circ {\bf s}_k (w_{k+1}^n) =
\left (
       \begin{array} {crr}
    \alpha_k  &  \si_k t_k  &  \si_k \,\Bu_k  \\[0.2em]
              &  \si_k &         \\[0.2em]
              &  \si_k \,\Bd_k  & \si_k \cdot \Id_{m \times m}
       \end{array}
\right )
\left (
       \begin{array}{c}
s_k(w_{k+1}^n) \\[0.3em]
0 \\[0.1em]
\Br_k (y_{k+1}^n)
        \end{array}
\right ) . 
\end{equation*}
\ssk \\
Moreover, the first partial derivatives of each coordinate are as follows
\msk
\begin{equation}  \label{1st partial}
\begin{aligned}
\frac {\partial x_{k+1}^n}{\partial x} &= \alpha_{n,\,k+1} \left (1+ \frac{\partial S_{k+1}^n}{\partial x}(w) \right) \\[0.3em]
\frac {\partial x_{k+1}^n}{\partial y} &= \alpha_{n,\,k+1} \frac{\partial S_{k+1}^n}{\partial y}(w) + \si_{n,\,k+1}\, t_{n,\,k+1} +\si_{n,\,k+1} \,\Bu_{n,\,k+1} \cdot ({\bf R}_{n,\,k+1})'(y) \\[0.3em]
\frac {\partial x_{k+1}^n}{\partial z_i} &= \alpha_{n,\,k+1} \frac{\partial S_{k+1}^n}{\partial z_i}(w) + \si_{n,\,k+1}\, u_{n,\,k+1}^i \\[0.3em]
\frac {\partial y_{k+1}^n}{\partial y} &= \frac {\partial \Bz_{k+1}^n}{\partial z_i} = \si_{n,\,k+1} \\[0.3em]
\frac {\partial \Bz_{k+1}^n}{\partial y} &= \si_{n,\,k+1}\, \Bd_{n,\,k+1}+ \si_{n,\,k+1}\cdot ({\bf R}_{n,\,k+1})'(y) \\[0.3em]
\frac {\partial y_{k+1}^n}{\partial x} &= \frac {\partial y_{k+1}^n}{\partial z_i} = \frac {\partial z_{k+1}^n}{\partial x} = 0 
\end{aligned} \msk
\end{equation}
for every $ 1 \leq i \leq m $. In order to estimate of ${\bf R}_{n,\,k}(y)$, compare the third coordinates of the functions in \eqref{recursive form}. Recall $\si^{-1} = \la $. Then
\begin{equation*}
\begin{aligned}
\Bz_{k}^n =& \ \si_{n,\,k}\, \Bd_{n,\,k} \cdot y+ \si_{n,\,k}(\Bz+{\bf R}_{n,\,k}(y)) \\[0.2em]
=& \ \si_{n,\,k} \,\Bd_{n,\,k} \cdot y+ \si_{n,\,k}(\Bz+ {\bf R}_{n,\,k+1}(y)) +\si_k \cdot \Br_k (y_{k+1}^n) 
\end{aligned} 
\end{equation*}
Then
\begin{equation*}
\begin{aligned}
 {\bf R}_{n,\,k}(y)  =& \ {\bf R}_{n,\,k+1}(y) +\si_{n,\,k}^{-1} \cdot \si_k \cdot \Br_k(y_{k+1}^n)   
\end{aligned} \msk
\end{equation*}
where $ \si_{n,\,k}^{-1} \cdot \si_k $ is $ (-\lambda)^{n-k-1}(1+O(\rho^k)) $. \ssk
By the equation \eqref{1st partial}, the recursive relation between ${\bf R}^n_k(y)$, $ {\bf R}_{k+1}^n(y) $ and the bounds of $ \Br_k(y_{k+1}^n) $, we obtain the following formulas
\msk
\begin{equation*}
\begin{aligned}
 {\bf R}_{n,\,k}(y) &= {\bf R}_{n,\,k+1}(y) + O\big( (-\lambda)^{n-k-1} \Br_k(y_{k+1}^n) \big)   \\[0.2em]
 ({\bf R}_{n,\,k})'(y) &= ({\bf R}_{n,\,k+1})'(y) + O\big( \Br_k'(y_{k+1}^n) \big)         \\[0.2em]
\textrm{and}  \quad ({\bf R}_{n,\,k})''(y) &= ({\bf R}_{n,\,k+1})''(y) + O\big( \si^{n-k} \cdot  \Br_k''(y_{k+1}^n) \big) .
\end{aligned} \msk
\end{equation*}
Hence, by the equation \eqref{the image of the Psi from nth level to k+1th level} and the chain rule
\begin{equation*}
\begin{aligned}
\|{\bf R}_{n,\,k}\| &\leq \|{\bf R}_{n,\,k+1}\| + K_0 \bar \eps^{2^k}   \\ 
\| \: \! ({\bf R}_{n,\,k})'\| &\leq  \|({\bf R}_{n,\,k+1})'\| + K_1 \bar \eps^{2^k}\si^{n-k} \\
\| \: \!({\bf R}_{n,\,k})''\| &\leq  \|({\bf R}_{n,\,k+1})''\| + K_2 \bar \eps^{2^k} \si^{2(n-k)}
\end{aligned} \msk
\end{equation*}
 for all $ k<n  $. Then, 
\begin{align*}
\|{\bf R}_{n,\,k}\| = O(\bar \eps^{2^k}), \ \
\| \: \!({\bf R}_{n,\,k})'\| = O(\bar \eps^{2^k} \si^{n-k} ) \ \ \text{and} \ \ 
\|({\bf R}_{n,\,k})''\|= O(\bar \eps^{2^k} \si^{2(n-k)}) 
\end{align*}
 for all $ k<n  $.
\end{proof}

\msk
\begin{lem} \label{asymptotics of non-linear part 1}
For $k<n$, we have \msk
\begin{enumerate}
\item $ | \; \! \partial_x S^n_k| \; = O(1), \qquad  \qquad | \; \! \partial_y S^n_k|  \ = O(\bar \eps^{2^k}), \qquad \qquad \ | \; \! \partial_{z_i} S^n_k| \ = O(\bar \eps^{2^k})  $ \ssk
\item $ | \; \! \partial^2_{xy} S^n_k| = O(\bar \eps^{2^k} \si^{n-k}), \quad \; | \; \! \partial^2_{xz_i} S^n_k| = O(\bar \eps^{2^k} \si^{n-k})$ \ssk
\item  $| \; \! \partial^2_{yz_i}S^n_k| = O(\bar \eps^{2^k} ),  \qquad \quad \ | \; \! \partial^2_{z_i z_j} S^n_k| = O(\bar \eps^{2^k} ) $ . \ssk
\end{enumerate}
for every $ 1 \leq i,j \leq m $.
\end{lem}
\ssk
\begin{proof}
Compare the first coordinates of $\Psi^n_k$ in \eqref{recursive form}. Thus 
\begin{align*}
x^n_k &= \alpha_{n,\,k}(x+S_k^n(w)) + \si_{n,\,k}\,t_{n,\,k} \cdot y + \si_{n,\,k}\,\Bu_{n,\,k} \cdot \big(\,\Bz+ {\bf R}_{n,\,k}(y) \big)  \\
      &= \alpha_{n,\,k}(x+S_{k+1}^n(w)) + \si_{n,\,k}\,t_{n,\,k} \cdot y + \si_{n,\,k}\,\Bu_{n,\,k} \cdot  \big(\,\Bz+ {\bf R}_{n,\,k+1}(y) \big) + \alpha_k \cdot s_k(w_{k+1}^n) \\
      & \quad \ \ + \si_k \,\Bu_k \cdot \Br_k(y_{k+1}^n) .
\end{align*}
Then we obtain the recursive formula for $S^n_k$ as follows
\begin{equation*}
\begin{aligned}
S_k^n(w) &=  S_{k+1}^n(w)+\alpha_{n,\,k}^{-1}\, \alpha_k \cdot s_k(w_{k+1}^n) +\alpha_{n,\,k}^{-1}\, \si_{n,\,k} \, \Bu_{n,\,k} \cdot \big({\bf R}_{n,\,k+1}(y)- {\bf R}_{n,\,k}(y) \big) \\
 & \quad \ \ +\alpha_{n,\,k}^{-1} \,\si_k \, \Bu_k \cdot \Br_k(y^n_{k+1}) .
\end{aligned} \ssk
\end{equation*}
\nin Let us take the first partial derivatives of each side of above equation and use \eqref{1st partial}. Then we can have the recursive formulas of each first partial derivatives of $S_k^n(w)  $. Let us take the coordinate expression of $ w^n_{k+1} $ as $ (x^n_{k+1}, y^n_{k+1}, (z^n_{k+1})_1, (z^n_{k+1})_2, \ldots , (z^n_{k+1})_m ) $. Then \msk
\begin{align*}
\frac{\di S^n_k}{\di x} = \ &   \frac{\di S^n_{k+1}}{\di x}\left(1+  \frac{\di s_k}{\di x^n_{k+1}}\right)+  \frac{\di s_k}{\di x^n_{k+1}}\\[0.6em]
\frac{\partial S_k^n}{\partial y} 
= \ & \left (1 + \frac{\partial s_k}{\partial x^n_{k+1}} \right ) \frac{\partial S_{k+1}^n}{\partial y} + 
K_1 \la^{n-k-1} \left[ \Big (t_{n,\,k+1} + \Bu_{n,\,k+1} \cdot ({\bf R}_{n,\,k+1})'(y) \Big )  \frac{\partial s_k}{\partial x^n_{k+1}}  \right. \\[0.2em]
& \quad  \left.  + \frac{\partial s_k}{\partial y^n_{k+1}} +  \sum_{j=1}^m  \Big (d_{n,\,k+1}^j + ({\bf R}_{n,\,k+1}^j)'(y) \Big ) \frac{\partial s_k}{\partial (z^n_{k+1})_j} \; \right]  \\[0.2em]
& \ + 
K_1\la^{n-k-1} \Bu_{n,\,k} \cdot \Big ( ({\bf R}_{n,\,k+1})'(y) - ({\bf R}_{n,\,k})'(y) \Big) +
 K_2 \la^{n-k} \Bu_k \cdot \Br_k'(y_n^{k+1}) \\[0.6em]
\frac{\partial S_k^n}{\partial z_i}  
= \ & \left( 1 + \frac{\partial s_k}{\partial x^n_{k+1}} \right) \frac{\partial S_{k+1}^n}{\partial z_i} + K_1 \la^{n-k-1} \left[ u_{n,\,k+1}^i  \frac{\partial s_k}{\partial x^n_{k+1}} +\frac{\partial s_k}{\partial (z^n_{k+1})_i} \; \right]
\end{align*} 
where \,$ \alpha_{n,\,k}^{-1} \;\! \alpha_k \;\! \si_{n,\,k+1} = K_1 (-\la)^{n-k-1} $\, and 
 $ \alpha_{n,\,k}^{-1} \;\! \si_{n,\,k+1} = K_2 (-\la)^{n-k+1} $ for each $ 1 \leq i \leq m $. 
By Corollary \ref{bounds at k level} and Proposition \ref{bounds of R}, $| \; \! \di s_k / \di x^n_{k+1} |$ is $ O(\si^{n-k}) $ and $| \; \! \di s_k / \di y^n_{k+1} | $ and $| \; \! \di s_k / \di (z^n_{k+1})_i | $ is $ O(\bar \eps^{2^k} \si^{n-k})$ for all $ 1 \leq i \leq m $. Moreover, $| \: \! t_{n,\,k} |,\, \| \: \! \Bu_{n,\,k} \|$ and $ \| \: \! \Bd_{n,\,k} \|$ are $ O(\bar \eps^{2^k}) $. With all these facts, the bounds of each partial derivatives of $ S^n_k $ are as follows \msk
\begin{equation*}
\begin{aligned}
\left| \frac{\di S^n_k}{\di x} \right|  \leq & ( 1+ O(\rho^{n-k}) ) \left| \frac{\di S^n_{k+1}}{\di x} \right|+  C \si^{n-k}  \\[0.2em]
\left |\dfrac{\partial S_k^n}{\partial y} \right | \leq & \big (1 + O(\rho^{n-k}) \big ) \left |\dfrac{\partial S_{k+1}^n}{\partial y} \right | + C \bar \eps^{2^k} \\[0.2em]
 \left |\dfrac{\partial S_k^n}{\partial z_i} \right | \leq & \big (1 + O(\rho^{n-k}) \big ) \left| \dfrac{\partial S_{k+1}^n}{\partial z_i} \right| + C \bar \eps^{2^k}
\end{aligned} \msk
\end{equation*}
for all $ 1 \leq i \leq m $, for some constant $ C>0 $ and $ \rho \in (0,1) $. 
Hence, using above recursive formulas we have
\msk
\begin{equation*}
\begin{aligned}
\left| \frac{\di S^n_k}{\di x} \right| =O(\si),  \quad \left | \dfrac{\partial S_k^n}{\partial y} \right | =  O( \bar \eps^{2^k}) \quad \text{and} \quad \left| \dfrac{\partial S_k^n}{\partial z} \right| = O( \bar \eps^{2^k})
\end{aligned} \msk
\end{equation*}
for all $ k<n $.
\comm{***********
\nin 
The second partial derivatives of $w^n_{k+1}$ are
\msk
\begin{equation}  \label{2nd partial}
\begin{aligned}
\frac {\partial^2 x_{k+1}^n}{\partial x^2} &= \alpha_{n,\,k+1}  \frac{\partial^2 S_{k+1}^n}{\partial x^2}(w), \qquad 
\frac {\partial^2 x_{k+1}^n}{\partial xy} = \alpha_{n,\,k+1}  \frac{\partial^2 S_{k+1}^n}{\partial xy}(w)  \\[0.3em]
\frac {\partial^2 x_{k+1}^n}{\partial xz_i} &= \alpha_{n,\,k+1}  \frac{\partial^2 S_{k+1}^n}{\partial xz_i}(w) \\[0.3em]
\frac {\partial^2 x_{k+1}^n}{\partial y^2} &= \alpha_{n,\,k+1}  \frac{\partial^2 S_{k+1}^n}{\partial y^2}(w) + \si_{n,\, k+1} \Bu_{n,\,k+1} \cdot ({\bf R}_{n,\,k+1})''(y) \\[0.3em]
\frac {\partial^2 x_{k+1}^n}{\partial yz_i} &= \alpha_{n,\,k+1}  \frac{\partial^2 S_{k+1}^n}{\partial yz_i}(w),  \qquad
\frac {\partial^2 \Bz_{k+1}^n}{\partial y^2} =  \si_{n,\, k+1}({\bf R}_{n,\,k+1})''(y)
\end{aligned} \msk
\end{equation}
for each $ 1 \leq i \leq m $ and other second order partial derivatives are identically zero. 
*************}
The second partial derivatives of $S^n_k$ are as follows by the chain rule 
\msk
\begin{equation*}
\begin{aligned}
\frac{\partial^2 S_k^n}{\partial xy}
&=  \left( 1+ \frac{\partial s_k}{\partial x^n_{k+1}} \right) \frac{\partial^2 S_{k+1}^n}{\partial xy} + \alpha_{n,\,k+1} \left( 1 + \frac{\partial S_{k+1}^n}{\partial x} \right)\frac{\partial^2 s_k }{\partial (x^n_{k+1})^2} \frac{\partial S_{k+1}^n}{\partial y}\\[0.3em]
 & \quad + \si_{n,\, k+1} \left( 1 + \frac{\partial S_{k+1}^n}{\partial x} \right)  \left[ \Big( t_{n,\,k+1} + \Bu_{n,\,k+1} \cdot ({\bf R}_{n,\,k+1})'(y) \Big) \frac{\partial^2 s_k }{\partial (x^n_{k+1})^2} + \frac{\partial^2 s_k }{\partial x^n_{k+1}y^n_{k+1}} \right.       \phantom{***}  
\\
  & \quad \left. + \sum_{j=1}^m \Big( d_{n,\,k+1}^j + ({\bf R}_{n,\,k+1}^j)'(y) \Big) \frac{\partial^2 s_k }{\partial x^n_{k+1}z^n_{k+1}} \; \right] 
\end{aligned}
\end{equation*}

\begin{align*}
 \frac{\partial^2 S_k^n}{\partial xz_i} 
& = \left( 1+ \frac{\partial s_k}{\partial x^n_{k+1}} \right) \frac{\partial^2 S_{k+1}^n}{\partial xz_i} + \alpha_{n,\,k+1} \left( 1 + \frac{\partial S_{k+1}^n}{\partial x} \right)\frac{\partial^2 s_k }{\partial (x^n_{k+1})^2} \frac{\partial S_{k+1}^n}{\partial z_i} \hspace{1.5in} \\[0.3em]
 & \quad + \si_{n,\, k+1} \left( 1 + \frac{\partial S_{k+1}^n}{\partial x} \right)  
 \cdot \left[  u_{n,\,k+1}^i \frac{\partial^2 s_k }{\partial (x^n_{k+1})^2} +  \frac{\partial^2 s_k }{\partial x^n_{k+1} (z^n_{k+1})_i} \; \right] 
\end{align*}

\begin{align*}
\frac{\partial^2 S_k^n}{\partial yz_i}  
& =  \left (1+\frac{\partial s_k}{\partial x^n_{k+1}} \right) \frac{\partial^2 S_{k+1}^n}{\partial yz_i} + \left[ \alpha_{n,\,k+1} \frac{\partial S_{k+1}^n}{\partial z_i} \frac{\partial S_{k+1}^n}{\partial y} + \si_{n,\, k+1} u_{n,\,k+1}^i \frac{\partial S_{k+1}^n}{\partial y} \right.  \\
   & \quad + \left. \si_{n,\, k+1} \Big( t_{n,\,k+1} + \Bu_{n,\,k+1} \cdot ({\bf R}_{n,\,k+1})'(y) \Big) \Big( \frac{\partial S_{k+1}^n}{\partial z_i} + K_1 (-\la)^{n-k-1} u_{n,\,k+1}^i \Big) \right]    \frac{\partial^2 s_k}{\partial (x^n_{k+1})^2}   \\
& \quad + \left( \si_{n,\, k+1} \frac{\partial S_{k+1}^n}{\partial z_i} + K_4 u_{n,\,k+1}^i \right) \frac{\partial^2 s_k}{\partial x^n_{k+1}y^n_{k+1}}  \\
& \quad + \left[\, \si_{n,\, k+1} \frac{\partial S_{k+1}^n}{\partial y} + K_4 \big(t_{n,\,k+1} + \Bu_{n,\,k+1} \cdot ({\bf R}_{n,\,k+1})'(y) \Big) \right] \frac{\partial^2 s_k}{\partial x^n_{k+1}(z^n_{k+1})_i} \\
   & \quad + \left( \si_{n,\, k+1}\frac{\partial S_{k+1}^n}{\partial z_i} + K_4 u_{n,\,k+1}^i \right)  \sum_{j=1}^m \Big( d_{n,\,k+1}^j + ({\bf R}_{n,\,k+1}^j)'(y) \Big) \cdot \frac{\partial^2 s_k}{\partial (z^n_{k+1})_j(z^n_{k+1})_i}
\end{align*}

\begin{align*}
\frac{\partial^2 S_k^n}{\partial z_i z_j}  
=  & \left (1+\frac{\partial s_k}{\partial x^n_{k+1}} \right) \frac{\partial^2 S_{k+1}^n}{\partial z_i z_j} 
 + \left[\, \alpha_{n,\,k+1} \frac{\partial S_{k+1}^n}{\partial z_i} \cdot \frac{\partial S_{k+1}^n}{\partial z_j} + \si_{n,\,k+1} u_{n,\,k+1}^j \frac{\di S^n_{k+1}}{\di z_i} \right. \phantom{********\;\,}  \\
& \quad + \left. \si_{n,\,k+1} \frac{\di S^n_{k+1}}{\di z_j} + K_4  (u_{n,\,k+1}^j)^2 \right] \frac{\di^2 s_k}{\di (x^n_{k+1})^2}  \\
& + \left[\, \si_{n,\,k+1} \frac{\di S^n_{k+1}}{\di z_i} + K_4 u_{n,\,k+1}^i \right] \frac{\di^2 s_k}{\di x^n_{k+1}(z^n_{k+1})_j} \\
& + \left[\, \si_{n,\,k+1} \frac{\di S^n_{k+1}}{\di z_j} + K_4 u_{n,\,k+1}^j \right] \frac{\di^2 s_k}{\di x^n_{k+1}(z^n_{k+1})_i} 
 + K_4 \frac{\di^2 s_k}{\di (z^n_{k+1})_i(z^n_{k+1})_j}
\end{align*}
where $ K_4 = \alpha_{n,\,k}^{-1} \, \alpha_k \, \si_{n,\,k+1}^2 = O(1)$ for every $ 1 \leq i,j \leq m $.
\ssk \\
By Lemma \ref{bounds of sk}, Corollary \ref{bounds at k level}, 
and Proposition \ref{bounds of R}, the bounds of $| \;\! \di^2 s_k / \di (x^n_{k+1})^2 | $ is $ O( \si^{n-k})$ and 
 $| \;\! \di^2 s_k / \di uv | $ is $ O(\bar \eps^{2^k} \si^{n-k})$ where $ u, v = x^n_{k+1},\, y^n_{k+1},\, (z^n_{k+1})_1,\, \ldots \,,(z^n_{k+1})_m $ but both $ u $ and $ v $ are not $ x^n_{k+1} $ simultaneously. The upper bounds of the norm of the first and the second partial derivatives of $ s_k $ and the estimation of $| \;\!t_{n,\,k}|, \| \;\! \Bu_{n,\,k} \|$ and $ \| \;\! \Bd_{n,\,k} \|$ imply the bounds of norm of second partial derivatives of $ S_k^n $ as follows.
\msk 
\begin{equation*}
\begin{aligned}
\left |\dfrac{\partial^2 S_k^n}{\partial xy} \right | & \leq \big (1 + O(\rho^{n-k}) \big ) \left |\dfrac{\partial^2 S_{k+1}^n}{\partial xy} \right | + C \bar \eps^{2^k}\si^{n-k}  \\[0.2em]
 \left |\dfrac{\partial^2 S_k^n}{\partial xz_i} \right | & \leq \big (1 + O(\rho^{n-k}) \big ) \left |\dfrac{\partial^2 S_{k+1}^n}{\partial xz_i} \right | + C \bar \eps^{2^k}\si^{n-k} \\[0.2em]
  \left |\dfrac{\partial^2 S_k^n}{\partial yz_i} \right | & \leq \big (1 + O(\rho^{n-k}) \big ) \left |\dfrac{\partial^2 S_{k+1}^n}{\partial yz_i} \right | + C \bar \eps^{2^k} \\[0.2em]
  \left |\dfrac{\partial^2 S_k^n}{\partial z_i z_j} \right | & \leq \big (1 + O(\rho^{n-k}) \big ) \left |\dfrac{\partial^2 S_{k+1}^n}{\partial z_i z_j} \right | + C \bar \eps^{2^k} \ .
\end{aligned} \msk
\end{equation*}
Hence, \
$ | \:\!\partial^2_{xy} S^n_k| = O(\bar \eps^{2^k}\si^{n-k}), \  | \:\!\partial^2_{xz_i} S^n_k| = O(\bar \eps^{2^k}\si^{n-k})$, \ 
 $| \:\!\partial^2_{yz_i}S^n_k| = O(\bar \eps^{2^k}),$ \ and   \ $| \:\!\partial^2_{z_iz_j} S^n_k| = O(\bar \eps^{2^k}) $ for every $ 1 \leq i,j \leq m $.
\end{proof}
\msk

\subsection{Universal properties of coordinate change map, $ \Psi^n_k $} \label{universal functions}
On the following Lemma \ref{asymptotics of S for k}, we would show that the non-linear part of the coordinate change map $ \id + S(x,y,\Bz) $ is a small perturbation of the one-dimensional universal function. The content of this section is to rephrase some parts of Section 7 in \cite{CLM}. 
\ssk \\
Recall the one dimensional map $f_* \colon I \ra I$ is the fixed point of the (periodic doubling) renormalization operator of the unimodal maps, namely, $R f_* =f_*$. Let the critical point of $ f_* $ be $ c_* $ and $ I = [-1,1] $. Also assume that $f_*(c_*) =1$ and $f^2_*(c_*) = -1$. Let us take the intervals $J_c^* = [-1, f^4_*(c_*)]$ and $J_v^*= f_*(J^*_c) = [f^3_*(c_*),1]$. Then these intervals are the smallest invariant intervals under $f_*^2$ around the critical point and the critical value respectively. Observe that the critical point $ c_* $ is in $ J_c^* $ and $ f_*(J_v^*) = J_c^* $. 
Let the onto map $s \colon J^*_c \ra I$ be the orientation reversing affine rescaling. Thus $s \circ f_* \colon J^*_v \ra [-1, 1]$ is an expanding diffeomorphism.  We can consider the inverse contraction
$$ g_* \colon I \ra J^*_v, \quad g_*=f^{-1}_* \circ s^{-1}
$$
where $f^{-1}_*$ is the branch of the inverse function which maps $J^*_c$ onto $J^*_v$. The map $g_*$ is called the {\em presentation function} and it has the unique fixed point at 1. 
By the definition of $g_*$ implies that 
$$ f^2_* | J^*_v = g_* \circ f_* \circ (g_* )^{-1}
$$
Then by the appropriate rescaling of the presentation function, $g_*$, we can define the {\em renormalization at the critical value}, 
$R^n_v f_*$. 
Inductively we can define $g^n_*$ on the smallest interval $J^*_v(n)$ containing the critical value 1 with period $2^n$. Let $G^n_* \colon I \ra I$ be the diffeomorphism of the rescaled map of $g_*^n$.
\msk \\
Then the fact that $g_*$ is the contraction implies the existence of the limit.
$$ u_* = \lim_{n \ra \infty} G^n_* \colon I \ra I $$
 and the convergence is exponentially fast in $C^3$ topology.
Moreover, we see the following lemmas in \cite{CLM}.

 \begin{lem} [Lemma 7.1 in \cite{CLM}]   \label{renormalization at critical value}
 For every $n \geq 1$ \msk
 \begin{enumerate}
 \item  $J^*_v(n) = g^n_*(I)$ \msk
 \item  $ R^n_v f_* = G^n_* \circ f_* \circ (G^n_*)^{-1}$ \msk
 \item  $ u_* \circ f_* = f^* \circ u_*$ \ssk
 \end{enumerate}
 \end{lem}
 
 \begin{lem} [Lemma 7.3 in \cite{CLM}]          \label{composition of the presentation function}
 Assume that there is the sequence of smooth functions  $g_k \colon I \ra I, \ \ k=1,2, \ldots n$ such that $\|g_k - g_*\|_{C^3} \leq C\rho^k$ where the $g_* = \lim_{k \ra \infty} g_k$ \, for some constant $C>0$ and $\rho \in (0,1)$. Let $g^n_k = g_k \circ \cdots \circ g_n$ and let $G^n_k = a^n_k \circ g^n_k \colon I \ra I$, where  $a^n_k$ is the affine rescaling of $\Im g^n_k$ to $I$. Then $\| G^n_k - G^{n-k}_* \|_{C^1} \leq C_1\rho^{n-k}$, where $C_1$ depends only on $\rho$ and $C$.
\end{lem}

\nin Let us normalize the functions $u_*$ and $g_*$ which have the fixed point at the origin and the derivatives at the origin is 1. Let
$$ v_*(x) = \frac{u_*(x+1)-1}{u_* ' (1)} 
$$
Abusing notation, we denote the normalized function of $g_*(x)$ to be also the $g_*(x)$ in the following lemma.

\begin{lem} \label{asymptotics of S for k}
There exists the positive constant $\rho <1$ such that for all $k<n$ and for every $y \in I^y $ and $ z \in   I^z$
\begin{align*}
  | \id + S^n_k(\, \cdot \,, y,\Bz) - v_*(\,\cdot \,) \: \!| =& \ O (\bar \eps^{2^k} y + \bar \eps^{2^k} \sum_j z_j + \rho^{n-k}) \\
 \textrm{and} \quad  | \,1+ \di_x S^n_k (\,\cdot \,, y, \Bz) - v_*'(\,\cdot \,)\: \! | =& \ O (\rho^{n-k}) .
\end{align*}
\end{lem}

\begin{proof}
The map \,$ \id + S^n_k(\,\cdot \, , y,\Bz)$ is the normalized map of $\Psi^n_k$ such that the derivative at the origin is the identity map, and $v_*(\, \cdot \,)$ is also the normalized map of $u_*$, which is the conjugation of the renormalization fixed point at the critical point and the critical value in Lemma \ref{renormalization at critical value}. Thus the normalized map, $ \id + S^n_k(\,\cdot \, , 0,\B0)$ and the one dimensional map, $G_*^n$ converge to the same function $v_*(\, \cdot \,)$ as $ n \ra \infty $ because the critical value of $f$ and the tip of $F$ moved to the origin as the fixed point of each function $g_*^n$ by the appropriate affine conjugation. 
\ssk \\
By Lemma \ref{asymptotics of non-linear part 1} we have
\begin{align*}
| \; \! \partial_y S^n_k|  \ = O(\bar \eps^{2^k}), \qquad  | \; \! \partial_{z_i} S^n_k| \ = O(\bar \eps^{2^k}) 
\end{align*}
for all $ 1 \leq i \leq m $ and moreover,
 $$ | \;\! \partial^2_{xy} S^n_k| = O(\bar \eps^{2^k}\si^{n-k}), \qquad  | \;\! \partial^2_{xz_i} S^n_k| = O(\bar \eps^{2^k}\si^{n-k}) $$
for all $ 1 \leq i \leq m $. Thus the proof of asymptotic along the section parallel to $ x- $axis is enough to prove the whole lemma. By Lemma \ref{exponential convergence to 1d map}, 
$$  \dist_{C^3} (\, \id + s_k(\,\cdot \, , 0,\B0),\ g_*(\,\cdot \,) ) = O(\rho^k)
$$
and by Lemma \ref{composition of the presentation function}, we obtain 
\begin{align} \label{C1 convergence of coordinate change}
\dist_{C^1} (\, \id + S^n_k(\,\cdot \, , 0,\B0),\ G_*^{n-k}(\,\cdot \,) ) = O(\rho^{n-k}) .
\end{align}
Since the $G_* ^{n} \ra v_*$ exponentially fast, we have the exponential convergence of the function $\id + S^n_k(\,\cdot , 0,\B0)$ to $v_*(\,\cdot \,)$. 
Hence, the above asymptotic and the exponential convergence at the origin prove the first part of the lemma. Furthermore, $ C^1 $ convergence of \eqref{C1 convergence of coordinate change} implies that 
$$ | \, 1+ \di_x S^n_k (\,\cdot \, , 0, \B0) - v_*'(\,\cdot \,) | = O (\rho^{n-k})$$
 where $\rho \in (0,1)$. 
\end{proof}
\msk

\subsection{Estimation of the quadratic part of $S^n_k$ for $n$} \label{asymptotic coordinate for n}
We estimate the asymptotic of $S^n_k$ using the estimation of the partial derivatives and recursive formulas. Then it implies the estimation of the asymptotic of the non-linear part of $\Psi^n_k$ for $ n $. 
In order to simplify notations, we would treat the case $ k =0$ and consider the behavior of $S^n_0$ instead of $S^n_k$. 

\nin In this section, let the variable $ y $ be $ z_0 $ if we express the quadratic sum of $ y $ and $ z_j $ variables to simplify the notations.

\begin{lem} \label{asymptotics of non linear part}
The following asymptotic is true
$$ \big| \: \! [\; x + S^n_0(x,y,\Bz) ] - [\, v_*(x) + \sum_{0 \leq i,j \leq m} \!a_{F,\,ij}\,z_i \:\! z_j \,] \big| = O(\rho^n)
$$
where constants $ |a_{F,\,ij}| $ are $ O(\bar \eps) $ for all $ 0 \leq i,j \leq m $ and for some $\rho \in (0,1)$.
\end{lem}
\begin{proof}
For any fixed $k \geq 0$, the recursive formula for $n > k$ comes from the $\Psi^{n+1}_k = \Psi^n_k \circ \Psi^{n+1}_n$. Thus
\begin{align} \label{recursive formula for n}
{\bf S}^{n+1}_k (w) = {\bf s}_n (w) + D^{-1}_n \circ {\bf S}^n_k \circ D_n \circ (\id + {\bf s}_n) (w) .
\end{align}
Let $k=0$ for simplicity, and compare each coordinates of the both sides. Then \ssk
\begin{equation*}
\begin{aligned}
& \  (S^{n+1}_0(w),\ 0 ,\ {\bf R}_{n+1,\,0}(y))  \\[0.3em]
 = & \ (s_n(w),\ 0,\ \Br_n(y)) 
 +  \left (
   \begin{array} {ccc}
    \alpha_n^{-1}  & \alpha_n^{-1}\,(-t_n + \Bu_n \cdot \Bd_n)  & - \alpha_n^{-1} \Bu_n \\ [0.3em]
              &\ \  \si_n^{-1} &         \\ [0.3em]
              & -\si_n^{-1} \Bd_n  & \ \  \si_n^{-1} \cdot \Id_{m \times m}
       \end{array}
\right )     
\left (
       \begin{array}{c}
S^n_0(w) \\ [0.3em]
0 \\ [0.3em]
{\bf R}_{n,\,0}(y)
        \end{array}
\right ) \\[0.8em]
& \hspace{2.5in} \circ  
\left( 
       \begin{array} {ccc}
 \alpha_n & \si_n t_n & \si_n \Bu_n  \\ [0.3em]
               & \si_n       &      \\ [0.3em]
               & \si_n \Bd_n & \si_n \cdot \Id_{m \times m}
       \end{array}  
\right)       
\left(
       \begin{array} {c}
       x+ s_n(w) \\ [0.3em]
       y \\ [0.3em]
       \Bz + \Br_n(y)
       \end{array}  
\right) .
\end{aligned} \msk
\end{equation*}
\nin By the direct calculation, we obtain the following equation
\msk
\begin{equation*}
\begin{aligned}
& \  (S^{n+1}_0(w),\ 0 ,\ {\bf R}_{n+1,\,0}(y))  \\[0.3em]
 = & \ (s_n(w),\ 0,\ \Br_n(y))  
  +  \left( \dfrac{1}{\alpha_n}\,S^n_0(w) - \dfrac{1}{\alpha_n}\,\Bu_n \cdot {\bf R}_{n,\,0}(y),\ 0 ,\ \dfrac{1}{\si_n}\, {\bf R}_{n,\,0}(y) \right) \circ \\
  & \quad  \Big( \alpha_n (x+ s_n(w)) + \si_n t_n \:\! y + \si_n \Bu_n \cdot (\Bz + \Br_n(y)), \ \si_n y, \  \si_n \Bd_n \:\! y + \si_n (\Bz + \Br_n(y)) \Big) \\[0.5em]
 = & \  (s_n(w),\ 0,\ \Br_n(y))  \\
 & + \left(  \dfrac{1}{\alpha_n}\,S^n_0 \Big( \alpha_n (x+ s_n(w)) + \si_n t_n \:\! y + \si_n \Bu_n \cdot (\Bz + \Br_n(y)), \  \si_n y, \  \si_n \Bd_n \:\! y + \si_n (\Bz + \Br_n(y)) \Big) \right. \\
 & \qquad \left. - \:\! \dfrac{1}{\alpha_n}\, \Bu_n \cdot {\bf R}_{n,\,0}(\si_n y) ,\ 0 , \ \dfrac{1}{\si_n}\, {\bf R}_{n,\,0}(\si_n y) \right) .
\end{aligned} \msk
\end{equation*}

\nin Firstly, let us compare from third to $ m+2 $ coordinates of each side of the above equation. Using the Taylor's expansion and Lemma \ref{bounds on domain}, we obtain
\begin{align*}
{\bf R}_{n+1,\,0}(y) & = \Br_n(y) + \dfrac{1}{\si_n}\,{\bf R}_{n,\,0}(\si_n y) 
\end{align*}
\ssk
Recall the coordinate expression of $ {\bf R}_{n,\,0}(y) = \big(\,{R}_{n,\,0}^1(y),\ {R}_{n,\,0}^2(y),\ \ldots ,\ {R}_{n,\,0}^m(y)\,\big) $. 
Then we have the following form of ${R}_{n,\,0}^j(y)$.
\begin{align*}
{R}_{n,\,0}^j(y) & = a_n^j\, y^2 + A_n^j(y)\, y^3 \\
{R}_{n+1,\,0}^j(y) & = \dfrac{1}{\si_n}\, \Big( a_n^j \cdot (\si_n y)^2 + A_n^j(\si_n y) \cdot (\si_n y)^3 \Big) 
+ c_n^j \,y^2 + O(\bar \eps^{2^n} y^3) .
\end{align*}
for each $ 1 \leq j \leq m $. Thus $a_{n+1}^j = \si_n a_n^j + c_n^j $  and $ \| A_{n+1}^j \| \leq \| \:\!\si_n \|^2   \| A_n^j \| + O(\bar \eps^{2^n}) $. 
for all $ 1 \leq j \leq m $. Hence,  $A_n^j \ra 0$ \ and \ $ a_n^j \ra 0$ \ exponentially fast as $n \ra \infty$.

\nin Moreover, the second part of Lemma \ref{asymptotics of S for k} implies the following equation
\begin{align} \label{exponential convergence of S}
 \big| \: \! [\,x + S^n_0(x,y,\Bz) ] - [ \,v_*(x) + S^n_0 (0,y,\Bz) ] \big| = O(\rho^n) .
\end{align}
\msk
Secondly, compare the first coordinates of \eqref{recursive formula for n} at $(0,y,z)$
\begin{align*}
 & \ S^{n+1}_0(0,y,\Bz) \\[0.2em]
  = & \ s_n(0,y,\Bz)   \\
 \quad & + \dfrac{1}{\alpha_n}\,S^n_0 \Big( \alpha_n s_n(0,y,\Bz) + \si_n t_n y + \si_n \Bu_n \cdot (\Bz + \Br_n(y)),\ \si_n y,\ \si_n \Bd_n y + \si_n (\Bz + \Br_n(y)) \Big)   \\
 &    -  \dfrac{1}{\alpha_n}\, \Bu_n \cdot {\bf R}_{n,\,0}(\si_n y) .
\end{align*}
The estimation of\, $| \: \! \partial^2_{xy} S^n_k|,\,  |  \:\! \partial^2_{xz_i} S^n_k|$ \,and\, $|\: \! \partial^2_{yz_i}S^n_k|,\, | \; \! \partial^2_{z_i z_j} S^n_k| $ for $ 1 \leq i,j \leq m $ in Lemma \ref{asymptotics of non-linear part 1} implies that 
$$ \frac{\di S^n_0}{\di x} (0,y,\Bz) = O\Big(\si^{n} y + \si^n \sum_{j=1}^m z_j\,\Big) \quad \textrm{and} \quad
 \frac{\di S^n_0}{\di z_i} (0,y,\Bz) = O( y + z_i)
$$
respectively. Recall that the variable $ z_0 $ is the variable $ y $. The order of the $t_n, \Bu_n, \Br_n$ and Taylor's expansion of\, $S^n_0$\, at $(0, \si_n y, \si_n \Bz)$ implies that 
\begin{align*}
 & S^{n+1}_0(0,y,\Bz) \\[0.5em]
  = & \ s_n(0,y,\Bz)   \\[0.3em]
  \quad & + \dfrac{1}{\alpha_n} \bigg[ S^n_0(0,\,\si_n y,\, \si_n \Bz) +  \frac{\di S^n_0}{\di x} (0,\,\si_n y,\, \si_n \Bz) \cdot \Big(\alpha_n s_n(0,y,\Bz) + \si_n t_n y + \si_n \Bu_n \cdot (\Bz + \Br_n(y)) \Big) \\[0.3em]
  \quad & + \sum_{j=1}^m \frac{\di S^n_0}{\di z_j} (0,\,\si_n y,\,\si_n \Bz) \cdot \Big( \si_n d_n^j \,y + \si_n r_n^j(y) \Big) \bigg] -  \dfrac{1}{\alpha_n}\, \Bu_n \cdot {\bf R}_{n,\,0}(\si_n y) + O\bigg( \bar \eps^{2^n} \!\!\!\!\!\! \sum_{\substack {0 \leq \,i,\,j,\;k \leq \,m}} \!\!\!\! z_i z_j z_k \bigg)   \\
  = & \ \dfrac{1}{\alpha_n}\, S^n_0(0,\,\si_n y,\, \si_n \Bz) + \sum_{0 \leq \,i,j \leq \,m} e_{n,\,ij}\,z_i z_j +  O\bigg( \bar \eps^{2^n} \!\!\!\!\!\! \sum_{\substack {0 \leq \,i,\,j,\;k \leq \,m}} \!\!\!\! z_i z_j z_k \bigg) 
\end{align*} 
where $e_{n,\,ij} = O(\bar \eps^{2^n})$ \ for all\ $ 0 \leq i,j \leq m $. 
Then we can express $S^n_0(0,y,\Bz)$ as the quadratic and higher order terms, 
\begin{equation*}
\begin{aligned}
 S^n_0(0,y,\Bz) &= \sum_{0 \leq \,i,\,j \leq \,m} a_{n,\,ij}\,z_i z_j  + A_n(y,\Bz) \cdot \bigg(\, \sum_{\substack {0 \leq \,i,\,j,\;k \leq \,m}} \!\!\!\! z_i z_j z_k \bigg) .
\end{aligned} 
\end{equation*}
The recursive formula for $S^n_0(0,y,z)$ implies that
\begin{align*}
 \quad & S^{n+1}_0(0,y,\Bz) \\
=& \ \dfrac{1}{\alpha_n} \left[\,\sum_{0 \leq \,i,j \leq \,m} a_{n,\,ij}\,(\si_n z_i)(\si_n z_j)
+ A_n(\si_n y, \si_n \Bz) \cdot \bigg(\, \sum_{\substack {0 \leq \,i,\,j,\;k \leq \,m}} \!\!\!\! (\si_n z_i) 
(\si_n z_j) (\si_n z_k) \bigg) \right] \\
\quad &  + \sum_{0 \leq \,i,\,j \leq \,m} \!\!\! e_{n,\,ij}\,z_i\, z_j +  O\bigg( \bar \eps^{2^n} \!\!\!\!\!\! \sum_{\substack {0 \leq \,i,\,j,\;k \leq \,m}} \!\!\!\! z_i\, z_j\, z_k \bigg)  .
\end{align*}
Hence, $a_{n+1,\,ij} = \dfrac{\,\si^2}{\,\alpha_n} \, a_{n,\,ij} + \!\!\! \displaystyle\sum_{0 \leq \,i,j \leq \,m} \!\!\! e_{n,\,ij}$ \;for $ 0 \leq i,j \leq m $ \,and moreover,
$$ \|A_{n+1} \| \leq \|A_n\| \cdot \dfrac{\ | \;\! \si_n|^3}{| \;\!\alpha_n|} + O(\bar \eps^{2^n}) . $$
It implies that $a_{n,\,ij} \ra a_{F,\,ij}$ \  for $ 0 \leq i,j \leq m $\, and\ $\| A_n\| \ra 0$ exponentially fast as $n \ra \infty$. The exponential convergence of $ S^n_0(0,y, \Bz) $ to the quadratic function of $y$ and $ \Bz$ and the equation \eqref{exponential convergence of S} show the asymptotic of $S^n_0(x,y, \Bz)$.
\end{proof}

\comm{*********
\ssk
\begin{rem}
The above Lemma can be generalized for $S^n_k$ as follows.
$$ \big| \: \! [\; x + S^n_k(x,y,z) ] - [\, v_*(x) +  \sum_{0 \leq i,j \leq m} \!a_{F,\,ij}\,z_i \:\! z_j ] \big| = O(\rho^{n-k})
$$
The constants $ | \:\!a_{F,\,ij}| $ for $ 0 \leq i,j \leq m $ \,of $S^n_k$ are $O(\bar \eps^{2^k})$.
\end{rem}

\begin{rem}
The recursive formula of $ S^n_0(0,y,z) $ can be written using matrices, which expression may be easily extendible for the perturbed H\'enon-like map of the general dimension.
\begin{align*}
& S^{n+1}_0(0,y,z) \\
  = & \ s_n(0,y,z)    + \dfrac{1}{\alpha_n} S^n_0(0,\si_n y, \si_n z) -  \dfrac{1}{\alpha_n} u_n R^n_0(\si_n y) \\
  \quad & + \dfrac{\si^n}{\alpha_n} \Big(\di_x S^n_0(0,\si_n y, \si_n z) \ \ \di_x S^n_0(0,\si_n y, \si_n z) \Big)
\left(
  \begin{array} {cc}
  t_n& u_n \\
  d_n& 0
  \end{array}
\right)
\left(
\begin{array} {c}
y \\
z+ r_n(y)
\end{array}
\right)   \\
  \quad & + \Big(\di_x S^n_0(0,\si_n y, \si_n z) \ \ \di_x S^n_0(0,\si_n y, \si_n z) \Big)
 \left(
\begin{array} {c}
s_n(0,y,z) \\
\si_n / \alpha_n \cdot r_n(y)
\end{array} 
\right)      +  O\left( \bar \eps^{2^n} \sum^3_{j=0} y^{3-j}z^j \right) 
\end{align*}
\end{rem}
*********}
\msk

\subsection{Universality of $ \Jac R^nF $}
Let the $n^{th}$ renormalized map of $F$ be $R^nF \equiv F_n = (f_n -\eps_n,\;x,\; \bde_n)$. Recall that $\Psi^n_{\tip} \equiv \Psi^n_{\Bv}$ from $n^{th}$ level to $0^{th}$ level and the tip $\tau_F $ is contained in $\in B^n_{v^n}$ for all $n \in \N $. Thus $\Psi^n_{\tip}$ is the original coordinate change rather than the normalized function $\Psi^n_{\Bv}$ conjugated by translations $T_n$. 
\\
Recall the equation \eqref{chain rule} again
\begin{align*}
\Jac F_n(w) &= \Jac F^{2^n}(\Psi^n_{\tip} (w)) \frac{\Jac \Psi^n_{\tip} (w)}{\Jac \Psi^n_{\tip}(F_nw)}  \\
                  &= b^{2^n} \frac{\Jac \Psi^n_{\tip}(w)}{\Jac \Psi^n_{\tip}(F_nw)} (1+ O(\rho^n)).
\end{align*} 

\begin{thm} [Universal expression of Jacobian determinant] \label{Universality of the Jacobian}
For the function $F \in \II (\bar \eps)$ for sufficiently small $\bar \eps >0$, we obtain that
$$ \Jac F_n = b^{2^n} a(x) \: (1+ O(\rho^n))
$$
where $b$ is the average Jacobian of \;$F$ and $a(x)$ is the universal positive function for some $\rho \in (0, 1)$.
\end{thm}
\begin{proof}
For the higher dimensional H\'enon-like map, Lemma \ref{asymptotics of S for k} and Lemma \ref{asymptotics of non linear part} are the essentially same as for the two dimensional H\'enon-like maps. Then the proof of theorem is also the same as two dimensional nmaps. See Universality theorem in \cite{CLM}. 
\end{proof}

\bsk

\section{Toy model H\'enon-like map in higher dimension} \label{Tangent bundle splitting}
\msk
Let higher dimensional H\'enon-like map satisfying $ \eps(w) = \eps(x,y) $, that is, $ \di_{z_i} \eps \equiv 0 $ for all $ 1 \leq i \leq m $ be the {\em toy model map}. Denote the toy model map by $ F_{\mod} $. Then the projected map $ \pi_{xy} \circ F_{\mod} = F_{2d} $ is exactly two dimensional H\'enon-like map. Let the horizontal-like diffeomorphism of $ F_{\mod} $ be $ H_{\mod} $. Thus we see that $ \pi_{xy} \circ H_{\mod} = H_{2d} $. Then we obtain that $ \pi_{xy} \circ RF_{\mod} = RF_{2d} $. 
\msk

\begin{prop} \label{renormalization of toy model}
Let $ F_{\mod} = (f(x) - \eps(x,y),\ x,\ \bde(w)) $ be a toy model map in $ \II(\bar \eps) $. Then $ n^{th} $ renormalized map $ R^nF_{\mod} $ is a toy model map, that is,
$$ \pi_{xy} \circ R^nF_{\mod} = R^nF_{2d} $$
for every $ n \in \N $. Moreover, $ \eps_n $ is of the following form
$$ \eps_n(w) = \ b_1^{2^n}a(x)\:\! y (1+O(\rho^n)) $$ 
where $ b_1 $ is the average Jacobian of the two dimensional map, $ F_{2d} = \pi_{xy} \circ F_{\mod} $ 
and $ a(x) $ is the non vanishing diffeomorphism on $ \pi_x(B) $. 
\end{prop}

\msk

\subsection{Tangent bundle splitting under $ DF_{\mod} $}
Let $ DF $ and $ D\de $ be the Fr\'echet derivative of $ F $ and $ \de $ respectively. 
For the given point $ w = (x,\, y,\, \Bz) $, let us denote $ w_i = (x_i,\, y_i,\, \Bz_i) = F^i(x,\, y,\, \Bz) $. The (Fr\'echet) derivative of $ F $ has the block matrix form
\[
\renewcommand{\arraystretch}{1.3}
DF_{\mod} =  \left(
\begin{array} {cc|c}
\multicolumn{2}{c|}{\multirow{2}{*}{$DF_{2d}$}} & { \B0}  \\ 
 & & { \B0} \\    \hline  
\di_x \bde & \di_y \bde &\di_{\Bz} \bde
\end{array}
\right) 
\]
where $ \di_{\Bz} \bde $ is the $ m \times m $ block matrix 
\begin{equation*}
\begin{pmatrix}
\di_{z_1} \de^1 & \di_{z_2} \de^1 & \cdots & \di_{z_m} \de^1 \\
\di_{z_1} \de^2 & \di_{z_2} \de^2 & \cdots & \di_{z_m} \de^2 \\
\vdots & \vdots & \ddots & \vdots \\
\di_{z_1} \de^m & \di_{z_2} \de^m & \cdots & \di_{z_m} \de^m 
\end{pmatrix}
\end{equation*}
\ssk
The Lemmas in this section using block matrices are due to the same notions in \cite{Nam1} because the construction of invariant cone field could be applied to the object of any finite dimension. In the below we use the same notations in \cite{Nam1}. 
\begin{equation} \msk   \label{eq-block matrix form of derivative}
\begin{aligned}
DF_{\mod}(x,y,\Bz) =
\begin{pmatrix} 
A(w)& {\bf 0} \\
C(w)& D(w)
\end{pmatrix} \equiv
\begin{pmatrix}
A_1 & {\bf 0} \\
C_1 & D_1
\end{pmatrix} .
\end{aligned} 
\end{equation} 
where $ A_w = DF_{2d}(x,y) $, 
${\bf 0} = \bigl(\begin{smallmatrix}
{ \B0} \\ 
{ \B0}
\end{smallmatrix} \bigr)$ , 
$ C_w = (\di_x \bde(w) \ \, \di_y \bde(w)) $ and $ D_w = \di_{\Bz} \bde(w) $. 
Since we assume that $ F_{\mod} $ and $ F_{2d} $ are diffeomorphisms, $ DF_{\mod} $ and $ A_w $ are invertible. It implies that $ D_w $ is invertible at each $ w $. Let $ w_N $ be $ F^N(w) $ and the derivative of the $ N^{th} $ iterated map $ F^N_{\mod} $ be $ DF^N_{\mod} $. Denote $ DF^N_{\mod} $ as the block matrix form as follows \msk
\begin{equation} \label{block matrix of DF-mod infty}
\begin{aligned}
DF^N_{\mod}(x,y,\Bz) = 
\begin{pmatrix}
A_N(w) & {\bf 0} \\
C_N(w) & D_N(w) 
\end{pmatrix} \equiv
\begin{pmatrix}
A_N& {\bf 0} \\
C_N& D_N
\end{pmatrix} .
\end{aligned} 
\end{equation}

\nin 
Then for each $ N \geq 1 $, 
\begin{equation*}
\begin{aligned}
\begin{pmatrix}
A_N& {\bf 0} \\
C_N& D_N
\end{pmatrix} =
\begin{pmatrix}
A_1(w_{N-1}) & {\bf 0} \\
C_1(w_{N-1}) & D_1(w_{N-1})
\end{pmatrix} \cdot
\begin{pmatrix}
A_{N-1}& {\bf 0} \\
C_{N-1}& D_{N-1}
\end{pmatrix} .
\end{aligned} \msk
\end{equation*}

\nin Let $ A_0 \equiv 1 $, $ C_0 \equiv 1 $, $ D_0 \equiv 1 $ and $ w = w_0 $ for notational compatibility. Then by direct calculations, we obtain
\begin{equation} \label{component of DF-n infty}
\begin{aligned}
A_N &= A_1(w_{N-1})\;\! A_{N-1} = \prod_{i=0}^{N-1} A_1(w_{N-i-1}) \\
D_N &= D_1(w_{N-1})\;\! D_{N-1} = \prod_{i=0}^{N-1} D_1(w_{N-i-1}) \\[0.3em]
C_N &= C_1(w_{N-1})\;\! A_{N-1} + D_1(w_{N-1})\;\! C_{N-1} \\
&= \sum_{i=0}^{N-1} D_i(w_{N-1-i})\, C_1(w_{N-1-i}) \, A_{N-1-i} . \\
\end{aligned} 
\end{equation}

\nin We see that $ [DF^N_{\mod}(w)]^{-1} = DF^{-N}_{\mod}(F^{N}(w)) $ by \ssk inverse function theorem. Thus using block matrix expressions, $ [DF^N_{\mod}(w)]^{-1} $ is
\begin{equation} \label{block matrix of DF-1 infty}
\begin{aligned}
DF^{-N}_{\mod} = 
\begin{pmatrix}
A^{-1}_N & {\bf 0} \\[0.4em]
- D^{-1}_N \;\! C_N \;\! A_N^{-1} & D^{-1}_N
\end{pmatrix}
\end{aligned}
\end{equation}
at the point, $ F^{N}(w) $. 
\msk \\ Let the cone at $ w $ with some positive number $ \gamma $ to be
\begin{align} \label{complement cone infty}
\CC(\gamma)_w = \{ u + v \; | \; u \in \R^2 \times \{\B0\},\; v \in \{\B0\} \times \R^m \ \ \textrm{and} \ \ \dfrac{1}{\gamma}\,\| \:\! u\| >  \|v\|  \; \} .
\end{align}
The cone field over a given compact invariant set $ \Gamma $ is the union of the cones at every points in $ \Gamma $
\begin{align} \label{complement cone field infty}
\CC(\gamma) = \bigcup_{w \in \,\Gamma} \CC(\gamma)_w .
\end{align}
Let $ \| DF \| $ be the operator norm of $ DF $. \footnote{The operator norm is defined on the linear operator at each point. For example, 
$$ \| DF_w \| = \sup_{ \| v \| =1} \{ \| DF_w v \| \} $$
The value $ \| DF \| $ is defines as $ \sup_{ w \in B } \| DF_w \| $. } The minimum expansion rate (or the strongest contraction rate) of $ DF $ is defined by the equation, $ \| DF^{-1} \| = \dfrac{1}{m(DF)} $.

\msk
\begin{lem} \label{upper bound of C-N A-N infty}
Let $ A_N $, $ {\bf 0} $, $ C_N $ and $ D_N $ be components of $ DF^N_{\mod} $ defined on \eqref{block matrix of DF-mod infty}. Suppose that $ \| D_1 \| < m(A_1) $. Then $ \| \;\! C_N \;\! A_N^{-1} \| < \kappa $ for some $ \kappa > 0 $ independent of $ N $.
\end{lem}
\begin{proof}
See Lemma 7.2 in \cite{Nam1}. 
\comm{******************
By \eqref{block matrix of DF-1 infty},
\begin{equation} \label{A-N -1 prod}
A_N^{-1}(w_N) = \prod_{i=0}^{N-1} A_1^{-1}(w_i) = \prod_{j=0}^{N-1-i} A_1^{-1}(w_j) \, \prod_{j=N-i}^{N-1} A_1^{-1}(w_j) = A_{N-1-i}^{-1} \, A^{-1}_i(w_{N-i})
\end{equation}
By \eqref{component of DF-n infty}, $ \| D_k \| \leq \| D_1 \|^k $ and by \eqref{A-N -1 prod} $ m(A_k) \geq m(A_1)^k $ for any $ k \in \N $.
Then \msk
\begin{equation}
\begin{aligned}
& \quad \ \ \| \;\! C_N \;\! A_N^{-1}(w_N) \| \\
&=  \ \Big\| \sum_{i=0}^{N-1} D_i(w_{N-1-i})\, C_1(w_{N-1-i}) \, A_{N-1-i}  \;\! A_N^{-1}(w_N) \Big\| \\
&=  \ \Big\| \sum_{i=0}^{N-1} D_i(w_{N-1-i})\, C_1(w_{N-1-i}) \,A^{-1}_i(w_{N-i}) \Big\| \\
&\leq \ \sum_{i=0}^{N-1} \| D_i \| \| C_1 \| \| A^{-1}_i \| = \ \| \:\! C_1 \| \sum_{i=0}^{N-1} \frac{\| D_i \|}{m(A_i)} \\
&\leq \ \| \:\! C_1 \| \sum_{i=0}^{N-1} \left( \frac{\| D_1 \|}{m(A_1)} \right)^i \leq \ \| \:\! C_1 \| \sum_{i=0}^{\infty} \left( \frac{\| D_1 \|}{m(A_1)} \right)^i \\[0.4em]
& = \ \frac{\| \:\! C_1 \| \cdot  m(A_1)}{m(A_1) - \| D_1 \| }
\end{aligned} \msk
\end{equation}
Then we can choose $ \kappa = \dfrac{\| \:\! C_1 \| \cdot  m(A_1)}{m(A_1) - \| D_1 \| } $ which is independent of $ N $.
*****************************************}
\end{proof}

\msk
\begin{lem} \label{Invariance of cone field infty}
Let $ F_{\mod} \in \II(\bar \eps) $ with small enough $ \bar \eps >0 $. Suppose that $ \| D_1 \| \leq \frac{\ \rho}{\ 2} \cdot m(A_1) $ for some $ \rho \in (0,1) $. Let $ \CC(\gamma) $ over the given compact invariant set $ \Gamma $ be the cone field which is defined on \eqref{complement cone field infty} with cones in \eqref{complement cone infty}. Then $ \CC(\gamma) $ is invariant under $ DF^{-1}_{\mod} $ for all sufficiently small $ \gamma > 0 $. More precisely, any given invariant compact set $ \Gamma $ has the 
dominated splitting.
\end{lem}
\begin{proof}
See Lemma 7.3 in \cite{Nam1}. 
\end{proof}
\msk

\begin{rem}
In Lemma \ref{upper bound of C-N A-N infty}, \ssk components of matrix form, $ A_1 $, $ D_1 $, $ A_N $ and $ D_N $ depends on each point $ w \in \Gamma $. Then $ \dfrac{\| D_1(w) \|}{ m(A_1(w))} \leq \dfrac{\;1}{\;2}\,\rho_w $ for some positive $ \rho_w < 1 $. Then the actual assumption is that the set of $ \rho_w > 0 $ for $ w \in \Gamma $ is totally bounded above by the number less than $ \frac{\,1}{\;2} $. However, since $ \Gamma $ is compact, $ \{\, \rho_w \,|\; w \in \Gamma \,\} $ is the precompact set. Then $ \rho $ can be chosen as the supremum of $ \{\, \rho_w \,|\; w \in \Gamma \,\} $. Then $ \kappa $ in Lemma \ref{upper bound of C-N A-N infty} is independent of $ w \in \Gamma $. Thus $ \kappa $ is also uniform number independent of $ w $. Moreover, the cone field $ \CC(\gamma) $ in Lemma \ref{Invariance of cone field infty} is contracted in uniform rate by $ DF^{-1} $.
\end{rem}

\msk
\subsection{Tangent bundle splitting under a small perturbation of toy model map}
The existence of the invariant cone field under $ DF_{\mod} $ is still true when a small perturbation of $ DF_{\mod} $ is chosen. Let us consider the block diagonal matrix form of $ DF $. Let the following map be a {\em perturbation} of the toy model map, $ F_{\mod} (w) = (f(x) - \eps_{2d}(x,y),\;x,\;\bde(w)) $ 
\begin{equation}  \label{perturbation of model maps infty}
F(w) = (f(x) - \eps_{2d}(x,y) - \widetilde{\eps}(w),\ x,\ \bde(w))
\end{equation}
where $ \eps(w) = \eps_{2d}(x,y) + \widetilde{\eps}(w) $. Thus $ \di_{z} \eps(w) = \di_{z} \widetilde{\eps}(w) $. \msk
\begin{equation} \label{matrix symbolic expression of DF infty}
\begin{aligned}
DF  = 
\renewcommand{\arraystretch}{1.3}
 \left(
\begin{array} {cc|c}
\multicolumn{2}{c|}{\multirow{2}{*}{$D \widetilde {F}_{2d}$}} & \di_{\Bz} \eps  \\ 
 & & { \B0} \\    \hline  
\di_x \bde & \di_y \bde &\di_{\Bz} \bde
\end{array}
\right)
= \left( \renewcommand{\arraystretch}{1.3} \begin{array}{c|c}
A & B \\
\hline
C & D
\end{array} \right)
\end{aligned} \bsk
\end{equation}
where $ D\widetilde{F}_{2d} = \begin{pmatrix}
f'(x) - \di_x \eps(w) & -\di_y \eps(w) \\[0.3em]
1  & 0
\end{pmatrix} $ and $ \di_{\Bz} \eps $ is the row vector $ (\di_{z_1} \eps \ \; \di_{z_2} \eps \ \ldots \ \di_{z_m} \eps ) $. 
If $ B \equiv {\bf 0} $, then $ F $ is $ F_{\mod} $. 
\footnote{If the bounded operator $ T $ has $ \| T \| <1 $, then $ \Id - T $ is invertible. Moreover,
\begin{equation*}
(\Id - T)^{-1} = \sum_{n=0}^{\infty} T^{\,n}
\end{equation*}
Since, $ \| T^n \| \leq \| T \|^n $ for every $ n \in \N $, $ \| (\Id - T)^{-1} \| \leq \dfrac{1}{1 - \| T \|}  $. Equivalently, we get the lower bound of the minimum expansion rate, $ m( \Id - T) \geq 1 - \| T \| $. }

\msk
\begin{lem} \label{Invariance of cone filed perturbation infty}
Let $ F $ be a perturbation of the toy model map $ F_{\mod} $ defined in \eqref{perturbation of model maps infty} and $ A $, $ B $, $ C $ and $ D $ are components of block matrix form of $ DF $ defined in \eqref{matrix symbolic expression of DF infty}. Suppose that $ \| D_1 \| \leq \frac{\rho_1}{2} \cdot m(A_1) $ for some $ \rho_1 \in (0,1) $. Suppose \ssk also that $ \| B \| \| C \| \leq \rho_0 \cdot m(A) \cdot m(D) $ where $ \rho_0 < \frac{\kappa \gamma}{2} $ for sufficiently small $ \gamma > 0 $. Then the cone field $ \CC(\gamma) $ defined on \eqref{complement cone field infty} is invariant under $ DF^{-1} $ .
\end{lem}

\begin{proof}
See Lemma 7.4 in \cite{Nam1}. 
\end{proof}

\nin Then the tangent space $ T_{\Gamma}B $ of $ DF $ has the splitting of the invariant subbundles $ E^1 \oplus E^2 $ such that

\ssk
\begin{enumerate}
\item $ T_{\Gamma}B = E^1 \oplus E^2 $. \msk
\item Both $ E^1 $ and $ E^2 $ are invariant under $ DF $. \msk
\item $ \| DF^n |_{E^1(x)} \| \|  DF^{-n} |_{E^2(F^{-n}(x))} \| \leq C \mu^n$ for some $ C>0 $ and $ 0 < \mu < 1 $ and $ n \geq 1 $.
\end{enumerate}
\msk
\nin Thus $ T_{\Gamma}B $ is {\em dominated} over the compact invariant set $ \Gamma $. Moreover, the dominated splitting implies that invariant sections $ w \mapsto E^1(w) $ and $ w \mapsto E^2(w) $ are continuous by Theorem 1.2 in \cite{New}. Then the maps, $ w \mapsto E^i(w) $ for $ i =1,2 $ are continuous.

\bsk
\section{Single invariant surfaces}

\subsection{Invariant surfaces and two dimensional ambient space}
Let us consider dominated splitting of $ DF $ over the given compact set. One of the subbundle, for instance, $ E^{ss} $ is uniformly contracted under $ DF|_{E^{ss}} $. If there exists an invariant submanifold whose codimension is at least one such that this manifold is tangent to invariant tangent subbundle and it contains the invariant compact set, then we expect the dynamics is reduced in this invariant manifold. 
\msk
\begin{defn}
A $ C^r $ submanifold $ Q $ which contains $ \Gamma $ is {\em locally invariant} under $ f $ if there exists a neighborhood $ U $ of $ \Gamma $ in $ Q $ such that $ f(U) \subset Q $. 
\end{defn}

\nin The necessary and sufficient condition for the existence of these submanifolds, see \cite{CP} and its references. 
\begin{thm}[\cite{CP}] \label{Existence of invariant submanifold} Let $ \Gamma $ be an invariant compact set with a dominated splitting $ T_{\Gamma}M = E^1 \oplus E^2 $ such that $ E^1 $ is uniformly contracted. Then $ \Gamma $ is contained in a locally invariant submanifold tangent to $ E^2 $ if and only if the strong stable leaves for the bundle $ E^1 $ intersect the set $ \Gamma $ at only one point.
\end{thm}

\nin Moreover, the existence of invariant submanifold is robust under $ C^1 $ perturbation.
\begin{prop}[\cite{CP}]  Let $ \Gamma $ be an invariant compact set with a dominated splitting $ E^1 \oplus E^2 $ such that $ E^1 $ is uniformly contracted. If $ \Gamma $ is contained in a locally invariant submanifold tangent to $ E^2 $, then the same holds for any diffeomorphism $ C^1 $close to $ f $ and any compact set $ \Gamma' $ contained in a small neighborhood of $ \Gamma $.
\end{prop}

\nin Let us consider the toy model map which is infinitely renormalizable. Recall that $ DF_{\mod} $ be $ \bigl(\begin{smallmatrix}
A & {\bf 0} \\ 
C & D
\end{smallmatrix} \bigr)$ where $ A $ is the derivative of $ \pi_{xy} \circ F_{\mod} $. For a given invariant compact set, let us assume that $ \Gamma $ has the dominated splitting $ T_{\Gamma}B = E^{ss} \oplus E^{pu} $ where $ DF|_{E^{ss}} $ is $ D $, $ DF|_{E^{pu}} $ is $ A $ and $ \| D \| \| A^{-1} \| < \rho < 1 $. Let $ W^{ss}(w) $ be the strong stable manifold which is tangent to $ E^{ss}(w) $ for each $ w \in \Gamma $. A pseudo unstable manifold, $ W^{pu}(w) $ is defined similarly. Observe that $ W^{pu}(w) $ is two dimensional manifold at every point $ w \in \Gamma $.

\begin{lem}[Lemma A.1 and Lemma A.2 in \cite{Nam3}] \label{lem-closure of periodic points}
Let $ F $ be an infinitely renormalizable H\'enon-like map. Then the critical Cantor set, $ \OO_F $ is the set of accumulation points of $ \Per_F $. Moreover, $ W^s(w) \cap \overline{\Per}_F = \{ w\} $ for every point $ w \in \overline{\Per}_F $.
\end{lem}
\nin The toy model map has invariant set of codimension two hyperplane under $ F $
\begin{equation} \label{eq-invariant hyperplane}
\bigcup_{ (x,y) \in \,I^x \times I^y} \{ (x,y,\Bz) \ | \ \Bz \in {I}^{\Bz} \}
\end{equation}
where $ {I}^{\Bz} = I^{z_1}\times I^{z_2}\times \cdots \times I^{z_m} $. Moreover, any vector which perpendicular to $ xy- $plane is invariant under $ DF_{\mod} $ up to the size. Then by Theorem \ref{Existence of invariant submanifold} and Lemma \ref{eq-invariant hyperplane} implies the following lemma.

\begin{lem} \label{transversal intersection at a single point}
Let $ F_{\mod} $ be the toy model map in $ \II(\bar \eps) $. Suppose that
$$ \sup_{w \in \overline\Per_{F_{\mod}}} \frac{\|D_1(w)\|}{m(A_1(w))\|} \leq \frac{1}{2} $$
where $ A_1 $ and $ D_1 $ are the block matrix in \eqref{eq-block matrix form of derivative}. Then there exists a locally invariant $ C^1 $ single surface $ Q $ which contains $ \overline\Per_{F_{\mod}} $. The surface $ Q $ meets transversally and uniquely strong stable manifold, $ W^{ss} $ at each $ w \in \overline\Per_{F_{\mod}} $. 
\end{lem}
\begin{proof}
Lemma \ref{Invariance of cone field infty} implies the dominated splitting over $ \overline\Per_{F_{\mod}} $. This dominated splitting implies that the hyperplanes in \eqref{eq-invariant hyperplane} at each point $ w $ of $ \overline\Per_{F_{\mod}} $ is the strong stable manifold at $ w $. Transversal intersection of invariant cone fields implies that the surface $ Q $ tangent to $ E^{pu} $ over $ \overline\Per_{F_{\mod}} $ meets transversally each $ W^{ss}(w) $. We may assume that $ Q $ is locally invariant by Theorem \ref{Existence of invariant submanifold}. Let us show the uniqueness of the intersection point. Suppose that $ w $ and $ w' $ are points in $ Q \cap W^{ss}(w) $. If $ w \neq w' $, then $ w' \notin \overline\Per_{F_{\mod}} $ by Lemma \ref{lem-closure of periodic points}. Take a small neighborhood $ U $ of $ w' $ in the invariant surface $ Q $. Then $ U $ converges to the neighborhood of $ F^n(w) $ in $ Q $ as $ n \ra \infty $ by Inclination Lemma. Thus $ Q $ cannot be a submanifold of the ambient space because it accumulates itself. This contradicts to Theorem \ref{Existence of invariant submanifold}. Hence, $ w $ is the unique intersection point. 
\end{proof}
\nin Let $ F_{\mod} \in \II(\bar \eps) $ which is sectionally dissipative at fixed points. Recall that the invariant plane field, $ E^{pu} $ over $ \overline\Per_{F_{\mod}} $ is two dimensional. Thus $ E^{pu} $ contains the unstable direction at every periodic points. Then the invariant surface $ Q $ tangent to $ E^{pu} $ contains the set
$$ \AAA \equiv \OO \cup \bigcup_{n \geq 0} W^u(\Orb(q_n)) $$
where each $ q_n $ is the periodic point with period $ 2^n $. The set $ \AAA $ is called the {\em global attracting set}.

\msk
\subsection{Invariant surfaces containing $ \overline{\Per} $ as the graph of $ C^r $ map}
Recall $ b_1 $ be the average Jacobian of $ F_{2d} \equiv \pi_{xy} \circ F_{\mod} $. Suppose that $ \| \di_{\Bz} \bde \| \leq \rho \cdot m(F_{2d}) $ for some $ \rho < 1 $ with dominated splitting over $ \overline\Per_{F_{\mod}} $. Denote $ \overline\Per_{F_{\mod}} $ by $ \Gamma $. Then there exists a locally invariant single surface which contains $ \AAA_{F_{\mod}} $.  
The set of $ m- $dimensional hyperplanes which are perpendicular to $ xy- $plane 
\begin{equation*}
\begin{aligned}
\bigcup_{(x,\,y) \in\,\pi_{xy}(B)} \{ (x,\, y,\, \Bz ) \,|\; \Bz \in { I}^{\Bz} \,\}
\end{aligned} 
\end{equation*}
is invariant under $ F_{\mod} $. Since tangent subbundle $ E^{ss} $ for $ DF_{\mod} $ is constant and strong stable manifold is unique at each point in $ \Gamma $, the above set contains the strong stable foliation over $ \Gamma $. The angle between each tangent spaces $ E^{ss}_w $ and $ E^{pu}_w $ is (uniformly) positive. Thus the maximal angle between $ E^{pu} $ and $ T\R^2 $ is less than $ \frac{\pi}{2} $. 
\begin{rem}
If \,$ T_{\Gamma}B = E^{ss} \oplus E^{pu} $ \,is $ r- $dominated splitting, then $ Q $ which is invariant single surface tangent to $ E^{pu} $ is a $ C^r $ surface. Moreover, since the strong stable manifolds at each point is the set of perpendicular lines to $ xy- $plane, $ Q $ is the graph of $ C^r $ function from a region in $ I^x \times I^y $ to $ I^{\Bz} $. 
\end{rem}
\nin We may assume that invariant surfaces tangent to the invariant plane field has the neighborhood, say also $ Q $, of the tip, $ \tau_{F_{\mod}} $ in the given invariant single surface which satisfies the following properties by Lemma \ref{transversal intersection at a single point}. 
\ssk
\begin{enumerate}
\item $ Q $ is contractible. \ssk
\item $ Q $ contains $ \tau_{F_{\mod}} $ in its interior and is locally invariant under $ F^{2^N} $ for big enough $ N \in \N $. \ssk
\item Topological closure of $ Q $ is the graph of $ C^r $ map from a neighborhood of $ \tau \big( \pi_{xy} \circ F_{\mod} \big) $ in $ xy- $plane to $ I^z $. \ssk 
\end{enumerate}
By $ C^1 $ robustness of the existence of single invariant surfaces, if $ F $ be a {\em sufficiently small perturbation} of $ F_{\mod} $, then there exist invariant surfaces each of which is the graph of $ C^r $ map from a region in the $ xy- $plane to $ I^z $. 
\msk \\
There exists an invariant surface $ Q $ under $ F $ on $ \pi_{xy}(B^n_{0}) $ as the graph of the $ C^r $ function $ \xi $ with $ \| D\xi \| \leq C\bar \eps $ \,only if\, there exists the dominated splitting in the previous subsection. Then the image of $ Q $ under the map $ \big(\Psi^n_{\tip}\big)^{-1} $ is also an invariant surface under $ R^nF $ for every big enough $ n \in \N $ (See Lemma \ref{invariant surfaces on each deep level} below). The existence is based on the global implicit function theorem in \cite{ZG}.
\begin{thm}[Theorem 1 in \cite{ZG}] \label{Global implicit function thm} Let $ f: \R^n \times \R^m \lra \R^m $ is continuous mapping and it is continously differentiable for second variable $ u \in \R^m $. Suppose that 
\begin{equation*}
\left| \left[ \frac{\di}{\di u}\,f(x,u) \right]_{ii} \right| - \sum_{i \neq j} \left| \left[ \frac{\di}{\di u}\,f(x,u) \right]_{ij} \right| \geq d > 0
\end{equation*}
for all $ (x,u) \in \R^n \times \R^m $ and $  i = 1,2, \ldots ,m $ where $ \left[ \frac{\di}{\di u}\,f(x,u) \right]_{ij} $ is the $ ij $ entry of the Jacobian matrix of $ f $ over the second variable $ u $. Then there exists the unique mapping $ g \colon R^n \ra \R^m $ such that $ f(x, g(x)) = 0 $. Moreover, $ g $ is continuous. If $ f $ is continuously differentiable, then so is $ g $.
\end{thm}
\comm{********************
\msk
\nin The proof of above theorem in \cite{ZG} used Banach contraction mapping theorem at every point on the domain. Let $ c $ be $ f(x_0, 0) $ for each $ x_0 $ and define $ g(x_0) = 0 $. Consider the map from the closed Banach ball to $ \R^m $
$$ \TT \colon \overline{B}(O, \|c\| / d) \ra \R^m $$
where the ball $ \overline{B}(O, \|c\| / d) = \{ u \in \R^m \ | \ \|u \| \leq \| c\| / d \} $ and the map 
$$ \TT u = u - \frac{1}{\ell}\,f(x_0,u) $$
for sufficiently big constant $ \ell > 0 $. The contraction mapping theorem implies the uniqueness and continuity of $ g $. For $ C^r $ map $ f $, $ g $ is the global $ C^r $ map by $ C^r $ continuation of local maps.
\msk
****************************}
\begin{prop} \label{invariant surfaces on each deep level}
Let $ F \in \II(\bar \eps) $ and $ Q $ be an invariant surface under $ F $, which is the graph of $ C^r $ function $ \bxi = (\xi^1,\, \xi^2, \ldots , \xi^m ) $ on $ \pi_{xy}(B^n_{\tip}) $ such that $ \| D\bxi \| \leq C_0\,\bar \eps $ for some $ C_0 >0 $. Then $ Q_n \equiv \big(\Psi^n_{\tip}\big)^{-1}(Q) $ is an invariant surface under $ R^nF $ which is the graph of $ \bxi_n = (\xi^1_n,\, \xi^2_n, \ldots , \xi^m_n ) \,\colon \pi_{xy}\big(B(R^nF)\big) \ra \pi_{\Bz}\big(B(R^nF)\big) $ such that 
$$ \bxi_n(x,y) = {\bf c}y(1 + O(\rho^n)) $$
where $ {\bf c} = (c_1,c_2,\ldots,c_m) $ for some constants $ c_i $ for $ 1 \leq i \leq m $. 
\end{prop}
\begin{proof}
The $ n^{th}$ renormalization of $ F $, $ R^nF $ is $ \big(\Psi^n_{\tip}\big)^{-1} \circ F^{2^n} \circ \Psi^n_{\tip} $. Thus $ Q_n \equiv \big(\Psi^n_{\tip}\big)^{-1}(Q) $ is an invariant surface under $ R^nF $. Let us choose a point 
$ w' = (x',y',\Bz') \in Q \cap B^n_{\tip} $ where $ B_{\tip}^n \equiv \Psi^n_{\tip}(B(R^nF)) $ and $ \Bz' = \bxi(x',y') $ where $ \bxi(x',y') = \big( \xi^1(x',y'),\,  \xi^2(x',y'),\, \ldots,  \xi^m(x',y') \big) $. 
Thus 
\begin{align*}
\textrm{graph}(\bxi) = (x',\,y',\,\bxi(x',y')) =& \ (x',\,y',\,\Bz') .
\end{align*}
Moreover, let $ \big(\Psi^n_{\tip}\big)^{-1}(x',\,y',\,\Bz') = (x,\,y,\,\Bz) \in Q_n $. Thus by equation \eqref{the image of the Psi from nth level to k+1th level}, each coordinates of $ \Psi^n_{\tip} $ as follows
\msk
\begin{align} 
x' =& \ \alpha_{n,\,0} ( x + S^n_0(w)) + \si_{n,\,0} \,t_{n,\,0} \cdot y + \si_{n,\,0}\,\Bu_{n,\,0}\cdot (\Bz + {\bf R}_{n,\,0}(y)) \label{image under the conjugation map xi for x} \\
y' =& \ \si_{n,\,0}\cdot y \label{image under the conjugation map xi for y} \\
\Bz' =& \ \si_{n,\,0}\, \Bd_{n,\,0} \cdot y + \si_{n,\,0}\,(\Bz + {\bf R}_{n,\,0}(y)) \label{image under the conjugation map xi for z}
\end{align}
where $ w = (x,\,y,\,\Bz) $. Let us show that $ Q_n $ is the graph of a well defined function $ \bxi_n = (\xi^1_n,\,\xi^2_n, \ldots , \xi^m_n ) $ from $ \pi_{xy}(B(R^nF)) $ to $ \pi_{\Bz}(B(R^nF)) $, that is, $ z_i = \xi_n^i(x,\,y) $ for $ 1 \leq i \leq m $.  
By the equations \eqref{image under the conjugation map xi for y} and \eqref{image under the conjugation map xi for z}, we see that \ssk
\begin{equation} \label{implicit expression of z}
\begin{aligned} 
\si_{n,\,0}\cdot \Bz = &\ \Bz' - \si_{n,\,0}\,\Bd_{n,\,0}\cdot y - \si_{n,\,0}\,{\bf R}_{n,\,0}(y) \\[0.2em] 
=& \ \bxi(x',\,y') - \si_{n,\,0}\,\Bd_{n,\,0}\cdot y - \si_{n,\,0}\,{\bf R}_{n,\,0}(y) \\[0.3em]
=& \ \bxi \circ \big( \alpha_{n,\,0} ( x + S^n_0(w)) + \si_{n,\,0} \,t_{n,\,0} \cdot y + \si_{n,\,0}\,\Bu_{n,\,0}\cdot (\Bz + {\bf R}_{n,\,0}(y)),\ \si_{n,\,0}\, y \big) \\
& \quad - \si_{n,\,0}\,\Bd_{n,\,0}\cdot y - \si_{n,\,0}\,{\bf R}_{n,\,0}(y)  .
\end{aligned} \msk
\end{equation}
\nin Let us show the existence of the solution of \eqref{implicit expression of z} for $ \Bz $. Define the function $ {\bf G}_n $ from $ B $ to $ \pi_{\Bz}(B) $. Each coordinate function $ G^i_n $ of $ {\bf G}_n $ is
\begin{equation*}
\begin{aligned}
G_n^i(x,y,\Bz) =& \ \xi^i \circ \big(\alpha_{n,\,0} ( x + S^n_0(w)) + \si_{n,\,0} \,t_{n,\,0} \cdot y + \si_{n,\,0}\,\Bu_{n,\,0}\cdot (\Bz + {\bf R}_{n,\,0}(y)),\ \si_{n,\,0}\, y \big) \\
& \quad - \si_{n,\,0}\,d_{n,\,0}^{\,i}\cdot y - \si_{n,\,0}\,{ R}_{n,\,0}^i(y) - \si_{n,\,0}\cdot z_i.
\end{aligned} \msk
\end{equation*}
where $ \Bz = (z_1,z_2,\ldots, z_m) $ for $ 1 \leq i \leq m $. Then the partial derivative of $ G_n^i $ over $ z_j $ is as follows \msk
\begin{equation*}
\begin{aligned}
\di_{z_j} G_n^i(x,y,\Bz) =& \ \di_x \xi^i \circ \big( \alpha_{n,\,0} ( x + S^n_0(w)) + \si_{n,\,0} \,t_{n,\,0} \cdot y + \si_{n,\,0}\,\Bu_{n,\,0}\cdot (\Bz + {\bf R}_{n,\,0}(y)),\ \si_{n,\,0}\cdot y \big) \\
& \quad \cdot \big[\, \alpha_{n,\,0} \cdot \di_{z_j} S^n_0(w) +  u_{n,\,0}^j\,\si_{n,\,0}\,\big] 
\end{aligned} \msk
\end{equation*}
for $ i \neq j $. Moreover, if $ i=j $, then 
\begin{equation*}
\begin{aligned}
\di_{z_i} G_n^i(x,y,\Bz) =& \ \di_x \xi^i \circ \big( \alpha_{n,\,0} ( x + S^n_0(w)) + \si_{n,\,0} \,t_{n,\,0} \cdot y + \si_{n,\,0}\,\Bu_{n,\,0}\cdot (\Bz + {\bf R}_{n,\,0}(y)),\ \si_{n,\,0}\cdot y \big) \\
& \quad \cdot \big[\, \alpha_{n,\,0} \cdot \di_{z_i} S^n_0(w) +  u_{n,\,0}^i\,\si_{n,\,0}\,\big] - \si_{n,\,0}  
\end{aligned} \msk
\end{equation*}

\nin Recall that $ \alpha_{n,\,0} = \si^{2n}(1 + O(\rho^n)) $, $ \si_{n,\,0} = (-\si)^n (1 + O(\rho^n)) $, $ \| \:\! \di_{z_j} S^n_0 \| = O\big(\bar \eps \big) $ and $ | \:\! u_{n,\,0}^j | = O\big(\bar \eps \big) $ for all $ 1 \leq j \leq m $. Then for $ m \geq 2 $,
\begin{equation*}
\begin{aligned}
\| \:\! \di_{z_j} G_n^{\,i} \| \leq \| \:\!\di_x \xi^i \| \cdot \big[\, \si^{2n}C_0\, \bar \eps + \si^n C_1\, \bar \eps \,\big] 
\end{aligned} \msk
\end{equation*}
for $ i \neq j $ and for some positive numbers, $ C_0 $ and $ C_1 $. 
Recall that $ \| \:\! \Bu_{n,\,0} \| \asymp \| \:\! \Bu_{1,\,0} \| $ by the equation \eqref{sum of d, u and t respectively} and $ | \;\! u_{1,\,0}^j | \asymp | \;\! \di_{z_j} \eps | $ for each $ 1 \leq j \leq m $. Thus if small enough $ \| \di_{\Bz} \eps \| $ is chosen for a small perturbation of the toy model map, then we may assume that 
$$ \| \:\! \Bu_{n,\,0} \| \cdot |\, \si_{n,\,0} | < \frac{1}{4C m^2}\; \si^n $$
where the constant $ C \geq \max_{ 1 \leq i \leq m} \| \:\!\di_x \xi^i \| $. Then we may also assume that 
\begin{equation*}
\sum_{ i \neq j} \| \:\! \di_{z_j} G_n^{\,i} \| \leq \frac{1}{4} \,\si^n
\end{equation*}
However,
\begin{equation*}
\min_{w \in B(R^nF)} | \;\! \di_{z_i} G_n^{\,i}(w) | \geq \big| -\| \:\!\di_x \xi^i \| \cdot \big[\, \si^{2n}C_0\, \bar \eps + \si^n C_1\, \bar \eps \,\big] + |\, \si_{n,\,0} | \,\big| \geq \frac{1}{2}\, \si^n
\end{equation*}
Let us consider the Jacobian matrix $ \left( \frac{\di}{\di z_j}\, G^i(x,y,\Bz) \right)_{ij} $ of $ \Bz $ variables. Then the sum of absolute value of diagonal elements dominates the sum of all other elements for every big enough $ n \in \N $. Then applying Theorem \ref{Global implicit function thm} to the map $ {\bf G_n} $ for every sufficiently big $ n \in \N $, there exists the $ C^r $ map $ \bxi_n $ from $ \pi_{xy}(B(R^nF)) $ to $ \R^m $. Furthermore, since each surface $ Q_n $ is contractible, the function $ \bxi_n(x,y) $ is defined globally by the $ C^r $ continuation of coordinate charts.
\ssk \\
Let us calculate the bounds of the norm $ \| D\bxi_n(x,y) \|$. By \eqref{image under the conjugation map xi for y} and \eqref{image under the conjugation map xi for z} with chain rule, we obtain the following equations 
\begin{equation*}
\begin{aligned}
%
\di_x \bxi \cdot \frac{\di x'}{\di x} = & \ \si_{n,\,0} \cdot \di_x \bxi_n \\[0.3em]
%
\di_x \bxi \cdot \frac{\di x'}{\di y} + \di_y \bxi \cdot \si_{n,\,0} = & \  \si_{n,\,0}\:\!\Bd_{n,\,0} + \si_{n,\,0} \cdot \di_y \bxi_n + \si_{n,\,0} \cdot ( {\bf R}_{n,\,0})'(y) .
\end{aligned} 
\end{equation*} 
%
\nin By the equation \eqref{image under the conjugation map xi for x}, each partial derivatives of $ \xi_n $ as follows
\begin{equation} \label{exponential convergence of the invariant surface xi}
\begin{aligned}
\frac{\di \bxi_n}{\di x} = & \ \frac{1}{\si_{n,\,0}} \cdot \di_x \bxi \cdot \left[ \alpha_{n,\,0} \big( 1+ \di_{x}S^n_0(w) \big) + \si_{n,\,0} \sum_{j=1}^m u_{n,\,0}^j\cdot \frac{\di \xi_n^j}{\di x} \right] \\[0.3em]
\frac{\di \bxi_n}{\di y} = & \ \frac{1}{\si_{n,\,0}} \cdot \di_x \bxi \cdot \left[ \alpha_{n,\,0}\: \di_{y}S^n_0(w) + \si_{n,\,0} \:\!t_{n,\,0} + \si_{n,\,0} \sum_{j=1}^m u_{n,\,0}^j \Big(\: \frac{\di \xi_n^j}{\di y} + ({ R}^j_{n,\,0})'( y ) \Big) \right] \\
\qquad & + \di_y \bxi - \Bd_{n,\,0} - ( {\bf R}_{n,\,0})'(y) .
\end{aligned} 
\end{equation}
\nin Recall that $ \si_{n,\,0} \asymp (-\si)^n $, $ \alpha_{n,\,0} \asymp \si^{2n} $ for each $ n \in N $. Thus
\begin{equation*}
\Big\| \frac{\di \xi_n^i}{\di x} \Big\| \leq  \ \| \di_x \bxi \| \, C_0 \:\! \si^n \leq C \bar \eps\, \si^n 
\end{equation*}
for some $ C_0>0 $. Recall also that $ \| \di_{y}S^n_0(w) \| \leq C_3\,\bar \eps $ for some $ C_3 > 0 $ by Proposition \ref{bounds of R} and moreover the values $ t_{n,\,0} $, $ \Bu_{n,\,0} $ and $ \Bd_{n,\,0} $ converge super exponentially fast to $ t_{*,\,0} $, $ \Bu_{*,\,0} $ and $ \Bd_{*,\,0} $ respectively as $ n \ra \infty $ by Lemma \ref{decomposition of derivative}. Furthermore each partial derivatives $ \di_x \bxi (w) $ and $ \di_y \bxi (w) $ to the value of each partial derivatives at the origin exponentially fast as $ n \ra \infty $. We can show that $ ({\bf R}_{n,\,0})'(y) $ converges to zero exponentially fast by the estimation of $ {\bf R}_{n,\,0} $ in Lemma \ref{asymptotics of non linear part}. With all of these facts we obtain that
\begin{equation*}
\bxi_n(x,y) = {\bf c}y(1 + O(\rho^n))
\end{equation*}
where $ {\bf c}= \dfrac{t_{*,\,0} \cdot \di_x \bxi(\tau_F) + \di_y \bxi(\tau_F) - \Bd_{*,\,0} }{1 -  \Bu_{*,\,0} \cdot \di_x \bxi(\tau_F)} \, . $ In other words, each coordinate of $ {\bf c} $ is
\begin{equation*}
 c_i = \dfrac{t_{*,\,0} \cdot \di_x \xi^i(\tau_F) + \di_y \xi^i(\tau_F) - d^{\,i}_{*,\,0} }{1 -  \Bu_{*,\,0} \cdot \di_x \bxi(\tau_F)} 
\end{equation*}
where $ {\bf c} = (c_1,c_2,\ldots,c_m) $ for $ 1 \leq i \leq m $. 
\end{proof}



\msk

\newpage

\section{Universality of conjugated two dimensional $ C^r $ H\'enon-like map}
\subsection{Renormalization of conjugated two dimensional H\'enon-like map}
Let $F \in \II(\bar \eps) $ be a small perturbation of toy model map $ F_{\mod} \in \II(\bar \eps) $. Let $ Q_n $ and $ Q_k $ be invariant surfaces under $ R^nF $ and $ R^kF $ respectively and assume that $ k < n $. Then by Lemma \ref{invariant surfaces on each deep level}, $ \Psi^n_k $ is the coordinate change map between $ R^kF^{2^{n-k}} $ and $ R^nF $ from level $ n $ to $ k $ such that $ \Psi^n_k(Q_n) \subset Q_k $. Let us define $ C^r $ two dimensional H\'enon-like map $ _{2d}F_{n,\,\bxi} $ on level $ n $ \ssk as follows
\begin{align} \label{Cr Henon map with invariant surface}
\,_{2d}F_{n,\,\bxi} \equiv \pi_{xy}^{\bxi_n} \circ R^nF|_{\,Q_n} \circ (\pi_{xy}^{\bxi_n})^{-1} 
\end{align}
\nin where the map $ (\pi_{xy}^{\bxi_n})^{-1} : (x,y) \mapsto (x,\,y,\, \bxi_n(x,y)) $ is a $ C^r $ diffeomorphism on the domain of two dimensional map, $ \pi_{xy}(B) $. In particular, the map $ F_{2d,\,\bxi} $ is defined as follows
\begin{equation} \label{Cr Henon map with invariant surface 0}
F_{2d,\:\bxi} (x,y) = (f(x) - \eps(x,y,\bxi),\,x )
\end{equation}
where $ \textrm{graph} (\bxi) $ is a $ C^r $ invariant surface under the $ m+2 $ dimensional map $ F \colon (x,y,\Bz) \mapsto (f(x) - \eps(x,y,\Bz),\;x,\;\bde(x,y,\Bz) ) $. Let us assume that $ 2 \leq r < \infty $. 
By Lemma \ref{invariant surfaces on each deep level}, the invariant surfaces, $ Q_n $ and $ Q_k $ are the graph of $ C^r $ maps $ \bxi_n(x,y) $ and $ \bxi_k(x,y) $ respectively. 
Then we can apply techniques for two dimensional conjugated map in three dimension to the maps in $ m+2 $ dimension. All results in this and following section are the same as those of two dimensional H\'enon-like maps by invariant surfaces in three dimension. See Section 4 and Section 5 in \cite{Nam3}. \ssk \\
The map $ {}_{2d}^{}\Psi^n_{k,\,\bxi,\,\tip} $ is defined as the map which satisfies the following commutative diagram
\begin{displaymath}
\xymatrix @C=1.8cm @R=1.8cm
 {
(Q_n, \tau_n)  \ar[d]_*+ {\pi_{xy,\; n}^{\bxi_n}}  \ar[r]^*+{\Psi^n_{k,\Bv,\tip}}    & (Q_k, \tau_k) \ar[d]^*+{\pi_{xy,\; k}^{\bxi_k}}   \\
 (_{2d}B_n, \tau_{2d,\;n}) \ar[r]^*+{ {}_{2d}^{}\Psi^n_{k,\, \bxi,\tip}}      & (_{2d}B_k, \tau_{2d,\;k})
   }
\end{displaymath}
\\ 
where $ Q_n $ and $ Q_k $ are invariant $ C^r $ surfaces with $ 2 \leq r < \infty $ of $R^nF$ and $R^kF$ respectively and $ \pi_{xy,\; n}^{\bxi_n} $ and $ \pi_{xy,\; k}^{\bxi_k} $ are the inverse of the graph maps, $ (x,y) \mapsto (x,y, \bxi_n) $ and $ (x,y) \mapsto (x,y, \bxi_k) $ respectively. 
\ssk \\
Using translations\, $ T_k : w \mapsto w-\tau_k $\, and\, $ T_n : w \mapsto w-\tau_n $,\, we can let the tip move to the origin as the fixed point of the new coordinate change map, $ \Psi^n_k := T_k \circ \Psi^n_{k,\,\tip} \circ T_n^{-1} $ defined on Section \ref{asymptotic coordinate}. Thus due to the above commutative diagram, the corresponding tips in $ _{2d}B_j $ for $ j=k,n $ is changed to the origin. Let $ \pi_{xy} \circ T_j $ be $ T_{2d,\,j} $ for $ j=k,n $. This origin is also the fixed point of the map $ {}_{2d}^{}\Psi^n_{k,\, \bxi} := T_{2d,\, k} \circ {}_{2d}^{}\Psi^n_{k,\,\bxi,\,\tip} \circ T_{2d,\, n}^{-1} $ where $ T_{2d,\, j} = \pi_{xy} \circ T_j $ with $ j = k, n $. 
By the direct calculation, we obtain the expression of the map $ {}_{2d}^{}\Psi^n_{k,\, \bxi} $ as follows
%
\begin{align} \label{coordinate change map of xi-2d}
{}_{2d}^{}\Psi^n_{k,\, \bxi} &= \pi_{xy,\, k}^{\bxi_k} \circ \Psi^n_k(x,y, \bxi_n) \nonumber \\[0.3em]
& = \pi_{xy,\, k}^{\bxi_k} \circ
\left ( \begin{array} {c c c}
\alpha_{n,\,k}     &  \si_{n,\,k} \, t_{n,\,k}\, &  \si_{n,\,k}\, \Bu_{n,\,k}\,  \nonumber  \\
                 & \si_{n,\,k}              &                \nonumber     \\
                 & \si_{n,\,k}\, \Bd_{n,\,k}\,  &  \si_{n,\,k} \cdot \Id_{m \times m}
\end{array} \right )
\left ( \begin{array} {c}
x+ S^n_{k,\, \bxi}  \nonumber \\
y  \nonumber \\
\bxi_n + {\bf R}_{n,\,k}(y)
\end{array} \right )  \nonumber  \\[0.3em]
& = \left( \alpha_{n,\,k} (x + S^n_{k,\,\bxi} ) +  \si_{n,\,k} t_{n,\,k}\, y +  \si_{n,\,k} \Bu_{n,\,k}\cdot (\bxi_n + {\bf R}_{n,\,k}(y)),\ \si_{n,\,k}\, y \right)
\end{align} \\ \msk
where $ S^n_{k,\, \bxi} = S^n_k (x,y, \bxi_n(x,y)) $. 
Then 
\begin{align}  \label{Jacobian of scope map of xi}
\Jac {}_{2d}^{}\Psi^n_{k,\, \bxi} = \ & \det \left( \begin{array} {c c}
\alpha_{n,\,k}\, \Big(1+ \di_x S^n_{k,\, \bxi} + \displaystyle{\sum_{j=1}^m \di_{z_j}} S^n_{k,\, \bxi} \cdot \di_x \xi_n^j \, \Big) +  \si_{n,\,k}\, \Bu_{n,\,k}\cdot \di_x \bxi_n & \bullet \nonumber  \\[0.6em]
0 & \si_{n,\,k} 
\end{array} \right) \nonumber  \\[0.3em]
= \ & \si_{n,\,k} \Big( \alpha_{n,\,k}\, \Big(1+ \di_x S^n_{k,\, \bxi} + \sum_{j=1}^m \di_{z_j} S^n_{k,\, \bxi} \cdot \di_x \xi_n^j \, \Big) +  \si_{n,\,k}\, \Bu_{n,\,k}\cdot \di_x \bxi_n  \,\Big) . 
\end{align}
%
If $ F \in \II(\bar \eps) $ has the invariant surfaces as the graph from $ I^x \times I^y $ to $ I^{\Bz} $ on every level, then $ _{2d}\Psi^{k+1}_{k,\, \bxi} $ is the conjugation between 
$ (_{2d}F_{k,\, \bxi})^2 $ and $ _{2d}F_{k+1,\, \bxi} $ for each $ k \in \N $. Then the two dimensional map $ F_{2d, \, \bxi} $ is called the {\em formally} infinitely renormalizable map with $ C^r $ conjugation. Moreover, the map defined on the equation \eqref{coordinate change map of xi-2d} with $ n =k+1 $, $ _{2d}\Psi^{k+1}_{k,\, \bxi} $ preserves each horizontal line and is the inverse of the horizontal map 
$$ (x,\, y) \mapsto (f_k(x) - \eps_k(x,y, \bxi_k),\; y) \circ (\si_k x,\; \si_k y)$$
by Proposition \ref{2d scaling map of Cr conjugation} as follows. 
\msk
\begin{prop} \label{2d scaling map of Cr conjugation}
Let the coordinate change map $ {}_{2d}^{}\Psi^{k+1}_{k,\,\bxi} $ between $ (_{2d}F_{k,\, \bxi})^2 $ and $ _{2d}F_{k+1,\, \bxi} $ 
be $ {}_{2d}^{}\Psi^{k+1}_{k,\, \bxi} $ which is defined on \eqref{coordinate change map of xi-2d} as the conjugation. Then
$$ {}_{2d}^{}\Psi^{k+1}_{k,\, \bxi} = H_{k,\, \bxi}^{-1} \circ \La_k^{-1} $$

\nin for every $ k \in \N $ where $ H_{k,\; \bxi}(x,y) = (f_k(x) - \eps_k(x,y, \bxi_{k}),\, y) $ and $ \La_k^{-1}(x,y) = (\si_k x,\; \si_k y) $.
\end{prop}
\begin{proof}
Recall the definitions of the horizontal-like diffeomorphism $ H_k $ and its inverse, $ H^{-1}_k $ as follows
\begin{align*}
H_k(w) & = \ (f_k(x) - \eps_k(w),\ y,\ \Bz - \bde_k (y, f_k^{-1}(y), \B0)) \\
 H^{-1}_k(w) & = \ (\phi_k^{-1}(w),\ y ,\ \Bz + \bde_k (y, f_k^{-1}(y), \B0)) .
\end{align*}
Observe that $ H_k \circ H^{-1}_k = \id $ and $ f_k \circ \phi_k^{-1}(w) - \eps_k \circ H_k^{-1}(w) = x $ for all points $ w \in \La_k^{-1}(B) $. Then if we choose the set $ \si_k \cdot \text{graph}(\bxi_{k+1}) \subset \La_k^{-1}(B) $, then the \ssk similar identical equation holds. 
By the definition of the map $ {}_{2d}^{}\Psi^n_{k,\; \bxi} $, the following equation holds
\msk
\begin{equation}
\begin{aligned}
{}_{2d}^{}\Psi^{k+1}_{k,\; \bxi}(x,y) &= \ \pi_{xy}^{\bxi_k} \circ \Psi^{k+1}_k \circ (\pi_{xy}^{\bxi_{k+1}})^{-1}(x,y) \\[0.2em]
&= \ \pi_{xy}^{\bxi_k} \circ \Psi^{k+1}_k (x,y, \bxi_{k+1}) \\[0.2em]
&= \ \pi_{xy}^{\bxi_k} \circ H^{-1}_k \circ \La_k^{-1}(x,y, \bxi_{k+1}) \\[0.2em]
&= \ \pi_{xy}^{\bxi_k} \circ H^{-1}_k (\si_k x,\,\si_k y,\, \si_k \bxi_{k+1}) \\[0.2em]
(*) \quad &= \ \pi_{xy}^{\bxi_k}\, \big( \phi_k^{-1}(\si_k x,\,\si_k y,\, \si_k \bxi_{k+1}),\ \si_k y,\ \bxi_k (\phi_k^{-1} ,\si_k y)\big) \\[0.2em]
&= \ ( \,\phi_k^{-1}(\si_k x,\ \si_k y,\ \si_k \bxi_{k+1}),\, \si_k y \, ) .
\end{aligned} \msk
\end{equation}
In the above equation, $ (*) $ is involved with the fact that $ H^{-1}_k \circ \La_k^{-1}(\,\text{graph}(\bxi_{k+1})) \subset \text{graph}(\bxi_{k}) $. 
Let us calculate $ H_{k,\, \bxi} \circ {}_{2d}^{}\Psi^{k+1}_{k,\; \bxi}(x,y) $. \ssk The second coordinate function of it is just $ \si_k y $. The first coordinate function is as follows
\begin{align*}
& \quad \ \ f_k \circ \phi_k^{-1}(\si_k x,\,\si_k y,\, \si_k \bxi_{k+1}) - \eps_k \big( \phi_k^{-1}(\si_k x,\,\si_k y,\, \si_k \xi_{k+1}),\; \si_k y,\; \bxi_k (\phi_k^{-1} ,\si_k y)\big) \\
(*) \quad &= \ f_k \circ \phi_k^{-1}(\si_k x,\,\si_k y,\, \si_k \bxi_{k+1}) - \eps_k \circ H^{-1}_k (\si_k x,\,\si_k y,\, \si_k \bxi_{k+1}) \\
&= \ \si_k x .
\end{align*}
Hence, $ H_{k,\; \bxi}\, \circ \, {}_{2d}^{}\Psi^{k+1}_{k,\; \bxi}(x,y) = ( \si_k x,\, \si_k y) $. However, \ssk $  H_{k,\, \bxi}\, \circ \,\big( H^{-1}_{k,\, \bxi}(x,y) \circ \La_k^{-1}(x,y) \big) = ( \si_k x,\, \si_k y) $. Therefore, by the uniqueness of the inverse map of $H_{k,\, \bxi}(x,y) $, 
$$ {}_{2d}^{}\Psi^{k+1}_{k,\; \bxi} = H_{k,\; \bxi}^{-1} \circ \La_k^{-1} . $$
\end{proof}
\nin Proposition \ref{2d scaling map of Cr conjugation} enable us to define the renormalization of the two dimensional $ C^r $ H\'enon-like maps as the extension of renormalization of the analytic H\'enon-like maps. 
\begin{defn} \label{definition of renormalizable Cr Henon map}
 Let $ F : (x,y) \mapsto (f(x) -\eps(x,y),\ x) $ be a $ C^r $ H\'enon-like map with $ r \geq 2 $. If $ F $ is renormalizable, then $ RF $, the {\em renormalization} of $ F $ is defined as follows
$$ RF = (\La \circ H) \circ F^2 \circ (H^{-1} \circ \La^{-1}) $$
where $ H(x,y) = (f(x) -\eps(x,y),\ y) $. Define the linear scaling map $ \La(x,y) = (sx, sy) $ if $ s : J \ra I $ is the orientation reversing affine scaling and $ J $ is minimal such that $  J \times I $ is invariant under $ H \circ F^2 \circ H^{-1} $ .
\end{defn}
\msk
\nin If $ F $ is renormalizable $ n $ times, then the above definition can be applied to $ R^kF $ for $ 1 \leq k \leq n $ successively. Two dimensional map $ _{2d}F_{n,\, \bxi} $ with the $ C^r $ function $ \bxi_n $ is the same as $ R^nF_{2d,\, \bxi} $ by Lemma \ref{2d scaling map of Cr conjugation} and the above definition. Thus if the maps $ _{2d}F_{n,\, \bxi} $ are defined on every $ n \in \N $, then the map $ _{2d}F_{n,\, \bxi} $ is realized to be $ R^nF_{2d,\, \bxi} $ and it is called the $ n^{th} $ {\em renormalization} of $ F_{2d,\, \bxi} $.

\subsection{Universality of two dimensional H\'enon-like maps}
Recall that $ Q $ is an $ F $ invariant surface which is tangent to $ E^{pu} $ over the critical Cantor set $ \OO_F $. The critical Cantor set restricted to any invariant surface $ Q $, say $ \OO_{F|_Q} $, is the same as $ \OO_F $. The ergodic measure on $ \OO_{F_{2d, \bxi}} $ is defined as the push forward measure $ \mu $ on $ \OO_{F} $ by the map $ \pi_{xy}^{\bxi} $. In particular, it is defined as follows
\begin{align*}
\mu_{2d,\,\bxi} \big(\pi_{xy}^{\bxi} (\OO_F \cap B^n_{\bf w}) \big) = \mu_{2d,\,\bxi} \big(\pi_{xy}^{\bxi}(\OO_F) \cap \pi_{xy}^{\bxi}( B^n_{\bf w}) \big) = \frac{1}{\;2^n} .
\end{align*}
Since $ \pi_{xy}^{\bxi}(\OO_F) $ is independent of $ \bxi $, $ \mu_{2d,\,\bxi} $ is also independent of $ \bxi $. Then we suppress $ \bxi $ in the notation of the measure $ \mu_{2d} $. Let us define the {\em average Jacobian} of $ F_{2d, \, \bxi} $
$$
b_{2d} = \exp \int_{\OO_{F_{2d}}} \log \Jac  F_{2d, \, \bxi} \; d\mu_{2d} .
$$
Observe that this average Jacobian is independent of the surface map $ \bxi $. 
\begin{lem} \label{Universal Jacobian determinant of Cr Henon map}
Let $ F \in \II(\bar \eps) $ with a sufficiently small perturbation of toy model map satisfying $ \| \di_{\Bz} \bde \| \ll b_{2d} $. \footnote{Every matrix norm is greater than equal to the spectral radius of any given matrix.} 
Suppose that there exist $ R^nF $ invariant \, $ C^r $ surfaces each of which, say $ Q_n $, is tangent to $ E^{pu} $ over the critical Cantor set for $ 2 \leq r < \infty $. Suppose also that $ Q_n =  \text{graph}\,(\bxi_n) $ where $ \bxi_n $ is $ C^r $ map from $ I^x \times I^y $ to $ I^{\Bz} $. Let $ R^nF_{2d,\, \bxi} $ be $ \pi_{xy}^{\bxi_n} \circ F_n|_{\,Q_n} \circ (\pi_{xy}^{\bxi_n})^{-1} $ for each $ n\geq 1 $. Then 
$$ \Jac R^nF_{2d,\, \bxi} = b_{2d}^{2^n} \; a(x) (1+ O(\rho^n))
$$
where $ b_{2d} $ is the average Jacobian of \,$F_{2d, \, \bxi} $ and $ a(x) $ is the universal function of $ x $ for some positive $ \rho < 1 $.
\end{lem}
\begin{proof}
By the distortion Lemma \ref{distortion} and Corollary \ref{average}, we obtain 
\begin{align*}
\Jac F^{2^n}_{2d,\, \bxi} = b_{2d}^{2^n} (1+ O(\rho^n)) .
\end{align*}
Moreover, the chain rule implies that
\begin{align*}
\Jac R^nF_{2d,\, \bxi} = b_{2d}^{2^n} \; \frac{\Jac {}_{2d} \Psi^n_{0,\,\bxi}(w)}{\Jac {}_{2d} \Psi^n_{0,\,\bxi}(R^nF_{2d,\, \bxi}(w))} (1+O(\rho^n))
\end{align*}
where $ w =(x,y, \Bz) $. After letting the tip on every level move to the origin by the appropriate linear map, the equation \eqref{Jacobian of scope map of xi} implies that 
\begin{align}
\Jac {}_{2d}^{}\Psi^n_{0,\, \bxi} = \si_{n,\,0} \big( \alpha_{n,\,0} \cdot \di_x \big(x + S^n_0(x,y,\,\bxi_n)\big) +  \si_{n,\,0} \Bu_{n,\,0}\cdot \di_x \bxi_n  \big) .
\end{align}
Then in order to have the universal expression of the Jacobian, we need the asymptotic of the following maps \msk
$$ \di_x \big(x + S^n_0(x,y,\,\bxi_n)\big) \ \ \text{and} \ \ \dfrac{\si_{n,\,0}}{\alpha_{n,\,0}}\; \di_x \bxi_n . $$ 
By Lemma \ref{asymptotics of non linear part}, 
\begin{align*}
x + S^n_0(x,y,\, \bxi_n) = v_*(x) + a_{F}\: y^2 + \sum_{j=1}^m a_{F,\,j}\: y \cdot \xi_n^j + \sum_{1 \leq i \leq j \leq m} a_{F,\,ij}\,\xi_n^i \cdot \xi_n^j + O(\rho^n) 
\end{align*} 
with $ C^1 $ convergence where $ v_*(x) $ is the universal function. Thus
\begin{align*}
\di_x \big(x + S^n_0(x,y,\,\bxi_n)\big) = v_*'(x) + \sum_{j=1}^m a_{F,\,j}\: y \cdot \di_x\xi_n^j  +  \sum_{1 \leq i \leq j \leq m} a_{F,\,ij}\,( \di_x\xi_n^i \cdot \xi_n^j + \xi_n^i \cdot \di_x \xi_n^j ) + O(\rho^n)  .
\end{align*} \ssk
By Lemma \ref{invariant surfaces on each deep level}, we see $ \| \di_x \bxi_n \| \leq C \bar \eps \, \si^n $.
Then
\begin{align} \label{exponential convergence of Sn with xi}
\di_x \big(x + S^n_0(x,y,\,\bxi_n)\big) = v_*'(x) + O(\rho^n) .
\end{align}
By the equation \eqref{exponential convergence of the invariant surface xi} in Lemma \ref{invariant surfaces on each deep level}, 
\msk
\begin{align*}
\frac{\si_{n,\,0}}{\alpha_{n,\,0}} \; \frac{\di \bxi_n}{\di x} = & \ \di_x \bxi(\bar x, \bar y) \cdot \left[ 1 + \di_x S^n_0(x,y,\,\bxi_n) + \frac{\si_{n,\,0}}{\alpha_{n,\,0}} \; \Bu_{n,\,0} \cdot \frac{\di \bxi_n}{\di x} \right] \\
\text{Thus we obtain that} \quad \frac{\si_{n,\,0}}{\alpha_{n,\,0}} \; \frac{\di \xi_n}{\di x} = & \ \frac{\di_x\bxi(\bar x, \bar y) }{1- \Bu_{n,\,0} \cdot \di_x\bxi(\bar x, \bar y)}\; \big(1 + \di_x S^n_0(x,y,\,\bxi_n) \big)
\end{align*}%
where $ (\bar x, \bar y) \in B(F_{2d,\,\xi}) $ for $ 1 \leq i \leq m $. Moreover, $ (\bar x, \bar y) $ converges to the origin $ (0,0) $ as $ n \ra \infty $ exponentially fast by Corollary \ref{diameter 2}.
\begin{align*} 
\diam(_{2d}\Psi^n_{0,\,\bxi}(B)) \leq \diam(\Psi^n_0(B)) \leq C \si^n
\end{align*}
for some $ C>0 $. In addition to the exponential convergence of $ \di_x \bxi(\bar x, \bar y) $ to $ \di_x \bxi(0,0) $, $ \Bu_{n,\,0} $ converges to $ \Bu_{*,\,0} $ super exponentially fast. Then,
\begin{align}
\frac{\si_{n,\,0}}{\alpha_{n,\,0}} \; \Bu_{n,\,0} \cdot \di_x\bxi_n = \frac{\Bu_{*,\,0} \cdot \di_x\bxi }{1- \Bu_{*,\,0} \cdot \di_x\bxi} \ v_*'(x) + O(\rho^n) .
\end{align}
Let $ (x',y') = w' = R^nF_{2d,\, \bxi}(w) $. Then we obtain
\begin{align} \label{asymptotic of the coordinate change with xi}
\frac{\Jac {}_{2d}^{}\Psi^n_{0,\, \bxi}(w)}{\Jac {}_{2d}^{}\Psi^n_{0,\, \bxi}(w')} = 
\frac{1+ \di_x(S^n_{0,\, \bxi}(w)) + \dfrac{\si_{n,\,0}}{\alpha_{n,\,0}}\; \Bu_{n,\,0}\cdot \di_x \bxi_n(x,y) } {1+ \di_x(S^n_{0,\, \bxi}(w')) + \dfrac{\si_{n,\,0}}{\alpha_{n,\,0}}\; \Bu_{n,\,0}\cdot \di_x \bxi_n(x',y') }
\end{align}
where $ S^n_0(x,y,\,\bxi_n) = S^n_{0,\, \bxi}(x,y) $. The translation does not affect Jacobian determinant and each translation from tip to the origin converges to the map $ w \mapsto \tau_{\infty} $ exponentially fast where $ \tau_{\infty} $ is the tip of two dimensional degenerate map $ F_*(x,y) = (f_*(x),\; x) $ which is the renormalization fixed point. Then by the similar calculation in Theorem \ref{Universality of the Jacobian}, the equation \eqref{asymptotic of the coordinate change with xi} converges to the following universal function exponentially fast. 
\begin{equation}
\begin{aligned}
\lim_{n \ra \infty} \frac{\Jac {}_{2d}^{}\Psi^n_{0,\, \bxi}(w)}{\Jac {}_{2d}^{}\Psi^n_{0,\, \bxi}(w')} &=  \frac{v_*'(x-\pi_x(\tau_{\infty})) + \dfrac{\Bu_{*,\,0}\cdot \di_x \bxi ( \pi_{xy}(\tau_F))}{1- \Bu_{*,\,0}\cdot \di_x \bxi ( \pi_{xy}(\tau_F))}\;v_*'(x-\pi_x(\tau_{\infty})) }{v_*'(f_*(x)- \pi_y(\tau_{\infty})) + \dfrac{ \Bu_{*,\,0}\cdot \di_x \bxi ( \pi_{xy}(\tau_F))}{1- \Bu_{*,\,0}\cdot\di_x \bxi ( \pi_{xy}(\tau_F))}\;v_*'(f_*(x)-\pi_y(\tau_{\infty}))} \\[0.5em]
&= \frac{v_*'(x-\pi_x(\tau_{\infty}))}{v_*'(f_*(x)-\pi_y(\tau_{\infty}))} 
\equiv a(x) .
\end{aligned}
\end{equation}
\end{proof}
\begin{thm}[Universality of\; $ C^r $ H\'enon-like maps with $ C^r $ conjugation for $ 2 \leq r < \infty $] \label{Universality of Cr Henon maps}
Let H\'enon-like map $ F_{2d,\, \bxi} $ be the $ C^r $ map with $ 2 \leq r < \infty $ which is defined on \eqref{Cr Henon map with invariant surface 0}. Suppose that $ F_{2d,\, \bxi} $ is infinitely renormalizable. Then
\begin{align}
R^nF_{2d,\, \bxi}(x,y) = (f_n(x) - b_{2d}^{2^n}\, a(x)\, y\, (1+ O(\rho^n)),\ x)
\end{align}
where 
$ b_{2d} $ is the average Jacobian of $ F_{2d,\, \bxi} $ and $ a(x) $ is the universal function for some $ 0 < \rho < 1 $.
\end{thm}
\begin{proof}
By the smooth conjugation of two dimensional map and $ F_n|_{\,Q_n} $, we see that
\begin{align*}
R^nF_{2d,\, \bxi}(x,\ y) = (f_n(x) - \eps_n(x,y, \bxi_n) ,\ x)
\end{align*}
\ssk Let $ \eps_n(x,y, \bxi_n) $ be $ \eps_{n,\,\bxi_n}(x,y) $. Then the Jacobian of $ R^nF_{2d,\, \bxi} $ is\; $ \di_y \eps_{n,\,\bxi_n}(x,y) $. By Lemma \ref{Universal Jacobian determinant of Cr Henon map}, $ \di_y \eps_{n,\,\bxi_n}(x,y) = b_{2d}^{2^n} \, a(x) (1+ O(\rho^n)) $. Then 
$$ \eps_{n,\,\bxi_n}(x,y) = b_{2d}^{2^n}\, a(x)\, y\, (1+ O(\rho^n)) + U_n(x) . $$ 
The map $ U_n(x) $ which depends only on the variable $ x $ can be incorporated to $ f_n(x) $. 
\end{proof}

\ssk
\begin{thm} \label{universal estimation of scaling maps}
Let $ R^kF \in \II(\bar \eps^{2^k}) $ be the map with invariant surfaces $ Q_k \equiv \text{graph}(\bxi_k) $ tangent to $ E^{pu} $ over the critical Cantor set. Then the coordinate change map $ {}_{2d}^{}\Psi^n_{k,\, \bxi} $ is as follows
\begin{equation} \label{scaling of Cr Henon maps}
{}_{2d}^{}\Psi^n_{k,\, \bxi} = \big(\,\alpha_{n,\,k}\, (\,x +\, _{2d}S^n_k(w)) + \si_{n,\,k}\cdot {}_{2d}t_{n,\,k}\cdot y,\ \si_{n,\,k}\,y\,\big)
\end{equation}
where $ x+{}_{2d}S^n_k(w) $ has the asymptotic
$$ x + {}_{2d}S^n_k(w) = v_*(x) + a_{F,\:k}\, y^2 + O(\rho^{n-k}) $$  
where $ |\;\!a_{F,\:k}| = O(\eps^{2^k}) $.
\end{thm}
\begin{proof}
By Proposition \ref{2d scaling map of Cr conjugation}, the coordinate change map, $ {}_{2d}^{}\Psi^n_{k,\; \xi} $ is the composition of the inverse of horizontal diffeomorphisms with linear scaling maps as follows
$$ H_{k,\; \bxi}^{-1} \circ \La_k^{-1} \circ H_{k+1,\; \bxi}^{-1} \circ \La_{k+1}^{-1} \circ \cdots \circ H_{n,\; \bxi}^{-1} \circ \La_n^{-1} . $$ 
Then after reshuffling non-linear and linear parts separately by the direct calculations and letting the tip move to the origin by the appropriate translations on each levels, the coordinate change map is of the form \eqref{scaling of Cr Henon maps}. 
However, the calculation in Section 7.2 in \cite{CLM} can be used because analyticity is not required for any calculation of recursive formulas. Thus we have the following estimation
$$ x + \,_{2d}S^n_k(x,y) = v_*(x) + a_{F,\;k}\,y^2 + O(\rho^{n-k}) $$  
where $ |\;\! a_{F,\;k}| = O(\eps^{2^k}) $. Alternatively, let us choose the equation \eqref{coordinate change map of xi-2d}
\begin{align*}
{}_{2d}^{}\Psi^n_{k,\, \bxi} = \left( \alpha_{n,\,k} (x + S^n_{k,\,\bxi} ) +  \si_{n,\,k} t_{n,\,k}\, y +  \si_{n,\,k} \Bu_{n,\,k}\cdot (\bxi_n + {\bf R}_{n,\,k}(y)),\ \si_{n,\,k}\, y \right)
\end{align*}
where $ S^n_{k,\,\bxi}(x,y) = S^n_{k}(x,y,\bxi_n(x,y)) $. By Proposition \ref{invariant surfaces on each deep level}, the map $ \bxi_n $ is
\begin{align*}
\bxi_n(x,y) = {\bf c}y + {\boldeta}(y) + O(\rho^n)
\end{align*}
where $ {\boldeta} = (\eta_1, \eta_2, \ldots, \eta_m) $ is quadratic or higher order terms with $ \| {\boldeta} \|_{C^1} \leq C_0 \si^{n-k} $ for some $ C_0 > 0 $. Recall that $ \Bu_{n,\,k} $ converges to $ \Bu_{*,\,k} $ super exponentially fast and $ \| {\bf R}_{n,\,k} \| \leq  C_1 \si^{n-k} $ for some $ C_1 > 0 $. Recall also that $ \alpha_{n,\,k} = \si^{2(n-k)}(1+O(\rho^n)) $ and $ \si_{n,\,k} = (-\si)^{n-k}(1+O(\rho^n)) $. Hence, we define each terms of $ {}_{2d}^{}\Psi^n_{k,\, \bxi} $ appropriately
\begin{align*}
_{2d}S^n_k(x,y) &= S^n_{k,\,\bxi}(x,y) + \frac{\alpha_{n,\,k}}{\si_{n,\,k}} \,\Bu_{n,\,k}\cdot \big[ \bxi_n(x,y) -  {\bf c}y + {\bf R}_{n,\,k}(y) \big] \\
{}_{2d}t_{n,\,k} &= t_{n,\,k} + \Bu_{n,\,k}\cdot {\bf c}
\end{align*}
as desired. 
\end{proof}

\msk

\section{Unbounded geometry of critical Cantor set}
The unbounded geometry of the Cantor set, $ \OO_F $ of small perturbation of the toy model map, $ F \in \II(\bar \eps) $ with $ b_1 \gg \| \:\! \di_{\Bz} \bde \| $ is involved with that of the map on the invariant surfaces of each level, $ F|_{\,Q} $. Since the $ C^r $ conjugation preserves this property, the fact that the Cantor set of $ C^r $ H\'enon-like map $ F_{2d} $ has the unbounded geometry is sufficient to show the same property of $ \OO_F $. 
\ssk \\
Recall that the minimal distance between two boxes $ B_1, B_2 $ is the infimum of the distance between all points of each boxes, $ \dist_{\min}(B_1, B_2) $.
\msk
\begin{defn}
$ F \in \II(\bar \eps) $ has {\em bounded geometry} if
\begin{align*}
 \dist_{\min}(B^{n+1}_{{\bf w}v}, B^{n+1}_{{\bf w}c}) &\asymp \diam(B^{n+1}_{{\bf w}\nu}) \quad \text{for} \ \nu \in \{v, c\} \\[0.2em]
 \diam(B^n_{\bf w}) &\asymp \diam(B^{n+1}_{{\bf w}\nu}) \quad \text{for} \ \nu \in \{v, c\}
\end{align*}
for all $ {\bf w} \in W^n $ and for all $ n \geq 0 $. 
\end{defn}
\nin By the definition of each $ B^{n}_{\bf w} $, if $ F $ does not have bounded geometry, then we call $ \OO_F $ has bounded geometry. Otherwise, we call $ \OO_F $ has {\em unbounded geometry}.
\ssk \\
\nin Let $ F_{2d} $ be an infinitely renormalizable two dimensional H\'enon-like map and $ b_1 $ be the average Jacobian of $ F_{2d} $. Then the unbounded geometry of the Cantor set depends on Universality theorem and the asymptotic of the tilt, $ - t_k \asymp b_1^{2^k} $ but it does not depend on the analyticity of the map. Moreover, unbounded geometry holds if we choose $ n>k $ such that $ b_{2d} \asymp \si^{n-k} $ for every sufficiently large $ k $. This is true on the parameter space of $ b_{2d} $ almost everywhere with respect to Lebesgue measure. 
\begin{thm}[\cite{HLM}]
The given any $ 0 < A_0 < A_1 $, $ 0< \si <1 $ and any $ p \geq 2 $, the set of parameters $ b \in [0,1] $ for which there are infinitely many $ 0<k<n $ satisfying
$$ A_0 < \frac{b^{p^k}}{\si^{n-k}} <A_1 $$
is a dense $ G_{\de} $ set with full Lebesgue measure.
\end{thm}

\begin{thm} \label{Unbounded geometry for model maps}
Let $ m+2 $ dimensional H\'enon-like map, $ F \in \II(\bar \eps) $ be a small perturbation of the toy model map with $ \| \!\: \di_{\Bz} \bde \| \ll b_{1} $ where $ b_{1} $ is the average Jacobian of $ F_{2d,\,\bxi} $. Let $ F_{b_1} $ be the parametrized $ m+2 $ dimensional H\'enon-like map for $ b_1 \in [b_{\circ}, b_{\bullet}] $ where $ \| \di_{\Bz} \bde \| \ll b_{\circ} < b_{\bullet} $. Then there exists $ G_{\de} $ subset $ S $ with full Lebesgue measure of $ [b_{\circ}, b_{\bullet}] $ such that the critical Cantor set $ \OO_{F_{b_1}} $ has unbounded geometry.
\end{thm}
\begin{proof}
The comparison of minimal distances between two adjacent boxes and the diameter of each boxes for every level. Two dimensional invariant surface under $ F_{b_1} $ enable us to apply the proof of two dimensional  H\'enon-like maps. See the proof of Theorem 6.3 in \cite{Nam3} for unbounded geometry of the critical Cantor set.  
\end{proof}
 
\bsk


\begin{thebibliography} {****}




\bibitem[CLM]{CLM} A. de Carvalho, M. Lyubich, M. Martens, Renormalization in the H\'enon family, I : Universality but non-rigidity. J. Stat. Phys. 121 No. 516, (2005), 611-669. 

\bibitem[CP]{CP} S. Crovisier, E. P. Pujals, Essential hyperbolicity and homoclinic bifurcations : A dichotomy phenomenon / mechanism for diffeomorphisms. IMPA preprint (2010)

\bibitem[CT]{CT} P. Coullet, C. Tresser. It\'eration d'endomorphismes et groupe de renormalisation.
J. Phys. Colloque C 539, C5(1978), 25-28.




\bibitem[HLM]{HLM} P. Hazard, M. Lyubich, M. Martens. Renormalizable H\'enon-like map and unbounded geometry. Nonlinearity, 25(2):397-420, 2012.


\bibitem[LM]{LM} M. Lyubich, M. Martens, Renormalization in the H\'enon family, II : The heteroclinic web. Invent. math. , (2011), 186, 115-189.




\bibitem[Nam1]{Nam1} Y. Nam, Invariant space under H\'enon renormalization, preprint (2014), available at \url{http://arxiv.org/abs/1408.4619v3}.

\bibitem[Nam2]{Nam2} Y. Nam, Renormalization of $ C^r $ H\'enon-map, preprint (2014), available at \url{http://arxiv.org/abs/1412.8337}.

\bibitem[Nam3]{Nam3} Y. Nam, Renormalization of three dimensional H\'enon-map I, preprint (2014), available at \url{http://arxiv.org/abs/1408.4289}.



\bibitem[WY]{WY} Q. Wang and L.-S. Young, Towards a theory of rank one attarctors, Annals of Math.,
167 (2008), 349-480.

\bibitem[ZG]{ZG} W. Zhang, S. S. Ge, A global implicit function theorem without initial point and its applications to control of non-affine systems of high dimensions. J. of math. analysis and applications 313 (2006), 251-261

\end{thebibliography}



\end{document}